\documentclass[12pt]{article}

\usepackage{mathtools}
\usepackage{amssymb}
\usepackage{amsthm}
\usepackage{tikz}
\usepackage{fullpage}
\usepackage{url}

\DeclareSymbolFont{bbold}{U}{bbold}{m}{n}
\DeclareSymbolFontAlphabet{\mathbbold}{bbold}

\newcommand{\N}{\mathbb{N}}
\newcommand{\Z}{\mathbb{Z}}

\newcommand{\R}{\mathbb{R}}
\newcommand{\C}{\mathbb{C}}
\newcommand{\K}{\mathbb{K}}
\newcommand{\1}{\mathbbold{1}}
\newcommand{\id}{\operatorname{id}}
\newcommand{\ii}{\mathrm{i}}
\newcommand{\dd}{\mathrm{d}}

\newcommand{\mm}{\mathfrak{m}}

\newcommand{\cowot}{\textnormal{co-}\tau_{\textnormal{w}}}
\newcommand{\cowotl}{\textnormal{co-}\tau_{\textnormal{w}}\textnormal{-}\lim}
\newcommand{\cow}{\textnormal{co-w}}
\newcommand{\cowl}{\textnormal{co-w-}\lim}

\renewcommand{\epsilon}{\varepsilon}

\newcommand{\m}{{\rm m}}

\DeclareMathAccent{\Circ}{\mathalpha}{operators}{"17}

\newcommand{\intereset}{\operatorname{int}}



\DeclareMathOperator{\conv}{conv}

\DeclareMathOperator{\ind}{ind}

\DeclareMathOperator{\diag}{diag}

\DeclareMathOperator{\sym}{sym}
\DeclareMathOperator{\dive}{div}

\DeclareMathOperator{\grad}{grad}

\renewcommand{\ker}{\operatorname{ker}}

\DeclareMathOperator{\ran}{ran}
\DeclareMathOperator{\dom}{dom}

\DeclareMathOperator{\sgn}{sgn}

\renewcommand{\Re}{\operatorname{Re}}
\renewcommand{\hat}{\widehat}

\newcommand{\e}{{\rm e}}

\newcommand{\rdd}{{\rm red}}

\renewcommand\binom[2]{\left[\begin{matrix} #1 \\ #2\end{matrix}\right]}

\let\phi\varphi

\let\leq\leqslant

\let\geq\geqslant

\makeatletter
\def\@row#1,{#1\@ifnextchar;{\@gobble}{&\@row}}
\def\@matrix{%
    \expandafter\@row\my@arg,;%
    \@ifnextchar({\\ \get@in@paren{\@matrix}}{\after@matrix}%
    }
\def\matrixtype#1#2#3{%
    \ifmmode\def\after@matrix{\end{#2}\right#3}%
    \else\def\after@matrix{\end{#2}\right#3$}$\fi
    \left#1\begin{#2}\get@in@paren{\@matrix}%
    }
\def\@column#1,{#1\@ifnextchar;{\@gobble}{\\ \@column}}
\newcommand\vect{}
\def\svect(#1){\left(\begin{smallmatrix}\@column#1,;\end{smallmatrix}\right)}
\def\vect{\get@in@paren{\@vect}}
\def\@vect{\left(\begin{matrix}\expandafter\@column\my@arg,;\end{matrix}\right)}
\def\get@in@paren#1({\def\my@arg{}\def\my@rest{}\def\after@get{#1}\get@arg}
\let\e@a\expandafter
\def\get@arg#1){\e@a\kl@test\my@rest#1(;}
\def\kl@test#1(#2;{\e@a\def\e@a\my@arg\e@a{\my@arg#1}%
                   \ifx:#2:\let\my@exec\after@get
                   \else\let\my@exec\get@arg
                        \e@a\def\e@a\my@arg\e@a{\my@arg(}%
                        \def@rest#2;%
                   \fi\my@exec}
\def\def@rest#1(;{\def\my@rest{#1\kl@zu}}
\def\kl@zu{)}

\makeatletter
\newcommand\MyPairedDelimiter{%
  \@ifstar{\My@Paired@Delimiter{{}}}
          {\My@Paired@Delimiter{}}%
}
\newcommand\My@Paired@Delimiter[4]{%
  \newcommand#2{%
    \@ifstar{\start@PD{#1}{\delimitershortfall=-1sp}{#3}{#4}}
            {\start@PD{#1}{}{#3}{#4}}%
  }%
}
\newcommand\start@PD[5]{%
  #1\mathopen{\mathpalette\put@delim@helper{\put@delim{#2}{#3}{.}{#5}}}%
  #5%
  \mathclose{\mathpalette\put@delim@helper{\put@delim{#2}{.}{#4}{#5}}}%
}
\newcommand\put@delim@helper[2]{%
  \hbox{$\m@th\nulldelimiterspace=0pt #2#1$}%
}
\newcommand\put@delim[5]{%
  \setbox\z@\hbox{$\m@th#5{#4}$}%
  \setbox\tw@\null
  \ht\tw@\ht\z@ \dp\tw@\dp\z@
  #1#5%
  \left#2\box\tw@\right#3%
}

\makeatother
\MyPairedDelimiter*{\abs}{\lvert}{\rvert}
\MyPairedDelimiter*{\norm}{\lVert}{\rVert}
\MyPairedDelimiter{\set}{\{}{\}}

\theoremstyle{plain} 
\newtheorem{theorem}{Theorem}[section]
\newtheorem{corollary}[theorem]{Corollary}
\newtheorem{lemma}[theorem]{Lemma}
\newtheorem{proposition}[theorem]{Proposition}
\theoremstyle{definition}
\newtheorem{example}[theorem]{Example}

\newtheorem{setting}[theorem]{Setting}

\newtheorem{remark}[theorem]{Remark}

\usepackage{enumitem}

\setenumerate[1]{nolistsep} 
\setenumerate[2]{nolistsep} 

\allowdisplaybreaks

\begin{document}

\title{Homogenisation of Laminated Metamaterials \\ and the  \\ Inner Spectrum}

\author{Marcus Waurick}

\date{}

\maketitle

\begin{abstract} We study homogenisation problems for divergence form equations with rapidly sign-changing coefficients. With a focus on problems with piecewise constant, scalar coefficients in a ($d$-dimensional) crosswalk type shape, we will provide a limit procedure in order to understand potentially ill-posed and non-coercive settings. Depending on the integral mean of the coefficient and its inverse, the limits can either satisfy the usual homogenisation formula for stratified media, be entirely degenerate or be a non-local differential operator of 4th order. In order to mark the drastic change of nature, we introduce the `inner spectrum' for conductivities. We show that even though $0$ is contained in the inner spectrum for all strictly positive periods, the limit inner spectrum can be empty. Furthermore, even though the spectrum was confined in a bounded set uniformly for all strictly positive periods and not containing $0$, the limit inner spectrum might have $0$ as an essential spectral point and accumulate at $\infty$ or even be the whole of $\C$. This is in stark contrast to the classical situation, where it is possible to derive upper and lower bounds in terms of the values assumed by the coefficients in the pre-asymptotics. In passing, we also develop a theory for Sturm--Liouville type operators with indefinite weights, reduce the question on solvability of the associated Sturm--Liouville operator to understanding zeros of a certain explicit polynomial and show that generic real perturbations of piecewise constant coefficients lead to continuously invertible Sturm--Liouville expressions. \end{abstract}



\medmuskip=4mu plus 2mu minus 3mu
\thickmuskip=5mu plus 3mu minus 1mu
\belowdisplayshortskip=9pt plus 3pt minus 5pt

\newpage

\tableofcontents
\vfill
\section*{Acknowledgments}

The author is very grateful for the hospitality extended to him by the Erwin-Schr\"odinger Institute. The workshop on ``Spectral Theory of Differential Operators in Quantum Theory'' in November 2022 led to a wonderful scientific atmosphere with valuable discussions and talks that have helped to improve the presentation of the material at hand. The authors thanks Jonathan Stanfill for careful reading. The support of the GrK2583/1 during an early stage of this research is gratefully acknowledged.

\newpage

\section{Introduction}\label{s:int}

The aim of the present article is to understand divergence form problems with piecewise constant possibly sign-changing coefficients and to address the limit of period tending to $0$ for divergence form equations with periodic coefficients of the said type. In more applied contexts sign-changing coefficients are used to describe so-called `metamaterials' and the mathematical challenge lies in the fact that the problem considered is non-coercive by nature. There is an abundance of literature concerning real-world applications of metamaterials, we exemplarily refer to \cite{SPW04,P04,FB05}.

The rigorous limit of period tending to $0$ is classically treated in the area of mathematical homogenisation, see, e.g., \cite{Cioranescu1999,ZKO94}. This theory aims at understanding microscopically heterogenous materials (e.g., periodic with `small' period) by finding a suitable homogeneous (i.e., constant coefficient) replacement sharing a more or less similar behaviour on a macroscopic level. The homogenisation problem (in the context of coercive coefficients) with coefficients depending on one variable only has been studied under the umbrella term of `stratified media' or `laminated materials'. Treating coefficients being constant on slabs, we shall thus analyse homogenisation theory for periodic laminated metamaterials. 

In order to obtain a flavour of the type of problems we encounter here, we recall one of the arguably `easiest' non-trivial examples (see also \cite{H14} for a (possibly) more elementary situation), namely that of a divergence form equation on an open rectangle symmetric about the $y$-axis:
On $\Omega=(-1,1)\times (0,1)$ consider
\[
   -\dive a \grad u = f \in L_2(\Omega)
\]
subject to $u|_{\partial\Omega}=0$, where the conductivity $a$ satisfies
\[
   a(x,y) = \sgn(x) =\begin{cases} 1,& x\geq 0,\\
   -1,&x<0.
   \end{cases}
\]
 This problem has been thoroughly analysed by \cite{BK18} and it was shown that the associated operator-realisation with homogeneous Dirichlet boundary conditions on $L_2$ is self-adjoint, but not invertible; it has $0$ as an eigenvalue of infinite multiplicity; thus implying that the corresponding operator possesses non-trivial essential spectrum and has no compact resolvent. The operator is studied by adding $\ii\eta$, $\eta\in \R$ small, to $-1$ on $x<0$, thus making the perturbed operator coercive (after suitable multiplication with a complex scalar) for all $\eta\neq 0$ and then consider the limit $\eta\to 0$. 

A detailed analysis of this problem with the techniques developed in this manuscript shows that \emph{any} (sufficiently small) number added to both sides of the coefficients leads to a well-posed albeit non-coercive problem, see Example \ref{exam:minusoneplusone1} below.

In a series of papers, the well-posedness of sign-changing coefficients for divergence form problems (also involving Maxwell type equations) has been addressed mainly by the French school, particularly by Bonnet-Ben Dhia and Ciarlet jr., and co-authors, see, \cite{BCZ10} and, e.g., \cite{BCH11,BCC12,BCC14}; see, however, also \cite{DT11}. In \cite{BCZ10}, the authors introduce the concept of $\mathbb{T}$-coercivity, which may be interpreted as coercivity after multiplication with an operator and is equivalent to well-posedness of the problem considered. The main task is in constructing a suitable topological isomorphism $\mathbb{T}$, which may or may not be easy to accomplish in applications, such that the standard bilinear form weighted with $\mathbb{T}$ becomes coercive. For given $f\in L_2(\Omega)$, in \cite{BCZ10}, the authors studied (among other things) problems of finding $u\in H_0^1(\Omega)$ such that 
\[
  \dive\varepsilon^{-1} \grad u = f\in L_2(\Omega)
\] for $\Omega\subseteq \R^{d}$ open, bounded, $d=2$ or $d=3$, with $\varepsilon,\varepsilon^{-1}\in L_\infty(\Omega;\R)$. Here, they additionally assume that $\Omega$ has two open, mutually disjoint (connected) subdomains $\Omega_+,\Omega_-$ with Lipschitz boundaries satisfying $\overline{\Omega} =\overline{\Omega}_+\cup\overline{\Omega}_-$. The coefficient $\varepsilon$ respectively satisfies on $\Omega_+$, $\varepsilon \geq c_+>0$ and, on  $\Omega_-$,  $\varepsilon \leq -c_-<0$ for some $c_\pm>0$. For the case
\[
    \frac{-\sup_{\Omega_-} \varepsilon }{\sup_{\Omega_+}\varepsilon} > \kappa^2 
\] $\mathbb{T}$-coercivity and, hence, well-posedness for the problem in question has been shown in \cite{BCZ10}; here $\kappa$ satisfies
\[
  \| v\|_{1/2,+} \leq \kappa \| v\|_{1/2,-}\quad (v\in H^{1/2}(\Sigma); \Sigma=\overline{\Omega}_+\cap   \overline{\Omega}_-)
\]  and $\|v\|_{1/2,\pm}$ denotes the norm of the trace $v=u|_{\Sigma}$ for $u\in H^1(\Omega_{\pm})$. With the more particular situation of piecewise constant, laminated materials, we are able to provide a more precise well-posedness condition, which in turn is even a characterisation of well-posedness. Moreoever, since we allow for several sign-changes, $\Omega_-$ and $\Omega_+$ are not connected anymore; thus, the results presented here paint a more detailed picture of a particular situation studied in \cite{BCZ10,BCC12} and they do complement the mentioned results in that we provide results for disconnected $\Omega_-$ and $\Omega_+$.

In \cite{Pan19}, K.~Pankrashkin discusses whether or not the sign-indefinite Laplacian given by
\[
   -\dive h \grad u  = f \in L_2(\Omega)
\]
subject to homogeneous Dirichlet boundary conditions for $u$ with $h\colon \Omega \to \{-\mu,1\}$ for some $\mu>0$ admits self-adjoint realisations. In the paper, similar to the approach in \cite{BK18}, the notion of boundary triplets are used. Similar to the discussion we develop in the present article (see Section \ref{sec:sturm}), the author analyses the indefinite Laplacian using operator-valued Sturm--Liouville problems by a separation of variables ansatz. We also mention \cite{CPP19}, where the case $\mu=1$ has been addressed in higher-dimensions again using the theory of boundary triples. 

By means of representation theory, the thesis \cite{Schmitz14} treats the case of sign-indefinite coefficients. We particularly note \cite[Theorem 8.2.2 and Corollary 8.4.11]{Schmitz14} for results in the flavour of the general Theorem \ref{thm:TW-Mana}.

To proceed with the present approach, we need a well-posedness theorem for divergence form problems other than the classical Lax--Milgram lemma. In this manuscript, we recall one that \emph{characterises} well-posedness in terms of invertibility of a projected variant of the conductivity matrix, see \cite{TW14_FE} and Theorem \ref{thm:TW-Mana} below. The core observation underlying  \cite{TW14_FE} is from \cite{P10} and its main consequence can be summarised as follows. Let $\Omega\subseteq \R^d$ be open and bounded. Then for $a\in L( L_2(\Omega)^d)$ the Dirichlet-problem of finding $u\in H_0^1(\Omega)$ such that for given $f\in H^{-1}(\Omega)$ we have
\[
   -\dive a \grad u = f
\]is \emph{equivalent} to 
\[
   0 \in \rho(\iota_0^*a\iota_0),
\]where $\iota_0\colon g_0(\Omega)\hookrightarrow L_2(\Omega)^d$, $g_0(\Omega)\coloneqq \grad[H_0^1(\Omega)]$,  is the canonical embedding. This observation gives rise to the definition of $\sigma_{g_0(\Omega)}(a)$, the \textbf{inner spectrum of $a$} with respect to this Dirichlet problem, given by
\[
    \sigma_{g_0(\Omega)}(a) = \{\lambda\in \C; (\iota_0^*(a-\lambda)\iota_0)^{-1}\in L(g_0(\Omega))\} (=\sigma(\iota_0^*a\iota_0)).
\] In order to determine the (essential) spectrum for Maxwell's equations the relevance of the inner spectrum has been shown in \cite{M1,M2}; some examples complementing the present setting are given there as well. Here, we will explore the inner spectrum for piecewise constant sign-changing coefficients, that is, for metamaterials. In \cite[Section 8.3]{Schmitz14} the inner spectrum for the one-dimensional case has been thoroughly discussed; see in particular \cite[Theorem 8.3.1]{Schmitz14}. We provide the concise arguments for this particular case in Section \ref{sec:1DWP}. For higher-dimensional cases and scalar coefficients $a$ admitting two values with $\Omega_+=[a>0]$ and $\Omega_-=[a<0]$ not necessarily connected, the inner spectrum of $a$ has been computed in \cite[Theorem 8.4.8]{Schmitz14}. The result is presented in terms of the spectrum of a certain combination of the Dirichlet-to-Neumann maps associated to finding harmonic functions on $\Omega_+$ and $\Omega_-$ on $H^{1/2}(\Gamma)$, where $\Gamma$ denotes the interface of $\Omega_+$ and $\Omega_-$.

Using a non-classical variant of the homogenisation process (allowing for the coercivity constant not to be uniformly bounded below or above in the homogenisation parameter), people were able to rigorously derive constitutive relations associated to metamaterials, see, e.g., \cite{BBF09,BS10,FB05,LS16}. Even though the coercivity constants of the problems considered were allowed to degenerate as $\varepsilon\to 0$, the $\varepsilon$-parameter problem is always coercive. Here, we adopt a different point of view in that the problem class we start out with is neither assumed to be coercive nor well-posed. Thus, we study highly oscillatory (periodic, laminated) metamaterials from the outset.

For the coercive case of periodic laminated materials explicit formulas for the homogenised coefficients are due to Tartar and Murat and are well-established, see \cite{Murat1997} or  \cite[Theorem 5.12]{Cioranescu1999}. As a consequence, see also Examples \ref{ex:class1D} and \ref{ex:classDD}, when the inner spectrum is concerned, the following formulas can be shown for periodic, real, coefficients $\alpha_{\#}$ depending on one variable only and assuming the values $\alpha_0,\ldots,\alpha_r\in \R\setminus\{0\}$ on slabs of equal width. Denote by $\mm(\alpha_{\#})$ the integral mean over the period interval, which for simplicity we assume here to be $(0,1)$. Then if $a_n$ is the multiplication operator on $L_2(0,1)^d$ induced by $\alpha_n(x_1,\ldots,x_d)\coloneqq \alpha_{\#} (nx_1)$, $a_n$ $G$-converges (see below for the definition) to 
\[a_\infty = \begin{cases} \frac{1}{\mm(\alpha_{\#}^{-1})}, &d=1,\\
\diag(\frac{1}{\mm(\alpha_{\#}^{-1})},\mm(\alpha_{\#}), \ldots,\mm(\alpha_{\#})  )\in \R^{d\times d}, &d\geq 2,
\end{cases}
\] and we have
\[
  \begin{cases}  \sigma_{g_0(\Omega)}(a_\infty)  
      = \big\{\frac{1}{\mm(\alpha_{\#}^{-1})}\big\},& d=1,\\
  \sigma_{g_0(\Omega)}(a_\infty)      \subseteq {\conv}\{ \frac{1}{\mm(\alpha_{\#}^{-1})}, \mm(\alpha_{\#})\},& d\geq 2.
   \end{cases}
\]
We emphasise that, classically, only properties of the spectra of the (self-adjoint) divergence form problem are considered, see, e.g., \cite{ZKO94}, and to the best of the author's knowledge the inner spectrum has not been addressed so far. In any case, the mentioned formulas reflect the rather weak convergence of the coefficients; see also \cite[Lemma 6.7]{Tartar2009} for a related statement. However, as the values of $\alpha_{\#}$ (for piecewise constant coefficients) belong to the inner spectrum (Theorem \ref{thm:isp-mr}), the inner spectrum of the homogenised coefficients still belongs to a suitable convex hull of the inner spectrum of $a_1$ (note $\mm(\alpha_{\#})\in \conv\{\alpha_1,\ldots,\alpha_r\}\subseteq \sigma_{g_0(\Omega)}(a_1)$ and $\mm(\alpha_{\#}^{-1})\in \conv\{1/\alpha_1,\ldots,1/\alpha_r\}\subseteq 1/[\sigma_{g_0(\Omega)}(a_1)]$).
Given the lack of coercivity and well-posedness of the problems considered here, the corresponding inclusion formulas do not hold in general.

In the series of papers \cite{BCRR20,BRT21,BRT23}, the homogenisation of sign-changing coefficients has been addressed with a combination of the $\mathbb{T}$-coercivity approach and the method of periodic unfolding. Even though the coefficients are changing sign on a complicated domain, the coefficients only assume two different values. Moreover, the setting is arranged in a way that both the pre-asymptotic regime together with the homogenised limit are well-posed allowing for standard arguments in homogenisation theory. 

This contrast the present situation, for example, take $\alpha$ to be the $1$-periodic extension of $\1_{(0,1/4)\cup [3/4,1)}-2\1_{[1/4,3/4)}$ and consider $a_n$ as before. Then we will show and make precise the following statement
\[
a_n  \text{ `$G$-converges to' } -2\Delta_{(0,1)}(-\Delta_{(0,1)^{d-1}}^{-1})
\]
or, written in terms of the corresponding divergence form operators, 
\[
   -\dive a_n \grad \text{  `converges to' } -2\Delta_{(0,1)}^2(-\Delta_{(0,1)^{d-1}}^{-1})-2\Delta_{(0,1)}
\]Thus, in a generalised sense, the limit of the highly oscillatory sign-changing problem is a 4th order pseudo-differential operator, where the differential operator on the left-hand side is subject to homogeneous Dirichlet boundary conditions and considered on $L_2((0,1)^{d})$ and $-\Delta_{\Omega}$ denotes the Dirichlet--Laplacian as the operator realisation in $L_2(\Omega)$ for an open and bounded subset $\Omega\subseteq \R^{d}$. Depending on the dimension and the values of $\mm(\alpha_\#)$ and $\mm(\alpha_\#^{-1})$ we will show that even though there exists a bounded set $\Sigma$ such that $\bigcup_{n\in \N} \sigma_{g_0(\Omega)} (a_n)\subseteq \Sigma$, the corresponding generalised $G$-limit, $a_\infty$, may either satisfy the above inclusions valid in the classical case or one of the following alternatives
\[
   \emptyset = \sigma_{g_0(\Omega)} (a_\infty) \text{ even though } 0\in \bigcap_{n\in \N} \sigma_{g_0(\Omega)} (a_n)\quad(d=1)
\]
or 
\[
   \sigma_{g_0(\Omega)} (a_\infty) =\C \text{ or } \sigma_{g_0(\Omega)} (a_\infty) = \overline{\{ \lambda_k; k\in \N\}}\quad (d\geq 2)
\]
for some (explicit) sequence $(\lambda_k)_k$ in $\R_{>0}$ accumulating at $0$ and $\infty$.

For making the $G$-convergence statement rigorous, we develop a new homogenisation concept, namely that of \emph{holomorphic $G$-convergence}. This convergence generalises classical $G$-convergence of symmetric-matrix-valued coefficients developed and initially studied by Spagnolo, see \cite{Spagnolo1967,Spagnolo1976}. The key ingredients for the method to work are standard facts from operator theory and complex analysis, we refer, e.g., to \cite{Kato95,ABHN11}. We note in passing that tools from analytic perturbation theory are of great importance in the context of (quantitative) homogenisation theory, see in particular the seminal series of papers by Birman and Suslina, see, e.g., \cite{BS04}.

The idea of holomorphic $G$-convergence is as follows: Revisiting the above example, for a dense set of right-hand sides $f\in L_2(\Omega)^{d}$ ($\dom(\Delta_{d})$ will be sufficient), $n\in \N$ and suitable $\lambda\in \C$, we consider $u_{n}(\lambda)\in H_0^1(\Omega)$ the solution of
\[
   -\dive (a_n-\lambda)\grad u_{n}(\lambda) = f.
\]
Then, it can be shown that there exists $\omega\subseteq \C$ open with $0\in \overline{\omega}$ such that $\langle v,u_n(\cdot)\rangle_{L_2(\Omega)}$ converges for all $v\in L_2(\Omega)$ in the compact open topology of holomorphic functions on $\omega$ to some $\langle v,u(\cdot)\rangle_{L_2(\Omega)}$, which extends holomorphically to $0$. Scalar holomorphy and vector-valued holomorphy being the same by Dunford's Theorem, see, e.g., \cite[Proposition A.3]{ABHN11}, this defines an element $u(0)\in H_0^1(\Omega)$, which satisfies
\[
\big(-2\Delta_{(0,1)}^2(-\Delta_{(0,1)^{d-1}}^{-1})-2\Delta_{(0,1)}\big)u(0) = f.
\]
The details of the definition of holomorphic $G$-convergence are provided in Sections \ref{sec:mr} and \ref{sec:holG}.

Homogenisation in an operator theoretic set-up has been addressed in various articles predominantly for the time-dependent situation focussing on so-called evolutionary equations, see, e.g., \cite{NVW21,W16_HPDE,W16_H,W14_FE,W13_HP,W12_HO,W11_P}. Problems in an (almost) time-independent situation were treated in \cite{W18_NHC,W22_NHC} or \cite{CW17_1D,CW17_FH}, where in the latter references also quantitative results were established. In the mentioned references, the core ideas recycled here to understand homogenisation processes are the well-posedness theorems from \cite{TW14_FE} and \cite{PicPhy,A11} (see also \cite[Theorem 6.2.1]{STW_EE21}) in conjunction with certain stability properties of operator-valued holomorphic functions under weak limits, see, e.g., \cite{W16_H} or the Montel type Theorem \ref{thm:MontalOV} below.

In the course of the manuscript, we also have occasion to study Sturm--Liouville-type problems with sign-changing conductivities. In a slightly different situation similar problems have been considered in \cite{BLM10,BMR12} (and the references stated there). Focussing rather on the classical spectrum, in these references however, the authors had as their key assumptions mere differentiability (except for finitely many points) of the said conductivity in order to allow for an analysis invoking an integrating factor. Moreover, the boundary conditions were periodic and the derivative at some value needed to be non-zero. This situation is entirely complementary to the equations we consider in the present research in that we mostly consider piecewise constant coefficients yielding vanishing of the classical derivative where ever it exists.

We briefly comment on the structure of the article at hand. In Section \ref{sec:mr} we provide a concise presentation of the main results of the present manuscript. Section \ref{sec:WPT} recalls the operator-theoretic perspective to divergence form operators due to \cite{TW14_FE}. This section particularly contains a characterisation result instrumental for the main results of the present article as it asserts well-posedness of abstract second order divergence form equations \emph{irrespective} of whether or not the corresponding conductivity is accretive, thus,  justifying the introduction of the inner spectrum. In Section \ref{sec:1DWP} we apply the well-posedness criterion to one-dimensional divergence form problems with $L_\infty$-conductivity and show that the essential condition to solve this type of equations is the inverse of the coefficient being bounded and this inverse having non-vanishing integral mean. In this 1-d set-up, we also compute the inner spectrum. A preparation for well-posedness results for sign-changing coefficients of a corresponding higher-dimensional setting is provided in Section \ref{sec:sturm} containing aspects of Sturm--Liouville theory. Considering piecewise constant conductivities, only, we prove a result characterising well-posedness involving \emph{all} values of the coefficients differing from the criterion presented in \cite{BCZ10}. We start to address the higher-dimensional set-up in Section \ref{sec:DDWP}. Making use of the laminated structure of the coefficients, we write the Laplacian with variable coefficients as sum of a $1$-dimensional second order differential operator plus a $(d-1)$-dimensional Laplacian with conductivity $1$ multiplied with a function not depending on the $d-1$ variables contained in the derivatives of this Laplacian.
 Using the spectral theorem for the $(d-1)$-dimensional operator, we are therefore in the position to apply results from Section \ref{sec:sturm} developing a solution theory for the higher-dimensional set-up in Section \ref{sec:hdp-wp}. Section \ref{sec:pertex} is devoted to describe the inner spectrum in the higher-dimensional setting. 

In Section \ref{sec:classG} we recap the concept of $G$-convergence for (self-adjoint, matrix-valued) coefficients and recall an operator-theoretic characterisation for this convergence from \cite{W16_Gcon}. Equipped with this knowledge, we proceed in Section \ref{sec:holG} to generalise $G$-convergence by introducing the concept of \emph{holomorphic $G$-convergence}. As a first application of this concept, we provide a homogenisation theorem for general sign-changing coefficients in Section \ref{sec:hom1d}. It turns out that the only requirement needed for a homogenisation theorem is convergence in the $L_\infty(\Omega)$-weak-* topology of the inverses of the conductivities. Holomorphic $G$-convergence requires mere periodicity of the conductivity and the consideration of the usual limit of period tending to $0$. The final application of the results and concepts developed here is carried out in Section \ref{sec:homdd}, where we consider the homogenisation problem in the higher-dimensional case. The more lengthy proofs are presented in Section \ref{sec:proofs}. We conclude the article with some summarising thoughts and future directions in Section \ref{sec:con}.

For linear operators the term `invertible' will always be reserved for surjective and injective transformations. Since we only consider closed operators defined on Hilbert spaces here; invertible will thus always mean `continuously invertible'. $\K\in \{\R,\C\}$, scalar products are anti-linear in the first and linear in the second component.

\section{The main theorems}\label{sec:mr}

In this short section, we present the major results of the present article in a rather concise form. For more details we refer to the subsequent sessions. We focus on results concerning the inner spectrum and piecewise constant coefficients. For the precise results and for Sturm--Liouville theory we refer to the individual sections. 

For $\alpha_0,\ldots,\alpha_r \in \C$, we let $\alpha_{\#}$ be the $1$-periodic extension of $\sum_{j=0}^r \alpha_j \1_{( jh,(j+1)h]}$ on the whole of $\R$. For all $d\geq 1$ and ${\Omega}\subseteq \R^{d}$ open, the multiplication operator on $L_2(\Omega)^d$ induced by $(x_1,\ldots,x_d)\mapsto \alpha_{\#}(x_1)$ is denoted by $\alpha(\m_1)$ (if $d=1$, we just write $\alpha(\m)$); correspondingly we will write $\alpha(n\m_1)$ for the respective multiplication operator induced by $(x_1,\ldots,x_d)\mapsto \alpha_{\#}(n x_1)$. For a Lebesgue integrable function $\beta$ defined on $(0,1)$, we  denote by $\mm (\beta)\coloneqq \int_{(0,1)}\beta$ its integral mean.

In the one-dimensional situation, the inner spectrum can be properly understood; for this we recall one of the results of \cite{Schmitz14} for convenience. Later we provide the proof of these in order to have a sample application of Theorem \ref{thm:TW-Mana}. 

\begin{theorem}[{{see Theorem \ref{thm:isp}, see \cite[Theorem 8.3.1]{Schmitz14} or \cite{Schmitz15}}}]\label{thm:isp-mr} Let $\alpha_0,\ldots,\alpha_r \in \C$. Then for $V_1=\{1\}^\bot\subseteq L_2(0,1)$ we have
\[
     \sigma_{V_1}(\alpha(\m)) = \{ \lambda \in \C; (\alpha_\#-\lambda)^{-1} \in L_\infty, \mm((\alpha_\#-\lambda)^{-1})=0\}\cup \{\alpha_j; j\in\{0,\ldots,r\}\}.
\]
\end{theorem}

With significantly more work, we then derive a condition involving checking for zeros of a certain explicit polynomial that characterises well-posedness of the higher-dimensional situation. To keep the presentation concise we just mention an idea of the sort of theorems we derive.

\begin{theorem}[see Corollary \ref{cor:WP1}, Theorem \ref{thm:ndtensor}, and Corollary \ref{cor:1D0}]\label{thm:easy-wp-mr} Let $\alpha_0,\ldots,\alpha_r \in  \R\setminus\{0\}$, $\hat{\Omega}\subseteq \R^{d-1}$ open, bounded; put $\Omega\coloneqq (0,1)\times \hat{\Omega}$ and assume that $\mm(\alpha_{\#}^{-1})\neq 0$. If for all $t\in (0,1]$ 
\[
  p_{\alpha}(t)= \begin{pmatrix}1 &  \alpha_{r}^{-1}  t  \end{pmatrix} \begin{pmatrix} 1 & \alpha_{r-1}^{-1}t \\ \alpha_{r-1}t  & 1\end{pmatrix}\cdots  \begin{pmatrix} 1 &\alpha_{1}^{-1}t  \\  \alpha_{1} t & 1\end{pmatrix} \begin{pmatrix}t  \\  \alpha_0  \end{pmatrix}\neq 0,
\]
then $-\Delta_{\alpha}\coloneqq -\overline{ \dive \alpha(\m_1) \grad}$, where
\[
    -\dive \alpha(\m_1) \grad \colon \{ u\in H_0^1(\Omega); \alpha(\m_1) \grad u \in H(\dive,\Omega)\}\subseteq L_2(\Omega) \to L_2(\Omega),
\]is continuously invertible.
\end{theorem}

A more detailed account of the understanding highlighted in the previous theorem helps us to deduce the following result on the inner spectrum for divergence form problems with sign-changing coefficients. 

We introduce for $\Omega\subseteq \R^d$ open and bounded, the set
\begin{equation}\label{eq:rangr}
    g_0(\Omega)\coloneqq \grad[H_0^1(\Omega)]\text{ and } \iota_0 \colon g_0(\Omega)\hookrightarrow L_2(\Omega)^d.
\end{equation} Then, by the standard Poincar\'e inequality, we have $g_0(\Omega)\subseteq L_2(\Omega)^d$ is a closed subspace and, thus,  $\iota_0^*$ is well-defined; $\iota_0^*q=q_0\in g_0(\Omega)$ where $q=q_0+q_1$ for some $q_1\in g_0(\Omega)^\bot$. 

\begin{theorem}[see Theorem \ref{thm:isp-dd}]\label{thm:isp-dd-mr} Let $\alpha_0,\ldots,\alpha_r \in \R\setminus\{0\}$ and $\hat{\Omega}\subseteq \R^{d-1}$ open and bounded; $\Omega\coloneqq (0,1)\times \hat{\Omega}$, $A>\max_j|\alpha_j|$. Then, there exists a countable nowhere dense set $\Sigma\subseteq [-A,A]$ such that
\[
    \sigma_{g_0(\Omega)}(\alpha(\m_1))\setminus \sigma_{g_0(\Omega),\textnormal{c}}(\alpha(\m_1)) =\Sigma \cup \{\alpha_j; j\in \{0,\ldots,r\}\}.
\]
where $\sigma_{g_0(\Omega),\textnormal{c}}(\alpha(\m_1))$ is the \textnormal{\textbf{inner continuous spectrum}} of $\alpha(\m_1)$; we have $\sigma_{g_0(\Omega),\textnormal{c}}(\alpha(\m_1))\subseteq [-A,A]$.
\end{theorem}

Equipped with the understanding of when sign-changing coefficients lead to well-posed divergence form operators, we can now proceed to the understanding of homogenisation of potentially ill-posed problems and present the respective results. For this we introduce a new notion of convergence related to homogenisation. 

Let $(\tilde{a}_n)_n$ in $L_\infty(\Omega)^{d\times d}$, identified with a sequence in $L(L_2(\Omega)^d)$. Define $a_n\coloneqq \iota_0^*\tilde{a}_n\iota_0$ and assume there exists $\omega\subseteq \bigcap_{n\in\N} \rho(a_n)$ open in $\C$ with $0\in \overline{\omega}$ such that $(a_n-\cdot)^{-1}$ is locally bounded on $\omega$. Let $a\subseteq g_0(\Omega)\times g_0(\Omega)$ be a relation. Then we say that $(a_n)_n$ \textbf{holomorphically $G$-converges (on $\omega$)} to $a$, $a_n\stackrel{\textnormal{hol}-G}{\to} a$,
if the set $\mathcal{G}$ of $\phi\in g_0(\Omega)$ satisfying the following two conditions (a) and (b) is dense, where
\begin{enumerate}
\item[(a)] there exists $f_\phi \colon \omega \to g_0(\Omega)$ holomorphic, such that, for all $\psi\in g_0(\Omega)$, $\langle\psi,(a_n-\cdot)^{-1}\phi\rangle \to \langle \psi, f_\phi(\cdot)\rangle$ in the \textbf{compact open topology} (i.e.,  uniform convergence on compact sets),
\item[(b)] $f_\phi$ admits a holomorphic extension to $0$;
\end{enumerate} and 
\[
   a^{-1} = \{ (\phi,\psi)\in \mathcal{G} \times g_0(\Omega); \psi = f_\phi(0)\}.
\]
$(\tilde{a}_n)_n$ in $L_\infty(\Omega)^{d\times d}$  \textbf{holomorphically $G$-converges (on $\omega$)} to  $\tilde{a} \subseteq L_2(\Omega)^d\times L_2(\Omega)^d$, $\tilde{a}_n\stackrel{\textnormal{hol}-G}{\to} \tilde{a}$, if $(a_n)_n$ {holomorphically $G$-converges on $\omega$} to $a\subseteq g_0(\Omega)\times g_0(\Omega)$ and $a^{-1} = (\iota_0^* \tilde{a}\iota_0)^{-1}$.

A comparison of classical $G$-convergence\footnote{Classically $G$-convergence is formulated for strictly positive definite coefficients, only; we use the same notion for the non-coercive situation, see Section \ref{sec:classG}.} and holomorphic $G$-convergence reads as follows.

\begin{theorem}\label{thm:hGcontw-mr} Let $(\tilde{a}_n)_n$ in $L_\infty(\Omega)^{d\times d}$, $\tilde{a}\in L_\infty(\Omega)^{d\times d}$; $a_n\coloneqq \iota_0^*\tilde{a}_n\iota_0$, $a\coloneqq \iota_0^*\tilde{a}\iota_0$. Assume $0\in \bigcap_{n\in\N} \rho(a_n)\cap \rho(a)$ with $(a_n^{-1})_n$ uniformly bounded.
Consider the following assertions:
\begin{enumerate}
\item[(i)] $(a_n)_n$ holomorphically $G$-converges to $a$;
\item[(ii)] every subsequence $(a_{\pi(n)})_n$ contains a subsequence holomorphically $G$-converging to $a$;
\item[(iii)]  $\tilde{a}_n$ \textbf{\textnormal{$G$-converges}} to $\tilde{a}$, i.e., for all $f\in H^{-1}(\Omega)$ and $u_n\in H_0^1(\Omega)$ such that $-\dive \tilde{a}_n\grad u_n =f $ we have
\[
  u_n \rightharpoonup u\in H_0^1(\Omega)\text{ and }-\dive \tilde{a}\grad u =f
\] 
\end{enumerate}
Then (i)$\Rightarrow$(iii)$\Leftrightarrow$(ii).
\end{theorem}

We are now in the position to present our main homogenisation theorems on (holomorphic) $G$-convergence for sign-changing coefficients. We start off with a result the formulation of which does not require the notion of holomorphic $G$-convergence.

\begin{theorem}[see Theorem \ref{thm:hom1D}]\label{thm:hom1D-mr} Let $(\alpha_n)_n$ be a sequence in $L_\infty(0,1)$ such that $(\alpha_n^{-1})_n$ converges in $\sigma(L_\infty,L_1)$ to some $\beta \in L_\infty(0,1)$ with $\mm(\beta)\neq 0$. If $\alpha_\infty\coloneqq\beta^{-1}\in L_\infty(0,1)$ then, for the associated multiplication operators on $L_2(0,1)$, 
\[
   \alpha_n(\m) \stackrel{G}{\to} \alpha_\infty(\m)
\]
as $n\to\infty$. \end{theorem}

\begin{example}\label{ex:class1D}
Let $\alpha_0,\ldots,\alpha_r\in \R_{>0}$. Then, by Theorem \ref{thm:isp-mr},
\[
\{\alpha_0,\ldots,\alpha_r\}\subseteq \sigma_{g_0(0,1)}(\alpha(n\m)).
\]
By Theorem \ref{thm:hom1D-mr} (and Theorem \ref{thm:perweak}), we obtain
\[
   \alpha(n\m) \stackrel{\textnormal{$G$}}{\to} \Big(\frac{1}{r+1}\big(\frac{1}{\alpha_0}+\cdots+\frac{1}{\alpha_r}\big)\Big)^{-1} =\mm(\alpha_{\#}^{-1})^{-1}\eqqcolon \alpha_\infty.
\]
Thus, for all $n\in \N$,
\[
\emptyset \neq \big(1/ [\sigma_{g_0(0,1)}(\alpha_\infty)]\big) =\{\mm(\alpha_{\#}^{-1})\}\subseteq \conv\big( 1/[\sigma_{g_0(0,1)}(\alpha(n\m))]\big),
\]
where $\conv$ denotes the convex hull.
\end{example}

The degenerate case reads as follows:

\begin{theorem}[see Theorem \ref{thm:hGcon1D}]\label{thm:hGcon1D-mr} Let $\alpha\in L_\infty(\R; \R)$ be $1$-periodic with $\alpha^{-1}\in L_\infty(\R)$ and assume $\mm (\alpha^{-1})=0$; define $\alpha_n\coloneqq \alpha(n\cdot)$. Then \[
   \alpha_n(\m)  \stackrel{\textnormal{hol-$G$}}{\to} \alpha_\infty\coloneqq \{0\}\times L_2(\Omega).
   \]
  In particular, 
  \[
     \sigma_{g_0(\Omega)} (\alpha_\infty) =\emptyset.
  \]
  \end{theorem}
  Hence, note that by Theorem \ref{thm:isp-mr}, $0\in \sigma_{g_0(\Omega)}(\alpha_n(\m))$ for all $n\in \N$; but still the limit has empty spectrum.

The higher-dimensional situation is more involved and provides -- in the degenerate case -- a more complex picture concerning the limit behaviour of the coefficients. Using some known results from homogenisation theory and the following representation we can provide the behaviour of the inner spectrum in the classical situation.

\begin{lemma}[see Lemma \ref{lem:computeproj}]\label{lem:computeproj-mr} Let $\Omega=(0,1)^d$ and $\gamma>0$; denote $\Gamma \coloneqq \diag(\gamma,1,\ldots,1)$. Then 
\[
  \sigma_{g_0(\Omega)}(\Gamma) = \overline{\{ \Big (\frac{\gamma k_1^2 +\sum_{m=2}^d k_m^2 }{  \sum_{m=1}^d k_m^2} \Big); k\in \N_{>0}^d\}}.
\]
\end{lemma}

\begin{example}\label{ex:classDD} Let $\alpha_0,\ldots,\alpha_r\in \R_{>0}$. Then, classically (see \cite{Murat1997} or Theorem \ref{thm:stmed} below), 
\begin{align*}
    \alpha(n\m_1) & \stackrel{\textrm{$G$}}{\to} \diag(1/(\mm(\alpha_{\#}^{-1})),\mm(\alpha_{\#}),\ldots,\mm(\alpha_{\#})) \\ & = \mm(\alpha_{\#})\diag(\frac{1}{\mm(\alpha_{\#}^{-1})\mm(\alpha_{\#})},1,\ldots,1)\eqqcolon \alpha_\infty.
\end{align*}
Hence, by Lemma \ref{lem:computeproj-mr}, we infer
\begin{align*}
   \sigma_{g_0(\Omega)}(\alpha_\infty) & = \mm(\alpha_{\#})\overline{\{ \Big (\frac{\frac{1}{\mm(\alpha_{\#}^{-1})\mm(\alpha_{\#})} k_1^2 +\sum_{m=2}^d k_m^2 }{  \sum_{m=1}^d k_m^2} \Big); k\in \N_{>0}^d\}}\\
 &  \subseteq {\conv} \{\frac{1}{\mm(\alpha_{\#}^{-1})},\mm(\alpha_{\#})\}
\end{align*}
where both $\mm(\alpha_{\#}^{-1})$ and $\mm(\alpha_{\#})$ can be determined from $\sigma_{g_0(\Omega)}( \alpha(n\m_1))\supseteq\{\alpha_0,\ldots,\alpha_r\}$, for the last inclusion see Theorem \ref{thm:isp-dd-mr}.
\end{example}

\begin{theorem}[see Theorem \ref{thm:mainhom} and Corollary \ref{cor:mainhom}]\label{thm:mainhom-mr} Let $\alpha_0,\ldots,\alpha_r\in \R\setminus\{0\}$.
Then the following statements hold\begin{enumerate}
\item[(a)] If $\mm(\alpha_{\#}^{-1})=0$, then there is $\alpha_\infty^{-1}\subseteq L_2(\Omega)^{d}\times \{0\}$ densely (but not everywhere) defined, such that
\[
  \alpha(n\m_1) \stackrel{\textnormal{hol-}G}{\longrightarrow} \alpha_\infty.\]
  We have $\sigma_{g_0(\Omega)}(a_\infty)=\C$
\item[(b)] If $\mm(\alpha_{\#})=0$ and $\mm(\alpha_{\#}^{-1})\neq 0$, then
\[
  \alpha(n\m_1)\stackrel{\textnormal{hol-}G}{\longrightarrow} -\frac{1}{\mm(\alpha_{\#}^{-1})}\Delta_{(0,1)}(-\Delta_{(0,1)^{(d-1)}})^{-1},\]
where $\Delta_{(0,1)}$ denotes the Dirichlet--Laplace operator on $L_2(0,1)$ tensorised with $d-1$ copies of the identity of $L_2(0,1)$; similarly $\Delta_{(0,1)^{d-1}}$ is the identity on $L_2(0,1)$ tensorised with the Dirichlet--Laplacian on $L_2((0,1)^{d-1})$. The limit inner spectrum is 
\[
\frac{1}{\mm(\alpha_{\#}^{-1})} \overline{\{\frac{k_1^2}{k_2^2+\ldots+k_d^2}; k_1,\ldots,k_d\in \N_{>0}\}}.
\]
\item[(c)]If $\mm (\alpha_{\#})\mm(\alpha_{\#}^{-1})\neq 0$, then\[     \alpha(n\m_1)\stackrel{\textnormal{hol-}G}{\longrightarrow} 
     \diag(\mm(\alpha_{\#}^{-1})^{-1}, \mm(\alpha_{\#}),\ldots,\mm(\alpha_{\#}))
\]
If, in addition, $\mm (\alpha_{\#})\mm(\alpha_{\#}^{-1})>0$, the limit inner spectrum is
\[
\overline{\{\frac{\mm(\alpha_{\#}^{-1})  \sum_{m=1}^d k_m^2}{k_1^2 +\mm(\alpha_{\#}^{-1})\mm(\alpha_{\#})\sum_{m=2}^d k_m^2 }; k_1,\ldots,k_m\in \N_{>0}\}}.
\]
\end{enumerate}
\end{theorem}

In the following we embark on the quest to prove the above statements. The starting point for all of the developed theory is the well-posedness statement from \cite{TW14_FE}, which is dealt with next.

\section{Abstract divergence-form operators}\label{sec:WPT}

The theory underlying the solution method for the differential equation discussed here is the identification of divergence form problems as a composition of three invertible mappings, see \cite{TW14_FE}. The main difference to the traditional Lax--Milgram lemma approach is that the main result in \cite{TW14_FE} \emph{characterises} well-posedness instead of providing a criterion for it. We refer to \cite{PTW15_WP_P} for a more detailed discussion of the Lax--Milgram lemma in connection to accretive/coercive operators/relations. What is more, the results in \cite{TW14_FE} provide a more explicit solution formula than the classical approach. This solution formula will prove instrumental for our approach proving our main results.

For a densely defined closed linear operator $C\colon \dom(C)\subseteq H_0\to H_1$ acting from the Hilbert spaces $H_0$ into $H_1$, we set $H^1(C)\coloneqq (\dom(C),\langle\cdot,\cdot\rangle_C)$ with $\langle\cdot,\cdot\rangle_C$ being the graph scalar product. We define $C^\diamond \colon H_1 \to H^{-1}(C)\coloneqq (H^1(C))^*$ via
\[
    C^\diamond \phi \colon H^1(C) \to \K, u\mapsto \langle \phi, Cu\rangle_{H_1}.
\]It is then not difficult to see that $C^\diamond$ extends $C^*$, where we identify $H_1$ with its dual via the unitary Riesz mapping; for details we refer to \cite[Chapter 9]{STW_EE21}.

For a Hilbert space $H$ and a closed subspace $V\subseteq H$, we define
\[
   \iota_V \colon V\hookrightarrow H, x\mapsto x,
\]\textbf{the canonical embedding}. Note that then $\iota_V^* \colon H \rightarrow V$ acts as the orthogonal projection onto $V$ and $\pi_V\coloneqq \iota_V\iota_V^*$ is the actual orthogonal projection. 

Now, throughout the rest of this section, let $H_0$, $H_1$ be Hilbert spaces. We assume that  $C\colon \dom(C)\subseteq H_0\to H_1$ is a closed, densely defined linear operator 
 with closed range 
 \[
 V_1\coloneqq \ran(C)\subseteq H_1;
 \] the closed range theorem yields that 
 \[
 V_0 \coloneqq \ran(C^*)\subseteq H_0
 \] is closed as well, see, e.g., \cite[Corollary 2.5]{TW14_FE}. In this case, we define, $C_{\rdd}\coloneqq \iota_{V_1}^*C\iota_{V_0}$, \textbf{the reduced operator}. For $ a\in L(H_1)$ we also define the operator
\[
    D_{\rdd, a} \colon \dom(D_{\rdd,a})\subseteq V_0 \to \ran(C_{\rdd}^*), u\mapsto C_{\rdd}^\diamond a C_{\rdd} u
\]with domain 
\[
   \dom(D_{\rdd,a}) = \{ u\in \dom(C_{\rdd}) ; a C_{\rdd} u \in \dom(C_{\rdd}^*)\}.
\]
The replacement for the Lax--Milgram lemma reads as follows.

\begin{theorem}[{{\cite[Theorem 3.1]{TW14_FE}}}]\label{thm:TW-Mana} The following conditions are equivalent:
\begin{enumerate}
\item[(i)] for all $f\in H^{-1}(C_{\rdd})$ there exists a unique $u\in H^1(C_{\rdd})$ such that
\[
    \langle a C u, C\phi\rangle = f(\phi)\quad(\phi\in H^1(C_{\rdd}));
\]
\item[(ii)] for all $f\in H^{-1}(C_{\rdd})$ there exists a unique $u\in H^1(C_{\rdd})$ such that
\[
     C_{\rdd}^\diamond \iota_{V_1}^* a\iota_{V_1} C_{\rdd} u = f;
\]
\item[(iii)] the operator $\iota_{V_1}^*a\iota_{V_1} \in L(V_1)$ is continuously invertible.
\end{enumerate}
In either case, if $f\in H^{-1}(C_{\rdd})$ then \[u=(C_{\rdd})^{-1} ( \iota_{V_1}^* a\iota_{V_1})^{-1} (C_{\rdd}^\diamond)^{-1} f\] and, if $f\in \ran(C_{\rdd}^*)$, 
\[u=(C_{\rdd})^{-1} ( \iota_{V_1}^* a\iota_{V_1})^{-1} (C_{\rdd}^*)^{-1} f.\]
$D_{\rdd, a}$ is densely defined, closed and continuously invertible. If $\dom(C)\cap \ker(C)^\bot \hookrightarrow H_0$ compactly, then $\dom(D_{\rdd,a})$ has compact resolvent. Finally, $D_{\rdd,a}^*=D_{\rdd,a^*}$.
\end{theorem}

Motivated by the observation that the abstract divergence form operator is invertible if, and only if, $\iota_{V_1}^*a\iota_{V_1}$ is, we define the \textbf{inner spectrum of $a$ (with respect to $V_1$)} via
\[
   \sigma_{V_1}(a)\coloneqq \{\lambda \in \C; (\iota_{V_1}^*a\iota_{V_1}-\lambda\id_{V_1})^{-1} \in L(V_1)\}=\sigma(\iota_{V_1}^*a\iota_{V_1}).  
\]
The \textbf{inner resolvent set of $a$ (with respect to $V_1$)} is defined by
\[
   \rho_{V_1}(a)\coloneqq \C\setminus  \sigma_{V_1}(a)
\]
\begin{remark}
The sets just introduced depend on $V_1$. In the following  we only use $V_1$ being the range of the gradient restricted to $H_0^1(\Omega)$ functions; that is, $V_1=g_0(\Omega)=\grad[H_0^1(\Omega)]$. 
\end{remark}

Introducing the spectral parts \textbf{inner point spectrum}, $\sigma_{V_1,\textnormal{p}}(a)$, \textbf{inner continuous spectrum}, $\sigma_{V_1,\textnormal{c}}(a)$, and \textbf{inner residual spectrum}, $\sigma_{V_1,\textnormal{r}}(a)$, respectively given by
\begin{align*}
\sigma_{V_1,\textnormal{p}}(a) & = \{ \lambda \in \sigma_{V_1}(a); \iota_{V_1}^* a \iota_{V_1}-\lambda \text{ not one-to-one}\}\\
\sigma_{V_1,\textnormal{c}}(a) & = \{ \lambda \in \sigma_{V_1}(a); \iota_{V_1}^* a \iota_{V_1}-\lambda \text{ one-to-one}, \ran(\iota_{V_1}^* (a-\lambda) \iota_{V_1})\subseteq V_1\text{ dense} \} \\
\sigma_{V_1,\textnormal{r}}(a) & = \{ \lambda \in \sigma_{V_1}(a); \iota_{V_1}^* a \iota_{V_1}-\lambda \text{ one-to-one}, \ran(\iota_{V_1}^* (a-\lambda) \iota_{V_1})\subseteq V_1\text{ not dense} \},
\end{align*}
we can use the operator $D_{\rdd,a}$ in order to understand parts of the inner spectrum. We use similar notation for the usual spectral parts.

\begin{lemma}\label{lem:pointinner} Let $\lambda\in \C$. Then
\[
   0\in  \sigma_{\textnormal{p}}(D_{\rdd,a-\lambda}) \iff \lambda\in  \sigma_{V_1,\textnormal{p}}(a).
\]
\end{lemma}
\begin{proof}
If $0\notin \sigma_{\textnormal{p}}(D_{\rdd,a})$, then
\[
   C_{\rdd}^* \iota_{V_1}^* a\iota_{V_1} C_{\rdd} u =0;
\] has only the trivial solution. Since $\ker(C_{\rdd}^*)=\ker(C_{\rdd}^\diamond)$ (see, e.g., \cite[Proposition 9.2.2]{STW_EE21}), the equation
\[
   C_{\rdd}^\diamond \iota_{V_1}^* a\iota_{V_1} C_{\rdd} u =0;
\] has only the trivial solution either. As both $C_{\rdd}$ and $C_{\rdd}^\diamond$ are topological isomorphisms (see \cite{TW14_FE}), $\iota_{V_1}^* a\iota_{V_1}$ is one-to-one. Reverting the argument, we infer that $\iota_{V_1}^* a\iota_{V_1}$ one-to-one implies $0\neq  \sigma_{\textnormal{p}}(D_{\rdd,a})$. The assertion follows from
\[
 0\in  \sigma_{\textnormal{p}}(D_{\rdd,a-\lambda}) \iff  0 \in \sigma_{V_1,\textnormal{p}}(a-\lambda) \iff \lambda \in \sigma_{V_1,\textnormal{p}}(a).\qedhere
\]
\end{proof}

\begin{theorem}\label{thm:dinner} If $0\in \rho(D_{\rdd,a})$, then $0\in \rho_{V_1}(a)\cup \sigma_{V_1,\textnormal{c}}(a)$. Moreover, 
\[
   \rho_{V_1}(a)\subseteq \{\lambda \in \C; 0 \in \rho(D_{\rdd,\alpha-\lambda})\} \subseteq \rho_{V_1}(a)\cup \sigma_{V_1,\textnormal{c}}(a),
\]or, equivalently,
\[
  \sigma_{V_1,\textnormal{r}}(a)\cup \sigma_{V_1,\textnormal{p}}(a) \subseteq \{ \lambda \in \C; 0\in \sigma(D_{\rdd,\alpha-\lambda})\} \subseteq \sigma_{V_1}(a).
\]
\end{theorem}
\begin{proof}
Assume $0\in \rho(D_{\rdd,a})$. By Lemma \ref{lem:pointinner}, $0\notin \sigma_{V_1,\textnormal{p}}(a)$. Furthermore, if $f\in \ran(C_{\rdd}^*)=V_0$, then $D_{\rdd,a}^{-1} f \eqqcolon u$ is a solution of 
\[
    C_{\rdd}^* \iota_{V_1}^* a\iota_{V_1} C_{\rdd} u =f.
\]
In particular, 
\[
      \iota_{V_1}^* a\iota_{V_1} C_{\rdd} u =(C_{\rdd}^*)^{-1}f,
\]
which shows that $\dom(C^*)\cap V_1 = \ran((C_{\rdd}^*)^{-1})\subseteq \ran(\iota_{V_1}^* a\iota_{V_1})$ is dense in $V_1$ (see, e.g., \cite[Lemma 4.4]{EGW17_D2N}). Thus, $0\in \rho_{V_1}(a)\cup \sigma_{V_1,\textnormal{c}}(a)$, as required. In particular, the second inclusion is thus proved. The first one is a consequence of Theorem \ref{thm:TW-Mana}. The remaining line of inclusions follows from the previous one by computing complements.
\end{proof}

\begin{example} The inclusion $ \{\lambda \in \C; 0 \in \rho(D_{\rdd,\alpha-\lambda})\} \subseteq \rho_{V_1}(a)\cup \sigma_{V_1,\textnormal{c}}(a)$ is optimal in the sense that  $\sigma_{V_1,\textnormal{c}}(a)$ cannot be dropped. Indeed, take a Hilbert space $H$ and $a\in L(H)$ with $0\in \sigma_{\textnormal{c}}(a)$. There exists $T\colon \ran(a)\subseteq H\to H$ closed, onto and one-to-one. In particular, $\ran(T)=H$ is closed and, thus, so is $\ran(T^*)\subseteq H$. Moreover, $T$ one-to-one leads to $T^*$ mapping onto $H$. Consider $\mathcal{T}_a\coloneqq TaT^*$. Then for all $f\in H_0$, $u\coloneqq (T^*)^{-1}a^{-1}T^{-1}f$ is a solution of $\mathcal{T}_au=f$. Moreover, $T^{-1}$ is continuous and, as an inverse of a bounded linear operator, $a^{-1}$ is closed. Thus, $a^{-1}T^{-1}$ is closed, and, as it maps from $H$ to $H$, it is continuous by the closed graph theorem. Finally, as $(T^*)^{-1}=(T^{-1})^*$ is, too, continuous, $\mathcal{T}_a^{-1} \in L(H)$. Hence, $0\in \rho(\mathcal{T}_a)$. However, $0\in \sigma_{\textnormal{c}}(a) =\sigma_{H,\textnormal{c}}(a) =\sigma_{\ran(T^*),\textnormal{c}}(a)$.
\end{example}


\begin{remark} (a) We recall that the standard assumption for the coefficients $a\in L(H_1)$ for weak variational problems of the type discussed in (i) is that $\Re a\geq c$ for some $c>0$ in the sense of positive definiteness. A moment's reflection reveals that this condition yields $\Re \iota_{V_1}^* a\iota_{V_1}\geq c\id_{V_1}$ and, by the boundedness of $\iota_{V_1}^*a\iota_{V_1}$, the invertibility of $\iota_{V_1}^* a\iota_{V_1}$ follows; see, e.g., \cite[Proposition 6.2.3(b)]{STW_EE21}. In fact, the same holds true for $a$ being a maximal monotone relation with $\dom(a)=H_1$. The corresponding argument is, however, somewhat more involved, see \cite[Theorem 4]{MrsRobinson}. This has been exploited in \cite{TW14_FE} to provide a solution theory for nonlinear divergence form problems. 

(b) Note that if $C$ is one-to-one (as it will always be the case in the applications to follow), $D_{\rdd,a}=D_a$, where 
$
   D_a \colon \dom(D_a)\subseteq H_0 \to H_0, u\mapsto C^*a Cu,
$ and 
\begin{align*}
 \dom(D_a) & =\{ u\in \dom(C_{\rdd}) ; a C_{\rdd} u \in \dom(C_{\rdd}^*)\}\\
 &=\{ u\in \dom(C) ; a Cu \in \dom(C^*)\},
\end{align*}
where we used that $\dom(C)=\dom(C_{\rdd})$ as $C$ is one-to-one and that $a Cu \in \dom(C^*)$ if and only if $a Cu \in \dom(C_{\rdd}^*)$. Also, by the injectivity of $C$ (together with the closed range of $C$), it follows that $C^*$ is onto; i.e., $\ran(C^*)=H_0$.
\end{remark}

\section{The one-dimensional problem -- well-posedness}\label{sec:1DWP}

Having discussed the underlying solution method, we may consider the one-dimensional case of sign-changing coefficients. In fact, we shall see that this case is rather elementary to deal with and can be treated quite elegantly with the results from \cite{TW14_FE}. Note that the results in the present section have been obtained also in \cite[Chapter 8]{Schmitz14}. Note that in \cite{Schmitz14} it suffices to have $\alpha^{-1}\in L_2(\Omega)$ for the well-posedness Theorem \ref{thm:1Dwp}. We recall these results here to demonstrate that in one dimensions the results are computable. To this extent, this is a property not shared by the higher-dimensional case.

\begin{setting}\label{set:1Dstart} Let $\Omega\subseteq \R$ be an open and bounded interval. Let $\alpha \in L_\infty(\Omega)$. Then we define $X_\alpha$ acting as
\[
   X_\alpha u = -(\alpha u')'.
\]in the distributional sense and consider this operator as an operator from $H_0^1(\Omega)$ into $H^{-1}(\Omega)$. The operator 
\[D_\alpha \colon \dom(D_\alpha)\subseteq L_2(\Omega) \to L_2(\Omega)
\] is the $L_2(\Omega)$-realisation of $X_\alpha$ with maximal domain $\dom(D_\alpha)$ contained in $H_0^1(\Omega)$.
\end{setting}

Under mild conditions on $\alpha$, the differential equation associated with $X_\alpha$ turns out to be well-posed. For this, we introduce $\mm_{\Omega}(\beta) \coloneqq \frac{1}{\lambda{(\Omega)}}\int_\Omega \beta$ for $\beta \in L_1 (\Omega)$, $\lambda{(\Omega)}$ being the Lebesgue measure of $\Omega$.

\begin{theorem}\label{thm:1Dwp} Assume Setting \ref{set:1Dstart}. In addition, assume $\alpha^{-1} \in L_\infty(\Omega)$ and $\mm_{\Omega} (\alpha^{-1}) \neq 0$. 

Then for all $f\in H^{-1}(\Omega)$ there exists a unique $u\in H_0^1(\Omega)$ such that
\[
     X_\alpha u = f\quad \text{(in the distributional sense)}.
\]The problem is equivalent to finding $\phi\in L_2(\Omega)$ given $\psi\in \{1\}^\bot \subseteq L_2(\Omega)$ such that
\[
    \alpha \phi - \mm_{\Omega}(\alpha\phi) = \psi;
\]
which can be found via
\[
   \phi = \alpha^{-1}\psi - \alpha^{-1}\frac{\mm_\Omega(\alpha^{-1}\psi)}{\mm_{\Omega}(\alpha^{-1})}.
   \]
   Moreover, $D_{\alpha}$, the $L_2(\Omega)$ realisation of $X_\alpha$ is continuously invertible and has compact resolvent and $D_{\alpha}^*=D_{\alpha^*}$.
\end{theorem}

 The rationale follows similar lines to the slightly simpler situation in \cite[Example 1.10]{PTW15_WP_P}.

\begin{proof}[Proof of Theorem \ref{thm:1Dwp}] We apply Theorem \ref{thm:TW-Mana} to $C=\mathring{\partial} \colon H_0^1(\Omega)\subseteq L_2(\Omega) \to L_2(\Omega)$ and $a=\alpha(\m)$, the multiplication operator induced by the multiplication by $\alpha$. Then $C$ is densely defined, closed. $C$ has closed range by the Poincar\'e inequality, see, e.g., \cite{Almansi}. Moreover, 
\[V_1 = \ran(\mathring{\partial}) = \ker(\partial)^\bot =\{1\}^\bot.\] We compute $\alpha_{1} \coloneqq  \iota_{V_1}^*a \iota_{V_1}$ next. For this, note that since $\ker(\partial) = \{1\}$ is spanned by the unit vector $g\coloneqq \frac{1}{\lambda(\Omega)^{1/2}}$, for any $f\in L_2(\Omega)$, we have
\[
    \iota_{V_1}^*f = f- \frac{1}{\lambda(\Omega)^{1/2}}\langle 1,f\rangle \frac{1}{\lambda(\Omega)^{1/2}}1= f- \mm_{\Omega}(f).
\]
Next, let $\phi, \psi \in \{1\}^\bot$ be such that
\[
    \alpha_{1}\phi = \psi.
\]
Then 
\[
    \psi = \alpha_{1}\phi = \alpha \phi - \mm_\Omega(\alpha\phi)
\] and so $\alpha^{-1}\psi = \phi - \alpha^{-1}\mm_\Omega(\alpha\phi)$. Hence,
\[
   \mm_{\Omega}(\alpha^{-1}\psi) = \mm_{\Omega}(\phi) - \mm_{\Omega}(\alpha^{-1})\mm_\Omega(\alpha\phi) =- \mm_{\Omega}(\alpha^{-1})\mm_\Omega(\alpha\phi),
\]since $\phi \in V_1$. Finally,
\[
   \phi = \alpha^{-1}\psi + \alpha^{-1}\mm_\Omega(\alpha\phi) = \alpha^{-1}\psi - \alpha^{-1}\frac{\mm_\Omega(\alpha^{-1}\psi)}{\mm_{\Omega}(\alpha^{-1})}, 
\]which leads to
\[
  \alpha_{1}^{-1} \colon \psi \mapsto \alpha^{-1}\psi - \alpha^{-1}\frac{\mm_\Omega(\alpha^{-1}\psi)}{\mm_{\Omega}(\alpha^{-1})}.
\]Since $\alpha_1^{-1}$ is continuous and everywhere defined, we infer the invertibility of $X_\alpha$ and $D_\alpha$. The remaining assertions follow from Theorem \ref{thm:TW-Mana}, together with $H_0^1(\Omega) \hookrightarrow\hookrightarrow L_2(\Omega)$.
\end{proof}

A closer look into the proof of Theorem \ref{thm:1Dwp} -- particularly for the case of piecewise constant coefficients -- reveals the inner spectrum with respect to $g_0(\Omega)=\{1\}^\bot$. For illustrational purposes we only consider a special case; for $\alpha_0,\ldots,\alpha_r \in \C$ write $\alpha(\m)$ for the multiplication operator induced by 
\[\alpha=\sum_{j=0}^r \alpha_j \1_{( jh,(j+1)h]}\] on $L_2(0,1)$.

\begin{theorem}\label{thm:isp} Let $\alpha_0,\ldots,\alpha_r\in \C$, $h\coloneqq 1/(r+1)$. Then for $V_1=\{1\}^\bot\subseteq L_2(0,1)$ we have
\[
     \sigma_{V_1}(\alpha(\m)) = \{ \lambda \in \C; (\alpha-\lambda)^{-1}\in L_\infty, \mm((\alpha-\lambda)^{-1})=0\}\cup \{\alpha_j; j\in\{0,\ldots,r\}\}.
\]
\end{theorem}
\begin{proof}
We apply Theorem \ref{thm:TW-Mana} and, thus, need to look for exactly those $\lambda\in \C$ such that the operator $a_\lambda \coloneqq \iota_{V_1}^* (\alpha(\m)-\lambda)\iota_{V_1}$ is not invertible. For this let $\phi,\psi\in V_1$ and consider the equation
\[
   a_\lambda \phi = \psi,
\]which reads
\begin{equation}\label{eq:psipho}
   \psi = (\alpha-\lambda)\phi - \mm_{(0,1)}((\alpha-\lambda)\phi).
\end{equation}
Now, if $\alpha-\lambda =0$ on some subinterval, $\psi$ is constant on this subinterval. Thus, $a_\lambda$ is not onto. Hence, 
\[
  \{\alpha_j; j\in\{0,\ldots,r\}\}\subseteq  \sigma_{V_1}(a).
\]
Next, consider the case $(\alpha-\lambda)^{-1}\in L_\infty$ (which complements the other as $\alpha$ is piecewise constant). If $\mm((\alpha-\lambda)^{-1} )\neq 0$, Theorem \ref{thm:1Dwp} applies and $\lambda \notin  \sigma_{V_1}(\alpha(\m))$. On the other hand, if $\mm((\alpha-\lambda)^{-1} ) = 0$, then \eqref{eq:psipho} leads to 
\[
    (\alpha-\lambda)^{-1}\psi = \phi - (\alpha-\lambda)^{-1}\mm((\alpha-\lambda)\phi)
\]showing that
\[
  \mm ((\alpha-\lambda)^{-1}\psi )=0.
\]
Since $\mm((\alpha-\lambda)^{-1} ) = 0$, we infer $(\alpha-\lambda)^{-1}$ is not constant. Thus, $\psi\in \{1, (\alpha-\lambda)^{-1}\}^\bot \neq \{1\}^\bot$. Hence, we again deduce that $a_\lambda$ is not onto. 
\end{proof}

\begin{remark}
The precise detection of the set  
\[\{ \lambda \in \C; (\alpha-\lambda)^{-1}\in L_\infty, \mm((\alpha-\lambda)^{-1})=0\} = \big\{\lambda\in \C; \sum_{j=0}^r \frac{1}{\alpha_j-\lambda} =0\big\}
\] is a classical albeit nontrivial task, see, e.g., \cite{W65}; see also \cite{ELR94} for a related question.
\end{remark}

Next we address the solution theory for problems that are not well-posed and consider small perturbations of these. In the higher-dimensional situation we need to revisit this result. Since, for these problems, we will consider piecewise constant coefficients only, the condition $\alpha^{-1}\in L_\infty(\Omega)$ is automatically satisfied. If, for general $\alpha\in L_\infty(\Omega)$, its inverse is contained in $L_\infty(\Omega)$ with vanishing integral mean, the spectral point $0$ of the inner spectrum is isolated. Again, let $\alpha(\m)$ be the multiplication operator on $L_2(\Omega)$ induced by $\alpha$.

\begin{corollary}\label{cor:1Dwp} Assume Setting \ref{set:1Dstart}. In addition, assume $\alpha^{-1} \in L_\infty(\Omega;\R)$ and $\mm_{\Omega} (\alpha^{-1}) = 0$. 

Then there exists $\varepsilon_0>0$ such that
\[
    B_\C(0,\varepsilon_0)\setminus\{0\} \subseteq \rho_{g_0(\Omega)}(\alpha(\m)).
\]
\end{corollary}

\begin{remark}
Since $\alpha$ is real, $\C\setminus\R \subseteq \rho_{g_0(\Omega)}(\alpha(\m))$ by the classical Lax--Milgram lemma. Indeed, for $\lambda\in \C\setminus\R$, we find $\gamma\in \C$ such that $\Re \gamma(\alpha-\lambda)\geq c$ for some $c>0$.
\end{remark}

\begin{proof}[Proof of Corollary \ref{cor:1Dwp}] By Theorem \ref{thm:1Dwp}, it suffices to show the existence of $\varepsilon_0>0$ such that for all $\lambda \in B(0,\varepsilon_0)\setminus\{0\}$ we have $\mm_{\Omega}((\alpha-\lambda)^{-1})\neq 0$. For this, we find $c>0$ such that $|\alpha^{-1}(x)|\leq c$ for a.e.~$x\in \Omega$. Thus, for all $\lambda \in (-\tfrac{1}{2c},\tfrac{1}{2c})$, we obtain for a.e.~$x\in \Omega$:
\[
   |\alpha(x)-\lambda| = |\tfrac{1}{c}(c\alpha(x) -c\lambda)|\geq \tfrac{1}{c} (|c\alpha(x)|-\tfrac{1}{2})\geq \tfrac{1}{2c},
\]thus proving $(\alpha-\lambda)^{-1}\in L_\infty(\Omega)$. Similar arguments show that $B_\C(0,\varepsilon_0) \ni\lambda \mapsto (\alpha-\lambda)^{-1}\in L_\infty(\Omega)$ is holomorphic (in fact it follows almost literally the same lines as the proof of the holomorphicity of the resolvent map for linear operators). In consequence, $F\colon  B_\C(0,\varepsilon_0)  \ni \lambda \mapsto \mm_\Omega((\alpha+\lambda)^{-1})$ is, too, holomorphic. Since, for $\lambda=i\delta$, $\delta\in \R$, we obtain
\[
  \Im \int_\Omega (\alpha(x)-i\delta)^{-1} \dd x =    \int_\Omega -\delta(\alpha(x)^2 +\delta^2)^{-1} \dd x \neq 0.
\]Hence, $F\neq 0$. Thus, the identity theorem yields that the zeros of $F$ cannot accumulate around $0$, implying the assertion.
\end{proof}

One might hope that the corresponding result for higher-dimensional cases with laminated media is a somewhat straightforward generalisation of the one-dimensional case. This is far from correct as the projection onto the range of the gradient is not one-co-dimensional but infinite-co-dimensional instead. That is why we need a lot more work to treat this case.

\section{Sturm--Liouville problems with indefinite coefficients}\label{sec:sturm}

The analysis of higher-dimensional stratified media requires a more detailed picture of the structure of the operator $X_\alpha$ in Setting \ref{set:1Dstart}. In passing, we shall provide some insight into the spectral theory of Sturm--Liouville operators with indefinite coefficients. The focus here is on invertibility rather than computing the whole (classical) spectrum even though the one yields the other.

In the following, we address the solvability of the resolvent type problem associated with the operator
\begin{equation}\tag{SLP}\label{eq:SLP}
    D_\alpha + \beta = (u\mapsto -(\alpha u')'+\beta u)
\end{equation}
in $L_2(0,1)$ for piecewise constant coefficients $\alpha$ and $\beta$ with $\dom (D_\alpha) \subseteq H_0^1(0,1).$

However, using the result from the previous section a first statement is rather immediate from Fredholm theory for general coefficients. We will use $\mm_\Omega$ for $\Omega=(0,1)$ only and, hence, just write $\mm$ to denote the integral mean on $(0,1)$. 

\begin{proposition}\label{prop:elementaryFredholm} Let $\alpha,\alpha^{-1},\beta\in L_\infty(0,1)$. If $\mm(\alpha^{-1})\neq 0$, then the following conditions are equivalent:
\begin{enumerate}
 \item[(i)] $D_\alpha +\beta$ is continuously invertible in $L_2(0,1)$.
 \item[(ii)] $D_\alpha + \beta$ is one-to-one.
 \end{enumerate}
\end{proposition}
\begin{proof}
(i)$\Rightarrow$(ii): Obvious.

(ii)$\Rightarrow$(i): Theorem \ref{thm:1Dwp} implies that $D_\alpha$ is continuously invertible and has compact resolvent. In consequence, the multiplication with $\beta$ is a $D_\alpha$-compact perturbation of $D_\alpha$. Thus, $\ind (D_\alpha+\beta)=\ind (D_\alpha)=\{0\}$, see for instance \cite[Theorem 3.6]{GW16} (and references therein) for the notion of unbounded Fredholm operators and their index. Hence, $D_\alpha+\beta$ is onto if and only if it is one-to-one; yielding (i). 
\end{proof}

We shall characterise the injectivity for a more particular situation in the results to come. The setting is as follows.

\begin{setting}\label{set:spectral1D} Let $\alpha_0,\ldots, \alpha_r \in \C\setminus\{0\}$ and $\beta_0,\ldots,\beta_r\in \C$ and define for $h\coloneqq 1/(r+1)$
\[
   \alpha =\sum_{j=0}^r \alpha_j \1_{( jh,(j+1)h]}\text{ and }    \beta =\sum_{j=0}^r \beta_j \1_{( jh,(j+1)h]}.
\] Let  $D_\alpha$ be as in Setting \ref{set:1Dstart} for $\Omega=(0,1)$.
\end{setting}

In order to properly study the intricacies of this situation, we introduce the following notion, which is related to the monodromy matrix from Floquet theory. However, the considered boundary conditions are different. In order not to confused the following with the notion from Floquet theory, we coined a different notion.

Let $h>0, \tau\in \K$, $A\in \K^{2\times 2}$. We call $A$  \textbf{transition matrix with respect to $(h,\tau)$}, if for all $v\in \K^2$ and $u$ being the solution of
\[
    \begin{cases} 
       u''=\tau u,& \text{ on }(0,h)\\
       \begin{pmatrix}u(0) \\ u'(0) \end{pmatrix} =v&
    \end{cases}
\]we have
\[
   A v = \begin{pmatrix}u(h) \\ u'(h) \end{pmatrix}.
\]
By the solution theory for linear ODEs, transition matrices are uniquely determined from $(h,\tau)$. Next, we provide the necessary examples needed here. Most prominently we shall require of case (b) in the subsequent sections.
\begin{example} (a) Let $\tau=0$. Then $u(x)=1$ and $v(x)=x$ are linearly independent solutions of $u''=0$. Since $u'=0$ and $v(0)=0$ we get
\[
    A = \begin{pmatrix} 1 & h \\ 0 & 1\end{pmatrix}
\]is the transition matrix with respect to $(h,0)$.

(b) Let $\tau=\mu^2>0$, $\mu>0$. Then $u(x)=\sinh(\mu x)$ and $v(x)=\cosh(\mu x)$ are linearly independent solutions of $u''=\tau u$. Then
\[
  A = \begin{pmatrix} \cosh(\mu h) & \tfrac{1}{\mu}\sinh(\mu h) \\ \mu\sinh(\mu h) & \cosh(\mu h)\end{pmatrix}
\]is the transition matrix w.r.t.~$(h,\mu^2)$,

(c) Let $\tau=-\mu^2<0$, $\mu>0$. Then $u(x)=\sin(\mu x)$ and $v(x)=\cos(\mu x)$ are linearly independent solutions of $u''=\tau u$. Then
\[
  A = \begin{pmatrix} \cos(\mu h) & \frac{1}{\mu}\sin(\mu h) \\ -\mu\sin(\mu h) & \cos(\mu h)\end{pmatrix}
\]is the transition matrix w.r.t.~$(h,-\mu^2)$. 
\end{example}

 The main result of this section and at the same time key for understanding the inner spectra considered in the forthcoming sections is as follows.
\begin{theorem}\label{thm:spectral1D} Assume Setting \ref{set:spectral1D}. If $\mm(\alpha^{-1})\neq 0$, then the following conditions are equivalent:
\begin{enumerate}
 \item[(i)] $D_\alpha +\beta$ is continuously invertible;
 \item[(ii)] $D_\alpha + \beta$ is one-to-one;
 \item[(iii)] $\begin{pmatrix}a_{00}^{(r)} &  \alpha_{r}^{-1} a_{01}^{(r)} \end{pmatrix} \tilde{A}_{\alpha_{r-1}}\cdots  \tilde{A}_{\alpha_{1}}\begin{pmatrix}a_{01}^{(0)} \\  \alpha_0 a_{11}^{(0)} \end{pmatrix}\neq 0$,
\end{enumerate}
where $\tilde{A}_{\alpha_{j}} =\begin{pmatrix} 1 & 0 \\ 0 & \alpha_{j} \end{pmatrix} A^{(j)} \begin{pmatrix} 1 & 0 \\ 0 & \alpha_{j}^{-1} \end{pmatrix}$ and  $A^{(j)} = (a_{kl}^{(j)})_{k,l\in \{0,1\}}$ is the transition matrix with respect to $(\frac{1}{1+r},\tfrac{\beta_j}{\alpha_j})$, $j\in \{1,\ldots,r\}$.
\end{theorem}

The proof of Theorem \ref{thm:spectral1D} is rather technical, though elementary, and postponed to Section \ref{sec:proofs}. In the Setting \ref{set:spectral1D}, we call 
\[
   p_{\alpha,\beta} \coloneqq \begin{pmatrix}a_{00}^{(r)} &  \alpha_{r}^{-1} a_{01}^{(r)} \end{pmatrix} \tilde{A}_{\alpha_{r-1}}\cdots  \tilde{A}_{\alpha_{1}}\begin{pmatrix}a_{01}^{(0)} \\  \alpha_0 a_{11}^{(0)} \end{pmatrix}
\] used in Theorem \ref{thm:spectral1D} the \textbf{characteristic function of $(\alpha,\beta)$}.

We conclude this section with an example, which in a way reproduces the results developed earlier.

\begin{example}\label{ex:1D} We consider the case $\beta=0$ and recall $h=1/(r+1)$. Then for $j\in \{0,\ldots,r\}$,
\[
    A^{(j)} = \begin{pmatrix} 1 & h \\ 0 & 1\end{pmatrix}
\]and, thus, for $j\in \{1,\ldots,r-1\}$
\[
  \tilde{A}_{\alpha_j} = \begin{pmatrix} 1 & 0 \\ 0 & \alpha_j \end{pmatrix} \begin{pmatrix} 1 & h \\ 0 & 1\end{pmatrix}
  \begin{pmatrix} 1 & 0 \\ 0 & \alpha_j^{-1} \end{pmatrix} = \begin{pmatrix} 1 & h \\ 0 & \alpha_j\end{pmatrix}\begin{pmatrix} 1 & 0 \\ 0 & \alpha_j^{-1} \end{pmatrix}=\begin{pmatrix} 1 & h \alpha_j^{-1} \\ 0 & 1\end{pmatrix}.
\]
Hence, 
\begin{align*}
p_{\alpha,0} & = \begin{pmatrix}a_{00}^{(r)} &  \alpha_{r}^{-1} a_{01}^{(r)} \end{pmatrix} \tilde{A}_{\alpha_{r-1}}\cdots  \tilde{A}_{\alpha_{1}}\begin{pmatrix}a_{01}^{(0)} \\  \alpha_0 a_{11}^{(0)} \end{pmatrix} \\
& = \begin{pmatrix}1 &  \alpha_{r}^{-1} h \end{pmatrix} \begin{pmatrix} 1 & h \sum_{j=1}^{r-1} \alpha_j^{-1} \\ 0 & 1\end{pmatrix}\begin{pmatrix}h\alpha_0^{-1} \\  1 \end{pmatrix} \alpha_0  = h\alpha_0 \sum_{j=0}^r \alpha_j^{-1} = \alpha_0 \mm(\alpha^{-1}).
\end{align*}
\end{example}

\section{The higher-dimensional problem -- preliminaries}\label{sec:DDWP}

In this section, we address the application of the results in the previous section to coefficients modelling stratified media. More precisely, we consider the $d$-dimensional open set $\Omega\coloneqq (0,1)\times \hat{\Omega}$ for some $\hat{\Omega}\subseteq \R^{d-1}$ open and bounded and address solvability of 
\[
    -\Delta_{\alpha} u \coloneqq -\overline{\dive \alpha(\m_1) \grad}\, u = f \in L_2(\Omega)
\] and $u \in H_0^1(\Omega)$, where $\alpha \in L_\infty(0,1)$ and $\alpha(\m_1)$ denotes the multiplication operator on $L_2(\Omega)^d$ induced by $(x_1,\ldots,x_d)\mapsto \alpha(x_1)$. The closure is taken as an operator acting in $L_2(\Omega)$.  The tensor product structure of $\Omega$  and the conductivity allows to reduce the problem to a Sturm--Liouville type situation with unbounded potential. Before we go into the details of this observation, we recall $g_0(\Omega)$ and the corresponding $\iota_0\colon g_0(\Omega)\hookrightarrow L_2(\Omega)^{d}$ from \eqref{eq:rangr} and apply our results on the inner spectrum from Section \ref{sec:WPT} to the present situation.
\begin{theorem}\label{thm:innerspectcont}
Let $\alpha \in L_\infty(0,1)$, $\lambda\in \C$. Then 
\begin{enumerate}
\item[(a)] $0\in \sigma(\Delta_{\alpha-\lambda})\Rightarrow \lambda \in \sigma_{g_0(\Omega)}( \alpha (\m_1) )$;
\item[(b)] $0\in \rho(\Delta_{\alpha-\lambda}) \Rightarrow \lambda \in \rho_{g_0(\Omega)}(\alpha(\m_1))\cup \sigma_{g_0(\Omega),\textnormal{c}}(\alpha(\m_1))$.
\end{enumerate}
\end{theorem}
\begin{proof}The statements (a) and (b) are immediate applications of Theorem \ref{thm:dinner}. (For the second implication prove the contraposition and use that both point spectrum as well as residual spectrum are stable under closure processes.) \end{proof}

\begin{remark}\label{rem:conj}
The statement in Theorem \ref{thm:innerspectcont} is not optimal. However, I still conjecture that at least for piecewise constant $\alpha$, the statement in (a) is, in fact, an equivalence. Note that for the case $d=1$, $\alpha(\m_1)$ has no inner continuous spectrum, so that equivalence holds.
\end{remark}

Concerning the inner spectrum, we provide an elementary observation next. If $\alpha=0$ on some nontrivial subinterval of $(0,1)$, then $-\Delta_{\alpha}$ is not one-to-one: Indeed, choose $\hat{\phi}\in C_c^\infty(\hat{\Omega})\setminus\{0\}$ and a non-zero $\check{\phi}\in C_c^\infty(0,1)$ with $\check{\phi}\alpha=0$. Then $\phi(x_1,\ldots,x_n)=\check{\phi}(x_1)\hat{\phi}(x_2,\ldots,x_d)$ satisfies
 \[
    -\Delta_\alpha \phi = 0. 
 \]
 Hence, as a consequence, we deduce the following result:
 \begin{corollary}\label{cor:eleminner} In the present situation assume $\alpha$ takes the values $\{\alpha_0,\ldots,\alpha_r\}\subseteq \C$ on nondegenerate subintervals of $(0,1)$. 
 
 Then 
 \[
   \{\alpha_0,\ldots,\alpha_r\}\subseteq \sigma_{g_0(\Omega)}(\alpha(\m_1)).
 \] 
 \end{corollary}
 We conclude this section by providing the relation of Sturm--Liouville problems to the Laplace equation with stratified media. For this, we call $\Lambda\subseteq \C$ \textbf{discrete}, if 
  \[
       \inf \{ |\lambda-\mu|;  \lambda,\mu \in \Lambda, \lambda\neq \mu\}>0.
  \]
  We call a sequence $(\lambda_k)_k$ in $\C$ \textbf{discrete}, if $\{\lambda_k; k\in \N\}\subseteq \C$ is discrete.

\begin{remark} Note that for a discrete sequence $(\lambda_k)_k$, the set of finite accumulation values is 
\[
\{ \lambda; \lambda_k=\lambda\text{ for infinitely many }k\in \N\}. 
\]As a consequence, for any discrete sequence $(\lambda_k)_k$ in $\R_{>0}$, we obtain 
\[
\inf_{k\in \N} |\lambda_k| = \min_{k\in \N} |\lambda_k| >0.
\]
\end{remark}

The announced result on the unitary equivalence of $\Delta_\alpha$ is stated next. Note that, as a consequence, since we provide results relative to \emph{any} discrete sequence in the following, we cover all geometries of $\hat{\Omega}$ as long as it is open and bounded.
\begin{theorem}\label{thm:ndtensor} Assume $\alpha,\alpha^{-1}\in L_\infty(0,1)$ with $\mm(\alpha^{-1})\neq 0$. Recall $D_{\rdd,\alpha}=D_{\alpha}$ from Setting \ref{set:1Dstart}. Then $-\Delta_{\alpha}$ is self-adjoint and unitarily equivalent to the diagonal operator
\[
(D_{\rdd,\alpha} + \alpha \lambda_k)_{k\in \N}
\]
on $L_2(0,1;\ell_2(\N))$ for a discrete sequence $(\lambda_k)_k$ in $\R_{>0}$.
\end{theorem}

We provide the proof in Section \ref{sec:proofs}. 

\section{The higher dimensional problem -- well-posedness}\label{sec:hdp-wp}

We turn back to piecewise constant coefficients. For this we recall that for $\alpha_0,\ldots, \alpha_r \in \R\setminus\{0\}$ and $h\coloneqq 1/(r+1)$ we defined
\[
   \alpha =\sum_{j=0}^r \alpha_j \1_{( jh,(j+1)h]}.
\] Let  $D_\alpha$ be as in Setting \ref{set:1Dstart} for $\Omega=(0,1)$. We address criteria that ensure that for a given  discrete sequence $(\lambda_k)_k $ in $\R_{>0}$, the diagonal operator 
\[
    S_{\alpha} \coloneqq ((D_\alpha + \alpha\lambda_k)^{-1})_{k\in \N}
\] acting in $L_2(0,1;\ell_2(\N))$ is both well-defined and bounded. For this particular situation, we rephrase Theorem \ref{thm:spectral1D}. Before doing so, we introduce a function with which the respective formulation is less cumbersome. Define
   \begin{multline*}
     \tilde{q}(m_0,\ldots, m_r)\coloneqq \\
     \begin{pmatrix}1 &  \alpha_{r}^{-1} m_r^{-1} t_r \end{pmatrix} \begin{pmatrix} 1 & \alpha_{r-1}^{-1}m_{r-1}^{-1} t_{r-1}\\ \alpha_{r-1}m_{r-1} t_{r-1} & 1\end{pmatrix}\cdots  \begin{pmatrix} 1 &\alpha_{1}^{-1}m_{1}^{-1} t_{1} \\  \alpha_{1}m_{1} t_{1} & 1\end{pmatrix}\begin{pmatrix}m_0^{-1} t_0 \\  \alpha_0  \end{pmatrix}
   \end{multline*}
   with $t_j = \tanh (m_j h)$, $j\in \{0,\ldots,r\}$, and $m_0,\ldots,m_r>0$. We present a reformulation of Theorem \ref{thm:spectral1D}. Recall $\mm(\beta) = \int_0^1\beta$ for $\beta\in L_1(0,1)$.

\begin{theorem}\label{thm:sp1Dff} Let $\alpha_0,\ldots, \alpha_r \in \R\setminus\{0\}$ with $\mm (\alpha^{-1})\neq 0$. Let $(\lambda_k)_k$ in $\R_{>0}$ be discrete and choose $C>1/(\min_{k\in \N}\lambda_k \min_j |\alpha_j|)$.

Then the following conditions are equivalent:
\begin{enumerate}
 \item[(i)] for all $k\in \N$, the operator $D_\alpha+\alpha\lambda_k$ is continuously invertible with $\|(D_\alpha+\alpha\lambda_k)^{-1}\|\leq C$;
 \item[(ii)] $(-1/C,1/C)\subseteq \rho(D_\alpha+\alpha\lambda_k)$
 \item[(iii)] for all $k\in \N$ and $\delta \in [-1,1]$, $ \tilde{q}(m_0,\ldots, m_r)\neq 0$ where $t_j=\tanh(m_j h)$ and $m_j =\sqrt{ \lambda_k+\tfrac{\delta}{\alpha_j C}}$ 
\end{enumerate}
\end{theorem} 

The proof is carried out in Section \ref{sec:proofs}. An elementary consequence of Theorem \ref{thm:sp1Dff} is the following.

\begin{corollary}\label{cor:1D0} Assume the situation of Theorem \ref{thm:sp1Dff}. The following conditions are equivalent:
\begin{enumerate}
 \item[(i)] $S_\alpha$ is well-defined and bounded;
 \item[(ii)] there exists $\delta_0>0$ such that for all $d\in [-\delta_0,\delta_0]$ and $k\in \N$, $ \tilde{q}(m_0,\ldots, m_r)\neq 0$, where $t_j =\tanh(m_jh)$ and $m_j = \sqrt{\lambda_k + \sgn(\alpha_j)d}$.
 \end{enumerate}
\end{corollary}
\begin{proof}
(ii)$\Rightarrow$(i): This is a consequence of Theorem \ref{thm:sp1Dff} (iii)$\Rightarrow$(i).

(i)$\Rightarrow$(ii): Since $S_\alpha$ is a diagonal operator, it is well-defined and bounded, if, and only if, $0\in \rho(D_\alpha+\alpha\lambda_k)$ for all $k\in \N$ and there exists $C\geq 0$ such that $\sup_{k\in \N}\|(D_\alpha+\alpha\lambda_k)^{-1}\|\leq C$; hence, for $\delta\in [-1,1]$ the non-zero condition holds with
\[
m_j^2 = \lambda_k + \tfrac{\delta}{\alpha_j C} =\lambda_k +\sgn(\alpha_j) \tfrac{\delta}{|\alpha_j| C} \quad (k\in \N,j\in \{0,\ldots,r\}),
\]which eventually leads to the assertion.
\end{proof}

Condition (ii) in the latter characterisation is rather difficult to apply for particular settings. Therefore, we require a more handy criterion ensuring (ii) (and thus (i)) to hold.  

Let $\alpha_0,\ldots,\alpha_r \in \R\setminus\{0\}$ and $(\lambda_k)_k$ a discrete sequence in $\R_{>0}$. We say that $\alpha_0,\ldots,\alpha_r$ and $(\lambda_k)_k$ satisfy the \textbf{$\tilde{q}$-criterion}, if 
 there exists $\delta_0>0$ such that for all $d\in [-\delta_0,\delta_0]$ and $k\in \N$
\begin{equation}\label{eq:qtildecrit}
\tilde{q}(m_0,\ldots, m_r) \neq 0,
\end{equation}
where $t_j =\tanh(m_jh)$ and $m_j = \sqrt{\lambda_k + \sgn(\alpha_j)d}$.

The satisfaction of the $\tilde{q}$-criterion  works as follows. There is a local part asking for a non-zero condition to be satisfied for $d=0$ and an asymptotic condition that warrants the existence of $\delta_0>0$ \emph{uniformly} in $k\in \N$.

Recall the characteristic function $p_{\alpha,\beta}$ from Theorem \ref{thm:spectral1D} and define for $\mu_0>0$
\[
   M_{\mu_0} \coloneqq [\mu_0,\infty)\times \mu_0[-\tfrac{1}{2},\tfrac{1}{2}]^{r+1}\text{ and }S(\mu,x)\coloneqq ((\mu+x_j))_{j\in \{0,\ldots,r\} }\quad ((\mu,x)\in M_{\mu_0})
   \]
   Moreover, we set
   \begin{align*}
      q & \colon  M_{\mu_0}  \to \R, (\mu,x)\mapsto \tilde{q}(S(\mu,x)) \\      
      w & \colon  M_{\mu_0} \to \R, (\mu,x)\mapsto \tilde{r}(S(\mu,x)),        \end{align*}
where
   \begin{align*}
         \tilde{w}(m_0,\ldots, m_r) \coloneqq 
         \begin{pmatrix}1 &  \alpha_{r}^{-1} m_r^{-1}  \end{pmatrix} \begin{pmatrix} 1 & \alpha_{r-1}^{-1}m_{r-1}^{-1} \\ \alpha_{r-1}m_{r-1}  & 1\end{pmatrix}\cdots  \begin{pmatrix} 1 &\alpha_{1}^{-1}m_{1}^{-1}  \\  \alpha_{1}m_{1} & 1\end{pmatrix}\begin{pmatrix}m_0^{-1}  \\  \alpha_0  \end{pmatrix} 
   \end{align*}
for $m_0,\ldots,m_r>0$.

\begin{proposition}\label{prop:pam} (a) Let $m_0,\ldots,m_r>0$, let $\alpha$ be composed of $\alpha_j$ and $\beta$ be as in Setting \ref{set:spectral1D} with $\beta_j = m_j^2\alpha_j$. Then
\[
   p_{\alpha,\beta}= \big(\prod_{j=0}^r \cosh(m_j h)\big) \tilde{q}(m_0,\ldots, m_r).
\]
(b) Let $\mu_0>0$. Then
\[
     \lim_{\mu\to\infty} \mu w(\mu,x)=  \mu_0 w(\mu_0,0)=\prod_{j=0}^{r-1}(1+\alpha_{j+1}^{-1}\alpha_j) \eqqcolon \chi(\alpha)
\]uniformly for $x\in \mu_0 [-\tfrac12,\tfrac12]^{r+1}$.

\noindent
(c) Let $\mu_0>0$.  Then, as $\mu\to\infty$, $q(\mu,\cdot), w(\mu,\cdot) \in \mathcal{O}(\tfrac{1}{\mu})$. Moreover,
\[
   \lim_{\mu\to\infty} \| \mu (q(\mu,\cdot)-w(\mu,\cdot))\|_{\infty} = 0.
\]
(d) For all $\mu_0>0$, the function
\[
  \hat{q} \colon M_{\mu_0} \ni (\mu,x)\mapsto \mu q(\mu,x)
\]is uniformly continuous.
\end{proposition}

The proof can be found in Section \ref{sec:proofs}. For the characterisation of the $\tilde{q}$-criterion, we introduce another function:
\begin{align*}
p_{\alpha}(t)\coloneqq \begin{pmatrix}1 &  \alpha_{r}^{-1}  t  \end{pmatrix} \begin{pmatrix} 1 & \alpha_{r-1}^{-1}t \\ \alpha_{r-1}t  & 1\end{pmatrix}\cdots  \begin{pmatrix} 1 &\alpha_{1}^{-1}t  \\  \alpha_{1} t & 1\end{pmatrix} \begin{pmatrix}t  \\  \alpha_0  \end{pmatrix}.
\end{align*}

A first consequence of Proposition \ref{prop:pam} is a condition guaranteeing the $\tilde{q}$-criterion for $\alpha_0,\ldots,\alpha_r$ and $(\lambda_k)_{k\geq k_0}$ for some $k_0\geq 0$ making use of $\chi(\alpha) \coloneqq    \prod_{j=0}^{r-1}(1+\alpha_{j+1}^{-1}\alpha_j),$ introduced in Proposition \ref{prop:pam} (b).

\begin{theorem}\label{thm:asy} Let $\alpha_0,\ldots, \alpha_r \in \R\setminus\{0\}$, $(\mu_k)_k$ a discrete sequence in $\R_{>0}$ with $\mu_k\to \infty$ as $k\to\infty$. If $\chi(\alpha)\neq 0$, then there exists $k_0\in\N$ and $\delta>0$ such that  $q(\mu_k,x)\neq 0$ for all $k\geq k_0$ and $x\in [-\delta,\delta]^{r+1}$.
\end{theorem}
\begin{proof}
By Proposition \ref{prop:pam} (b) and (c) and assumption, we get $c\coloneqq \lim_{\mu\to\infty} \mu \tilde{q}(\mu,\ldots,\mu)\neq 0$. Hence, we find $k_0\in \N$ such that, for all $k\geq k_0$, $|\mu_{k} \tilde{q}(\mu_{k},\ldots, \mu_{k})|\geq c/2$. By uniform continuity of $\hat q$ from Proposition \ref{prop:pam} (d), we find $\delta>0$ such that for all $x\in [-\delta,\delta]^{r+1}$ and $k\in \N$, we have
 \[
     \mu_k| q(\mu_{k},x)|\geq \tfrac{c}{4};
 \]yielding the desired assertion.
\end{proof}

We are now in the position to summarise our findings concerning the $\tilde{q}$-criterion, which we prove in Section \ref{sec:proofs}.

\begin{theorem}\label{thm:charcrit1}Let $\alpha_0, \ldots,\alpha_r \in \R\setminus\{0\}$, $(\lambda_k)_k$ a discrete sequence in $\R_{>0}$, $\mu_k\coloneqq \sqrt{\lambda_k}$,  with $\lambda_k\to \infty$ as $k\to\infty$. Consider the following conditions:
\begin{enumerate}
\item[(i)]  $\alpha_0, \ldots,\alpha_r$ and $(\lambda_k)_k$ satisfy the $\tilde{q}$-criterion;
\item[(ii)] for all $k\in \N$, $q(\mu_k,0)\neq 0$;
\item[(iii)] for all $k\in \N$, with $t_k\coloneqq \tanh(\mu_k h)$,  $p_{\alpha}(t_k)\neq 0$;
\item[(iv)] for all $k\in \N$,  $q(\mu_k,0)\neq 0$ and $\chi(\alpha) \neq 0$;
\item[(v)] for all $k\in \N$, with $t_k\coloneqq \tanh(\mu_k h)$,  $p_{\alpha}(t_k)\neq 0$ and $\chi(\alpha) \neq 0$.
\end{enumerate}
Then (v)$\Leftrightarrow$(iv)$\Rightarrow$(i)$\Rightarrow$(ii)$\Leftrightarrow$(iii).
\end{theorem}

Finally, we come to a rather handy sufficient condition for the satisfaction of the $\tilde{q}$-criterion.

\begin{corollary}\label{cor:WP1} Let $\alpha_0,\ldots,\alpha_r\in \R\setminus\{0\}$, $(\lambda_k)_k$ in $\R_{>0}$ discrete; $h=1/({r+1})$, $0<t_0\leq \tanh(\min_{k\in \N}\sqrt{\lambda_k}h)$. If for all $t\in [t_0,1]$, $p_{\alpha}(t)\neq 0,$
then  $\alpha_0, \ldots,\alpha_r$ and $(\lambda_k)_k$ satisfy the $\tilde{q}$-criterion.
\end{corollary}
\begin{proof}
The statement follows from Theorem \ref{thm:charcrit1} since $p_\alpha(1)=\chi(\alpha)$.
\end{proof}

\section{The inner spectrum in $d$ dimensions}\label{sec:pertex}

In this section we want to apply the well-posedness criterion from the previous section and want to obtain some information on the inner spectrum of $\alpha(\m_1)$ given as multiplication operator induced by the function 
\[
(x_1,\ldots,x_d)\mapsto \alpha(x_1)=\sum_{j=0}^{r} \alpha_j \1_{[j/{r+1}, (j+1)/(r+1))}(x_1), 
\] where $\alpha_0,\ldots,\alpha_r\in \R\setminus\{0\}$ for the case $\Omega=(0,1)\times \hat{\Omega}$, $\hat{\Omega}\subseteq \R^{d-1}$ open and bounded. For one, the inner spectrum is real and, two, bounded above and below by $\max_{j}|\alpha_j|$ (both statements follow from classical Lax--Milgram). We use the results obtained in the previous sections to deduce that it contains the actual values of $\alpha$ (see Corollary \ref{cor:eleminner}) and that the discrete part of the spectrum is countable. The latter follows from the representation of the inner spectrum as zeros of certain rational functions. Recall for $t>0$
\[
  p_{\alpha}(t)= \begin{pmatrix}1 &  \alpha_{r}^{-1}  t  \end{pmatrix} \begin{pmatrix} 1 & \alpha_{r-1}^{-1}t \\ \alpha_{r-1}t  & 1\end{pmatrix}\cdots  \begin{pmatrix} 1 &\alpha_{1}^{-1}t  \\  \alpha_{1} t & 1\end{pmatrix} \begin{pmatrix}t  \\  \alpha_0  \end{pmatrix}.
\]
Corollary \ref{cor:WP1} confirms that if $p_{\alpha}$ has no zeros for $t\in (0,1]$, then $\alpha_0,\ldots,\alpha_r$ and any discrete $(\lambda_k)_k$ in $\R_{>0}$ satisfy the $\tilde{q}$-criterion and thus the problem of finding $u\in H_0^1(\Omega)$ such that for given $f\in L_2(\Omega)$ and $\alpha$ piecewise constant $\alpha_j$ on slabs of width $1/(r+1)$ we have
\[
   -\Delta_{\alpha}=-\overline{\dive \alpha(\m_1) \grad}\, u = f
\]admits a unique solution $u$ continuously depending on $f$.

For the next statement, we recall that in a metric space $X$, a set $U\subseteq X$ is called \textbf{nowhere dense}, if $\intereset(\overline{U})=\emptyset$. As a consequence, if $X= [-A,A]\subseteq \R$ for some $A>0$ and $\overline{U}$ is countable, $U$ is nowhere dense.

\begin{theorem}\label{thm:perturb} Let $\alpha_0,\ldots,\alpha_r \in \R\setminus\{0\}$, $A>\max_{j\in \{0,\ldots,r\}}|\alpha_j|$. Then for all $(\lambda_k)_k$ in $\R_{>0}$ discrete,  the set 
\[
   W \coloneqq \{ s\in [-A,A]; \text{$\alpha_0-s, \ldots, \alpha_r-s$ and $(\lambda_k)_k$ satisfy the $\tilde{q}$-criterion}\} \setminus \{\alpha_0,\ldots,\alpha_r\}
   \] as a subset of $[-A,A]$ is the complement of a nowhere dense, countable set. 
\end{theorem}
\begin{proof}
We define for all $t>0$ (using the obvious meaning for $\alpha-s$)
\[
 [-A,A]\setminus\{\alpha_0,\ldots,\alpha_r\} \ni s\mapsto f_t (s)\coloneqq p_{\alpha-s}(t).
\]
Let $(\lambda_k)_k$ be a discrete sequence in $\R_{>0}$, $\mu_k\coloneqq \sqrt{\lambda_k}$, $h=1/(1+r)$, and choose any $0<t_0\leq \tanh(\min_k\mu_k h)$. With $T\coloneqq \{\tanh(\mu_kh); k\in \N\}\cup\{1\}$ consider the set
\[
  Z\coloneqq   \{ (t,s) \in [t_0,1]\times\big( [-A,A]\setminus\{\alpha_0,\ldots,\alpha_r\}\big); \exists t \in T: f_t(s) = 0 \} =\bigcup_{t\in T} [f_t=0].
\] Then $Z$ is countable as countable union of zeros of rational functions.
Next, the set 
\[
  U \coloneqq \{ s\in [-A,A]; \exists t\in [t_0,1]\colon (t,s)\in Z\}
\]is, too, countable. Indeed, if $U$ was uncountable, due to the countability of $Z$, we find $t^*\in [t_0,1]$ such that $U^*=\{s\in [-A,A]; f_{t^*}(s)=0\}$ is uncountable, which contradicts the fact that $f_{t^*}$, being a non-constant rational function $(f_{t^*}(A)\neq 0)$, has only finitely many zeros.

Define $\tilde{U}\coloneqq U\cup \{\alpha_0,\ldots,\alpha_r\}$. We show that $\tilde{U}$ is closed. For this, let $(s_n)_n$ be a sequence in $\tilde{U}$ converging to some $s^* \in [-A,A]$. If $s^*\in \{\alpha_0,\ldots,\alpha_r\}$, we are done. Otherwise, without restriction, we may assume the existence of $\varepsilon>0$ such that $(s_n)_n$ is a sequence in $B\coloneqq [-A,A]\setminus \big(\cup_{j=0}^r (\alpha_j-\varepsilon,\alpha_j+\varepsilon)\big)$; as a consequence $s^*\in B$. Then, for all $n\in \N$ we find $t(s_n) \in T$ such that $f_{t(s_n)}(s_n)=0$. Since $(\mu_k)_k$ is discrete possibly accumulating at $\infty$, we infer that $T\subseteq [t_0,1]$ is closed; hence compact. Consequently, we find a subsequence $(t(s_{\pi(n)}))_n$ converging to some $t^*\in T$. By continuity of $B\times T\ni (s,t)\mapsto f_t(s)$, $f_{t^*}(s^*)=0$ and, therefore, $s^*\in U$. Thus, $\tilde{U}$ is countable and closed.

Finally, note that, for all $k\in \N$ and $s\in V\coloneqq[-A,A]\setminus \tilde{U}$,  $f_{t_k}(s)\neq 0$ and $f_1(s)\neq 0$ where $t_k=\tanh(\mu_kh)$; hence $V\subseteq W$, by Theorem \ref{thm:charcrit1}. As $\tilde{U}$ is both countable and closed, the assertion follows.
\end{proof}

We are now in the position to describe the inner spectrum for laminated metamaterials of the form discussed in this manuscript; we recall $g_0(\Omega)=\grad[H_0^1(\Omega)]$.

\begin{theorem}\label{thm:isp-dd} Let $\alpha_0,\ldots,\alpha_r \in \R\setminus\{0\}$ and $\hat{\Omega}\subseteq \R^{d-1}$ open and bounded; $\Omega\coloneqq (0,1)\times \hat{\Omega}$, $A>\max_j|\alpha_j|$. Then 
\[
\sigma_{g_0(\Omega),\textnormal{c}}(\alpha(\m_1))\subseteq [-A,A]\] and there exists a countable and nowhere dense set $\{\alpha_0,\ldots,\alpha_r\}\subseteq \Sigma\subseteq [-A,A]$ such that
\[
    \sigma_{g_0(\Omega)}(\alpha(\m_1))\setminus \sigma_{g_0(\Omega),\textnormal{c}}(\alpha(\m_1)) =\Sigma.
\]
\end{theorem}
\begin{proof}
Since $\alpha_j\in \R$, for any $s\in \C\setminus\R$, $\gamma(\alpha(\m_1)-s)$ is coercive for a suitable $\gamma$, thus showing that $\sigma_{g_0(\Omega)}(\alpha(\m_1))\subseteq \R$. Similarly, it follows that $\sigma_{g_0(\Omega)}(\alpha(\m_1))\subseteq [-A,A]$. By Theorem \ref{thm:ndtensor} there exists a discrete sequence $(\lambda_k)_k$ in $\R_{>0}$ such that $-\Delta_\alpha$ is unitarily equivalent to $(D_\alpha+\alpha\lambda_k)_k$. From Theorem \ref{thm:perturb} it follows that  \[
   W = \{ s\in [-A,A]; \text{$\alpha_0-s, \ldots, \alpha_r-s$ and $(\lambda_k)_k$ satisfy the $\tilde{q}$-criterion}\}\setminus\{\alpha_0,\ldots,\alpha_r\}
\] is the complement of a closed and countable set. Note that $\mm((\alpha-s)^{-1})=0$ only for finitely many $s_1,\ldots,s_k\in \R$; recall $\rho_{g_0(\Omega)}(\alpha(\m_1))=\C\setminus \sigma_{g_0(\Omega)}(\alpha(\m_1))$.
By Corollary \ref{cor:1D0} and Theorem \ref{thm:innerspectcont} (b), we obtain that
\[
   W \setminus \{s_1,\ldots,s_k\} \subseteq \big( \rho_{g_0(\Omega)}(\alpha(\m_1)) \cap [-A,A]\big) \cup \sigma_{g_0(\Omega),\textnormal{c}}(\alpha(\m_1))
\]
Since, by Theorem \ref{thm:perturb},  the left-hand side is the complement of a countable and nowhere dense set, so is the right-hand side. Therefore, $\sigma_{g_0(\Omega)}(\alpha(\m_1))\setminus \sigma_{g_0(\Omega),\textnormal{c}}(\alpha(\m_1))\subseteq [-A,A]$ is countable and nowhere dense. The inclusion of  $\{\alpha_0,\ldots,\alpha_r\}$ in the spectrum is shown in Corollary \ref{cor:eleminner}.
\end{proof}

In the next couple of examples we explicitly compute $p_{\alpha+s}(t)$ for the sign-changing coefficients $\alpha_j=(-1)^j$.

\begin{example}\label{exam:minusoneplusone}
Let $r\in \N$ odd, $(\alpha_j)_j = ((-1)^j)_j$. Then we obtain for $s\neq \pm 1$
\begin{align*}
   p_{\alpha+s}(t)&= \begin{pmatrix}1 &  (-1+s)^{-1}  t  \end{pmatrix} \begin{pmatrix} 1 & (1+s)^{-1}t \\ (1+s)t  & 1\end{pmatrix}\cdots  \begin{pmatrix} 1 &(-1+s)^{-1}t  \\  (-1+s) t & 1\end{pmatrix} \begin{pmatrix}t  \\  1+s  \end{pmatrix}\\
   & =  \begin{pmatrix}1 &  \frac{1}{-1+s}  t  \end{pmatrix} \begin{pmatrix} 1+\frac{-1+s}{1+s}t^2 & (\frac{1}{1+s}+\frac{1}{-1+s})t \\ 2st  & 1+\frac{1+s}{-1+s}t^2\end{pmatrix}^{(r-1)/2}
    \begin{pmatrix}t  \\  1+s  \end{pmatrix}
\end{align*}
In particular, if $s=0$, we obtain
\begin{align*}
  p_{\alpha}(t) &=\begin{pmatrix}1 &  - t  \end{pmatrix} \begin{pmatrix} 1-t^2 & 0 \\ 0  & 1-t^2\end{pmatrix}^{(r-1)/2}
    \begin{pmatrix}t  \\  1  \end{pmatrix} \\
    & = (1-t^2)^{(r-1)/2} (t-t)\\
    &=0,
\end{align*}which confirms the irregular situation one is confronted with in \cite{BK18}.
\end{example}
Continuing the latter example, we shall now address whether the set $W$ in Theorem \ref{thm:perturb} is the entire interval (except $s=0$). It turns out that the answer depends on the number of slabs in the crosswalk.
\begin{example}\label{exam:minusoneplusone1}
Let $r\in \N$ odd, $(\alpha_j)_j = ((-1)^j)_j$. Then we obtain for $s\in (-1,1)$
\begin{align*}
   p_{\alpha+s}(t)   & =  \begin{pmatrix}1 &  \frac{1}{-1+s}  t  \end{pmatrix} \begin{pmatrix} 1+\frac{-1+s}{1+s}t^2 & (\frac{1}{1+s}+\frac{1}{-1+s})t \\ 2st  & 1+\frac{1+s}{-1+s}t^2\end{pmatrix}^{(r-1)/2}
    \begin{pmatrix}t  \\  1+s  \end{pmatrix}\\
   & =  \begin{pmatrix}1 &  \frac{1}{-1+s}  t  \end{pmatrix} \Big(1_2 +\begin{pmatrix} \frac{-1+s}{1+s}t^2 & (\frac{1}{1+s}+\frac{1}{-1+s})t \\ 2st  & \frac{1+s}{-1+s}t^2\end{pmatrix}\Big)^{(r-1)/2}
    \begin{pmatrix}t  \\  1+s  \end{pmatrix} \\
    & = \sum_{\ell=0}^{(r-1)/2} \binom{(r-1)/2}{\ell}\begin{pmatrix}1 &  \frac{1}{-1+s}  t  \end{pmatrix} \begin{pmatrix} \frac{-1+s}{1+s}t^2 & (\frac{1}{1+s}+\frac{1}{-1+s})t \\ 2st  & \frac{1+s}{-1+s}t^2\end{pmatrix}^{\ell}
    \begin{pmatrix}t  \\  1+s  \end{pmatrix},
\end{align*}
where $\binom{(r-1)/2}{\ell}$ denotes the binomial coefficent.
Using
\begin{align*}
 \begin{pmatrix} \frac{-1+s}{1+s}t^2 & (\frac{1}{1+s}+\frac{1}{-1+s})t \\ 2st  & \frac{1+s}{-1+s}t^2\end{pmatrix} & = 
 \frac{t}{s^2-1}  \begin{pmatrix}(-1+s)^2t & (-1+s+1+s) \\ 2s(s^2-1)  & (1+s)^2t\end{pmatrix} \\
 & =\frac{t}{s^2-1}  \begin{pmatrix}(-1+s)^2t & 2s \\ 2s(s^2-1)  & (1+s)^2t\end{pmatrix}, 
\end{align*}
we obtain
\begin{align*}
p_{\alpha+s}(t) &= \begin{pmatrix}1 &  \frac{1}{-1+s}  t  \end{pmatrix} \begin{pmatrix} 1+\frac{-1+s}{1+s}t^2 & (\frac{1}{1+s}+\frac{1}{-1+s})t \\ 2st  & 1+\frac{1+s}{-1+s}t^2\end{pmatrix}^{(r-1)/2}
    \begin{pmatrix}t  \\  1+s  \end{pmatrix}
\\ &= \sum_{\ell=0}^{(r-1)/2} \binom{(r-1)/2}{\ell} \big(\frac{t}{s^2-1}\big)^\ell \begin{pmatrix}1 &  \frac{1}{-1+s}  t  \end{pmatrix}  \begin{pmatrix}(-1+s)^2t & 2s \\ 2s(s^2-1)  & (1+s)^2t\end{pmatrix}^{\ell}
    \begin{pmatrix}t  \\  1+s  \end{pmatrix}.
\end{align*}
In the simplest case, $r=1$, we obtain
\[
p_{\alpha+s}(t) = \begin{pmatrix}1 &  \frac{1}{-1+s}  t  \end{pmatrix}  \begin{pmatrix}t  \\  1+s  \end{pmatrix} = t(1 +\frac{1+s}{-1+s}),
\]which, for all $s\neq 0$, leads to the applicability of Corollary \ref{cor:WP1} and, particularly, to the satisfaction of the $\tilde{q}$-criterion for all discrete $(\lambda_k)_k$ in $\R_{>0}$.

For $r=3$, however, we see that 
\begin{align*}
p_{\alpha+s}(t) & =   \begin{pmatrix}1 &  \frac{1}{-1+s}  t  \end{pmatrix} \begin{pmatrix}t  \\  1+s  \end{pmatrix} + \frac{t}{s^2-1}\begin{pmatrix}1 &  \frac{1}{-1+s}  t  \end{pmatrix}  \begin{pmatrix}(-1+s)^2t & 2s \\ 2s(s^2-1)  & (1+s)^2t\end{pmatrix}
    \begin{pmatrix}t  \\  1+s  \end{pmatrix} \\
    & = t(1 +\frac{1+s}{-1+s}) + 
    \frac{t}{s^2-1}\begin{pmatrix}(-1+s)^2t+2st(1+s) &  2s+2st\frac{(1+s)^2}{-1+s}  t  \end{pmatrix} 
    \begin{pmatrix}t  \\  1+s  \end{pmatrix} \\
    & = \frac{4 s t (s^2-1 + (s^2+1)t^2)}{(-1 + s)^2 (1 + s) }.
\end{align*}
For $s\neq 0$ and $|s|<1$, we obtain 
\[
   z(s)\coloneqq \sqrt{\frac{1-s^2}{1+s^2}}
\] as the only strictly positive zero of $p_{\alpha+s}$ in $(0,1)$. In particular, since for $s\to 1$, we get $z(s)\to 0$ and, for $s\to 0$, $z(s)\to 1$, we deduce that for all $(\mu_k)_k$ in $\R_{>0}$ discrete with $\mu_k\to \infty$ and $\varepsilon_0>0$, we find $s_0\in [-\varepsilon_0,\varepsilon_0]\setminus\{0\}$ so that $z(s_0) =\tanh(\mu_k h)$ for some $k\in \N$. In particular, using Theorem \ref{thm:charcrit1}, we find a sequence $(s_\ell)_\ell$ in $[-\varepsilon_0,\varepsilon_0]\setminus\{0\}$ such that the $\tilde{q}$-criterion is violated for $\alpha+s_\ell$ and $(\lambda_k)_k$. Hence, $\Sigma$ from Theorem \ref{thm:isp-dd} accumulates at $0$ showing it to be in the essential spectrum. In particular, in contrast to $r=1$, $0$ is \emph{not} an isolated inner spectral point.
\end{example}
The previous example shows that the statement in Theorem \ref{thm:perturb} is optimal, that is, in the case $\mm(\alpha^{-1})=0$, one cannot hope for that $0$ is isolated in $\sigma_{g_0(\Omega)}(\alpha(\m_1))$. On the other hand, unfortunately, the information on the inner continuous spectrum is rather coarse, which emphasises the importance of the conjecture mentioned in Remark \ref{rem:conj}.
 
In order to study highly oscillatory sign-changing coefficients, we provide some elements from operator theory next. 

\section{Classical $G$-convergence}\label{sec:classG}

We recall the concept of $G$-convergence, see \cite{Spagnolo1967,Spagnolo1976}. We let $\Omega\subseteq \R^d$ open and bounded, and recall $\iota_0 \colon g_0(\Omega)\hookrightarrow L_2(\Omega)^d$ from \eqref{eq:rangr}. We introduce the set
\[
  M(\Omega)\coloneqq \{ \tilde{a}\in L_\infty(\Omega)^{d\times d}; 0 \in \rho(\iota_0^*\tilde{a}\iota_0)\},
\]where we identify $\tilde{a}\in M(\Omega)$ with the corresponding multiplication operator on $L_2(\Omega)^d$. The restriction of $M(\Omega)$ to symmetric $L_\infty$-matrix valued functions will be denoted by $M_{\sym}(\Omega)$.

By Theorem \ref{thm:TW-Mana}, if $\tilde{a}\in M(\Omega)$, for any given $f\in H^{-1}(\Omega)$, the divergence form problem 
\[
   -\dive \tilde{a} \grad u= f
\]
is uniquely solvable for $u\in H_0^1(\Omega)$. Now, if $(\tilde{a}_n)_n$ in $M(\Omega)$ and $\tilde{a}\in M(\Omega)$, we say that $(\tilde{a}_n)_n$ \textbf{$G$-converges} to $\tilde{a}$, if for all $f\in H^{-1}(\Omega)$ and $u_n\in H_0^1(\Omega)$ with
\[
   -\dive \tilde{a}_n \grad u_n = f
\]implies $u_n\to u$ weakly in $H_0^1(\Omega)$ and
\[
    -\dive \tilde{a} \grad u = f.
\]
\begin{remark}
The concept of $G$-convergence is originally introduced for sequences contained in the set
\[
   M_{\sym}(\alpha,\beta,\Omega)\coloneqq \{\tilde{a}\in L_\infty(\Omega)^{d\times d}; \tilde{a}(x)=\tilde{a}(x)^*\geq \alpha, \|\tilde{a}\|\leq \beta\},
\]where $0<\alpha<\beta$. It can be shown that the $G$-limit is unique in $M_{\sym}(\alpha,\beta,\Omega)$. We address uniqueness in $M_{\sym}(\Omega)$ after the following characterisation of $G$-convergence.
\end{remark}
The basic operator theoretic description for $G$-convergence observed in \cite{W16_Gcon} and implicitly contained in \cite{Tartar2009} reads as follows, see also \cite{EGW17_D2N}. Again, it uses the representation of the solution for divergence form problems provided in Theorem \ref{thm:TW-Mana}. We will provide a slightly more general result in Theorem \ref{thm:Gconcocoer} below as it merely requires a closer inspection of the proof of the main result in \cite{W16_Gcon}.

\begin{theorem}[\cite{W16_Gcon}]\label{thm:homoGcon} Let $0<\alpha<\beta$ and $(\tilde{a}_n)_n$ in $M_{\sym}(\alpha,\beta,\Omega)$, $\tilde{a}\in M_{\sym}(\alpha,\beta,\Omega)$. Then the following conditions are equivalent:
\begin{enumerate}
\item[(i)] $(\tilde{a}_n)_n$ $G$-converges to $\tilde{a}$;
\item[(ii)] $(\iota_0^*\tilde{a}_n \iota_0)^{-1}\to (\iota_0^*\tilde{a} \iota_0)^{-1}$ in the weak operator topology.
\end{enumerate}
\end{theorem}

A (straightforward) consequence of the proof of Theorem \ref{thm:homoGcon} is the following statement, which does not require coercivity conditions (nor the symmetry).
\begin{theorem}\label{thm:Gconcocoer} Let $(\tilde{a}_n)_n$ in $M(\Omega)$, $\tilde{a}\in M(\Omega)$. Then the following conditions are equivalent:
\begin{enumerate}
\item[(i)] $(\tilde{a}_n)_n$ $G$-converges to $\tilde{a}$;
\item[(ii)] $(\iota_0^*\tilde{a}_n \iota_0)^{-1}\to (\iota_0^*\tilde{a} \iota_0)^{-1}$ in the weak operator topology.
\end{enumerate}
\end{theorem}
\begin{proof}
Let $f\in H^{-1}(\Omega)$. For $n\in \N$, by Theorem \ref{thm:TW-Mana}, 
\[
  u_n = (\iota_0^*\grad)^{-1} (\iota_0^*\tilde{a}_n\iota_0)^{-1} (\dive\iota_0)^{-1} f.
\]Since $(\iota_0^*\grad)^{-1}$ is a topological isomorphism, $u_n\to u$ weakly in $H_0^1(\Omega)$ with \[u=(\iota_0^*\grad)^{-1} (\iota_0^*\tilde{a}_n\iota_0)^{-1} (\dive\iota_0)^{-1} f,\]  if, and only if, $(\iota_0^*\tilde{a}_n\iota_0)^{-1} (\dive\iota_0)^{-1} f \to (\iota_0^*\tilde{a}_n\iota_0)^{-1} (\dive\iota_0)^{-1} f$ weakly. So, if the former convergence holds for all $f$, the latter holds for all $q\in (\dive\iota_0)^{-1}[H^{-1}(\Omega)]=\ran(\iota_0)=g_0(\Omega)$. The assertion follows.
\end{proof}
\begin{remark}
In the situation of Theorem \ref{thm:Gconcocoer} consider the $L_2(\Omega)$ realisations $D_{\tilde{a}_n}$ of the divergence form operators $-\dive \tilde{a}_n\grad$ initially considered as operator from $H_0^1(\Omega)$ to $H^{-1}(\Omega)$, see also Theorem \ref{thm:TW-Mana}. Then, if $\tilde{a}_n$ $G$-converges to $\tilde{a}$, then $D_{\tilde{a}_n}^{-1}\to D_{\tilde{a}}^{-1}$ in operator norm.

Indeed, this follows using a similar technique as in the proof of \cite[Theorem 4.2]{EGW17_D2N}. For this we observe that both $\iota_0^*\grad$ and $(\iota_0^*\grad)^*=-\dive\iota_0$ have compact resolvent. In consequence, for all $(f_n)_n$ in $L_2(\Omega)$ weakly convergent to some $f$, we obtain, using the representation in Theorem \ref{thm:TW-Mana}, 
\[
u_n\coloneqq (\iota_0^*\grad)^{-1}(\iota_0^*\tilde{a}_n\iota_0)^{-1}(-\dive\iota_0)^{-1}f_n \to u\coloneqq (\iota_0^*\grad)^{-1}(\iota_0^*\tilde{a}\iota_0)^{-1}(-\dive\iota_0)^{-1}f
\]strongly in $L_2(\Omega)$ (recall that compact operators map weakly convergent sequences to strongly convergent sequences). Now, assuming that $D_{\tilde{a}_n}^{-1}$ does not converge to $D_{\tilde{a}}^{-1}$ in operator norm, we find a sequence of unit vectors $(f_n)_n$ in $L_2(\Omega)$ such that 
\[
\|D_{\tilde{a}_n}^{-1}f_n-D_{\tilde{a}}^{-1}f_n\|\geq \varepsilon,
\] for some $\varepsilon>0$. Re-using the index $n$ for a suitable weakly convergent subsequence of $(f_n)_n$ and using that $D_{\tilde{a}}^{-1}$ is a compact operator, we obtain a contradiction, as we have shown in the argument above, that we also have that $D_{\tilde{a}_n}^{-1}f_n\to D_{\tilde{a}}^{-1}f$ strongly as $n\to\infty$.
\end{remark}
\begin{remark}
Note that $\tilde{a}\in M(\Omega)$ if and only if $\gamma \tilde{a} \in M(\Omega)$ for all $\gamma\in \C$, $\gamma\neq 0$. Thus, if
\[
    W(\tilde{a}) \coloneqq \{ \langle \tilde{a}\phi,\phi\rangle, \phi\in L_2(\Omega)^d, \|\phi\|=1\},
\] the numerical range of $\tilde{a}$, is contained in a half space of $\C$ and $0\notin \overline{W(\tilde{a})}$, then we find $\gamma \in \C\setminus\{0\}$ such that $\Re \gamma \tilde{a}\geq c$ for some $c>0$. Next, let $(\tilde{a}_n)_n$ be bounded in $M(\Omega)\subseteq L_\infty(\Omega)^{d\times d}$ with $\bigcup_n W(\tilde{a}_n)$ being contained in a half space and $0\notin \overline{\bigcup_n W(\tilde{a}_n)}$. Then we find $\gamma$ and $c>0$ such that for all $n\in \N$, $\Re \gamma \tilde{a}_n\geq c$. If $(\tilde{a}_n)_n$ $G$-converges to some $\tilde{a}\in M(\Omega)$, then, by the characterisation in Theorem \ref{thm:Gconcocoer}, $(\gamma \tilde{a}_n)_n$ $G$-converges to $\gamma \tilde{a}$. In particular, it follows that $\Re \iota_0^* \gamma \tilde{a}\iota_0\geq \tilde{c}$ for some $\tilde{c}>0$; see also \cite[Proof of Theorem 6.5]{Tartar2009}. 
\end{remark}
The $G$-limit fails to be unique if one keeps the definition as is and considers non-symmetric operators $\tilde{a}_n$. However, if the coefficients are symmetric, that is, belong to the set
\[
M_{\sym}(\Omega) =\{\tilde{a} \in M(\Omega); \tilde{a}(x)=\tilde{a}(x)^*\;\text{a.e.~}x\in\Omega\},
\] the limit is unique:
\begin{corollary}[Uniqueness of $G$-limit]\label{cor:uniqueG} Let $(\tilde{a}_n)_n$ in $M_{\sym}(\Omega)$ $G$-converge to $\tilde{a}$ and $\tilde{b}$ in $M_{\sym}(\Omega)$. Then $\tilde{a}=\tilde{b}$.
\end{corollary}
\begin{proof}
We follow the argument in \cite[p.~201]{EGW17_D2N}; see also \cite{Tartar2009}. By Theorem \ref{thm:Gconcocoer} it follows that $(\iota_0^* \tilde{a}\iota_0)^{-1}=(\iota_0^* \tilde{b}\iota_0)^{-1}$ and, thus, $\iota_0^* \tilde{a}\iota_0 = \iota_0^* \tilde{b}\iota_0$. By linearity, it suffices to show for $\tilde{a}\in M_{\sym}(\Omega)$ that $\iota_0^*\tilde{a}\iota_0 = 0$ implies $\tilde{a}=0$. For this, take $\tau\in C_c^\infty(\Omega)$ and $\xi\in \R^d$. Then define for $\lambda\geq 0$, $u_\lambda (x)\coloneqq \tau(x)\e^{\ii \lambda x\xi}$, it follows that $u_\lambda\in H_0^1(\Omega)$ and, by assumption,
\begin{align*}
 0 = \langle \tilde{a} \grad u_\lambda, \grad u_\lambda \rangle = \langle \tilde{a} (\grad \tau + \ii\lambda \xi \tau ), (\grad \tau + \ii\lambda \xi \tau )\rangle.
\end{align*}
Dividing by $\lambda^2$ and letting $\lambda\to \infty$, we deduce
\[
   0 = \int_{\Omega} |\tau(x)|^2 \langle \tilde{a}(x)\xi,\xi\rangle \dd x
\]
This leads to $\langle \tilde{a}(x)\xi,\xi\rangle =0$ for a.e.~$x\in \Omega$ and, by the separability of $\R^d$ and $\tilde{a}=\tilde{a}^*$, we infer $\tilde{a}=0$.
\end{proof}

\begin{remark}
The uniqueness statement can still be obtained for $H$-convergence replacing $G$-convergence, see \cite{Murat1997}. We shall not discuss this here since for our purposes the above uniqueness statement is sufficient.
\end{remark}

\section{Holomorphic $G$-convergence}\label{sec:holG}

This section is devoted to generalise the concept of $G$-convergence in order to deal with operators that are not necessarily invertible. 

In order to provide the broader context of the results to come, we present an example highlighting the intricacies of the weak operator topology. The example is based on the following well-known result.

\begin{theorem}[{{see, e.g., \cite{Cioranescu1999} or \cite[Theorem 13.2.4 and Remark 13.2.5]{STW_EE21}}}]\label{thm:perweak} Let $\Omega\subseteq \R^d$ open, $f\colon \R^d\to \C$ $1$-periodic (i.e., for all $k\in \Z^d$, $f(\cdot+k)=f$), bounded, measurable. Then 
\[
   f(n\cdot)\to \int_{[0,1)^d} f(x)dx 
\]as $n\to\infty$ in $\sigma(L_\infty(\Omega),L_1(\Omega))$, the weak*-topology of $L_\infty(\Omega)$.
\end{theorem}
The following remark recalls a well-known fact about the  relationship between the weak*-topology of $L_\infty(\Omega)$ and the weak operator topology on $L(L_2(\Omega))$.
\begin{remark} For a bounded, measurable function $f\colon \Omega\to \C$ we denote the associated multiplication operator by
\[
  f(\m)\colon L_2(\Omega)\to L_2(\Omega), \phi\mapsto [x\mapsto f(x)
  \phi(x)].
\]
Then it is easy to see that the mapping
\[
   (L_\infty(\Omega),\sigma(L_\infty(\Omega),L_1(\Omega))) \ni f\mapsto f(\m)\in (L(L_2(\Omega)),\tau_{\textnormal{w}})
\]is a homeomorphism onto its image, see also \cite[Proposition 13.2.1]{STW_EE21}. Indeed, this follows from the fact that any $L_1$-function can be written as a product of two $L_2$-functions and that, by the Cauchy--Schwarz inequality, the product of any two $L_2$-functions is in $L_1$.
\end{remark}
We come to the announced example showing that convergence of inverses in the weak operator does not imply convergence of the respective resolvents. 
\begin{example}\label{exa:holo} (a) Let $f = \frac12 \1_{(0,1/2]}+\frac13\1_{(1/2,1]}$. Periodically extend $f$ to whole of $\R$. Define $f_n\coloneqq f(n\cdot).$  Then for all $\lambda$, sufficiently small, we have
\[
   (f_n(\m)+\lambda)^{-1} \to \int_{0}^1 \frac{1}{f(x)+\lambda}dx = \frac12 (\frac{1}{1/2+\lambda}+\frac{1}{1/3+\lambda}).
\]
Note that for $\lambda=0$, we have 
\[
  f_n(\m)^{-1}\to \frac{5}{2}.
\]
(b) Define 
\[g_n\coloneqq \begin{cases} f_n(\m),& n\text{ odd},\\
\frac25,&n\text{ even}.
\end{cases}\] Then $g_n^{-1}$ converges in the weak operator topology to $\frac52$. However, for $\lambda\neq 0$ we have 
\[
   \frac12 (\frac{1}{1/2+\lambda}+\frac{1}{1/3+\lambda})\neq (\frac25+\lambda)^{-1}.
\]
\end{example}
Even though, we have just seen that the convergence of the resolvents is not necessarily implied by the convergence of the inverses, we shall provide an elementary characterisation next. This statement provides  a link between resolvents and inverses; which itself is a consequence of Cauchy's integral formula.

\begin{proposition}\label{thm:justification} Let $(a_n)_n$ in $L(H)$, $b\in L(H)$. Then the following conditions are equivalent:
\begin{enumerate}
\item[(i)] $a_n \to b$ in the weak operator topology;
\item[(ii)] for all $n\in \N$ there exists $f_n\colon \dom(f_n)\subseteq \C\to L(H)$ holomorphic in a neighbourhood of $0$ with $f_n(0)=a_n$ and
\[
   \int_{\gamma_n} f_n(\lambda)\frac{1}{\lambda} d\lambda \to    \int_{\gamma_\infty}f_\infty(\lambda)\frac{1}{\lambda} d\lambda
\]
in the weak operator topology, where $f_\infty\colon \dom(f_\infty)\subseteq \C\to L(H)$ holomorphic around $0$ with $f_\infty(0)=b$ and $\gamma_n$ is a closed, piecewise continuously differentiable path in $\dom(f_n)$  with $\ind_0 \gamma_n = 1$ for all $n\in \N\cup \{\infty\}$.
\end{enumerate}
\end{proposition}
\begin{proof}
(i)$\Rightarrow$(ii) Choose $f_n(\lambda)\coloneqq a_n$,$f_\infty(\lambda)\coloneqq b$, and any paths $\gamma_n$, $\gamma_\infty$ with the described properties in $\C$.

 (ii)$\Rightarrow$(i)
 By Cauchy's integral formula,  $ \int_{\gamma_n} f_n(\lambda)\frac{1}{\lambda} d\lambda = a_n$ and $\int_{\gamma_\infty}f_\infty(\lambda)\frac{1}{\lambda} d\lambda=b$.
\end{proof}

This observation provides a way of generalising convergence in the weak operator topology. The major application, we have in mind, of the above is $f_n(\lambda)=(a_n-\lambda)^{-1}$ for $a_n$ invertible. In the above theorem, if we had $(a_n-\lambda)^{-1}\to (a-\lambda)^{-1}$ in the weak operator topology for all $\lambda$ in a sufficiently small neighbourhood of $0$, we immediately get that $a_n^{-1}\to a^{-1}$ in the weak operator topology. The above example however shows that the weak operator topology limit of resolvents is not necessarily a resolvent again. Moreover, we are interested in cases for $a_n$ not necessarily invertible. Thus, the functions $f_n$ we are interested in have a singularity at $0$. Even worse, the singularities are not isolated at $0$ (uniformly in $n$). 

For a Hilbert spaces $H_1,H_2$ and an open $\omega\subseteq \C$, we introduce the set
\[
   \mathcal{H}(\omega;L(H_1,H_2))\coloneqq \{ f\colon \omega\to L(H_1,H_2); f \text{ holomorphic}\}
\]of operator-valued holomorphic functions.
Let $(f_n)_n$ in $   \mathcal{H}(\omega;L(H_1,H_2))$. We call $(f_n)_n$ \textbf{locally bounded} or \textbf{normal}, if for all $K\subseteq \omega$ compact, $\sup_{n\in \N} \|f_n\|_{\infty,K}<\infty$, where $\|f\|_{\infty,K}\coloneqq \sup_{\lambda\in K} \|f(\lambda)\|_{L(H_1,H_2)}$ for all $f\in    \mathcal{H}(\omega;L(H_1,H_2))$.

$(f_n)_n$ in $\mathcal{H}(\omega;L(H_1,H_2))$ is said to converge in the \textbf{compact open weak operator topology} to some $f$, $f_n\stackrel{\cowot}{\to}f$ or $f=\cowotl_{n\to\infty}f_n$, if for all $\phi,\psi\in H$
\[
    \langle \phi,f_{n}(\cdot)\psi\rangle \to     \langle \phi,f(\cdot)\psi\rangle
\]
in the compact open topology of uniform convergence on compact sets. If $H_2=\C$, then $L(H_1)=H_1$, by the Riesz representation theorem, and we say \textbf{compact open weak topology} instead of compact open weak operator topology and use $\stackrel{\cow}{\to}$ and $\cowl$ as suitable symbols.

With the help of Tikhonov's theorem in conjunction with Montel's theorem as well as using Riesz representation and Dunford's theorem characterising vector-valued holomorphic functions, one can show Montel's theorem also for the compact open weak operator topology. For general Hilbert spaces, it is only possible to provide the existence of an accumulation point. If $H$ is separable, the full statement of the scalar version of Montel's theorem can be recovered:

\begin{theorem}[{{Montel's theorem, \cite[Theorem 3.4]{W12_HO}, or \cite{W11_P}}}]\label{thm:MontalOV} Let $H_1,H_2$ be a separable Hilbert spaces, $\omega\subseteq \C$ open. $(f_n)_n$ a normal family of holomorphic functions from $\omega$ to $L(H_1,H_2)$. Then there exists a subsequence $(f_{\pi(n)})_n$ and $f\colon \omega\to L(H_1,H_2)$ with $\cowotl_{n\to\infty} f_{\pi(n)} = f$.
\end{theorem}
\begin{proof}
In \cite[Theorem 3.4]{W12_HO} it has been shown that uniformly bounded sequences admit a convergent subsequence in the compact open weak operator topology. Using a standard exhaustion of $\omega$ with open sets being compactly contained in $\omega$, the statement follows upon applying a standard diagonal procedure. 
\end{proof}

The abstract concepts now help to generalise convergence in the weak operator topology of inverses of bounded linear operators. For this let $H$ be a separable Hilbert space and $(a_n)_n$ a sequence in $L(H)$ and assume there exists $\omega\subseteq \bigcap_{n\in\N} \rho(a_n)$ open in $\C$ with $0\in \overline{\omega}$ such that $(a_n-\cdot)^{-1}$ is locally bounded on $\omega$. Let $a\subseteq H\times H$ be a relation. Then we say that $(a_n)_n$ \textbf{holomorphically $G$-converges (on $\omega$)} to $a$, if the set $\mathcal{G}$ of $\phi\in H$ satisfying the following two conditions (a) and (b) is dense in $H$, where
\begin{enumerate}
\item[(a)] there exists $f_\phi \colon \omega \to H$ holomorphic, such that $(a_n-\cdot)^{-1}\phi \to f_\phi$ in the compact open weak topology,
\item[(b)] $f_\phi$ admits a holomorphic extension to $0$;
\end{enumerate} and 
\[
   a^{-1} = \{ (\phi,\psi)\in \mathcal{G}\times H; \psi = f_\phi(0)\}.
\]
%

\begin{remark}\label{rem:holG}
(a) By the identity theorem, for $\phi\in \mathcal{G}$, the holomorphic extension of $f_\phi$ in a neighbourhood of zero is uniquely determined from its values on $\omega$ since $0\in \overline{\omega}$. Thus, the holomorphic $G$-limit is unique. 

(b) The holomorphic $G$-limit may depend on $\omega$. If, however, 
\[\hat{\omega}\coloneqq\{ \lambda \in \bigcap_{n\in \N} \rho(a_n); (a_n-\lambda)^{-1} \text{ bounded}\}
\] consists of only one connected component and $(a_n)_n$ is bounded, then the limit is independent of the choice of $\omega$. This property will \emph{always} be satisfied in the applications considered in the present manuscript.

Let us prove the just mentioned uniqueness statement. Let $(a_n)_n$ holomorphically $G$-converge on $\omega_1$ to $a$ and on $\omega_2$ to $b$. Denote $\mathcal{G}_a\coloneqq \dom(a^{-1})$ and $\mathcal{G}_b\coloneqq \dom(b^{-1})$. By definition, $(a_n-\cdot)^{-1}\phi\eqqcolon f_{n,\phi} \stackrel{\cow}{\to} f_{\phi}^{(a)}$ on $\omega_1$ and $(a_n-\cdot)^{-1}\psi\eqqcolon f_{n,\psi} \stackrel{\cow}{\to} f_{\psi}^{(b)}$ on $\omega_2$ for all $\phi\in \mathcal{G}_a$ and $\psi\in \mathcal{G}_b$. By the Vitali Convergence Theorem (\cite[Theorem 6.2.8]{Simon2015}) (carried out weakly), for all $\phi\in \mathcal{G}_a$, $(a_n-\cdot)^{-1}\phi= f_{n,\phi} \stackrel{\cow}{\to} f_{\phi}$ on the whole of $\hat{\omega}$ for some $f_\phi\colon \hat{\omega}\to H$ holomorphiccally extending $f_\phi^{(a)}$. In particular, $f_{n,\phi} \stackrel{\cow}{\to} f_{\phi}$ on $\omega_2$. By the identity theorem, it follows that $f_{\phi} = f_{\phi}^{(b)}$ and $\phi\in \mathcal{G}_b$ as well as $(\phi,f_{\phi}(0))\in b^{-1}$. Thus, $a^{-1}\subseteq b^{-1}$. By symmetry, $a=b$.

(c) Given the knowledge of $a_n^{-1}\to a^{-1}$ in the weak operator topology does not imply holomorphic $G$-convergence of $(a_n)_n$ to $a$ as Example \ref{exa:holo}(b) confirms. It is however possible to still obtain a suitable equivalence statement as Theorem \ref{thm:hGcontw} below shows.

(d) If $(a_n)_n$ holomorphically $G$-converges to $a$ with $\mathcal{G}_a=H$ and if there exists $\varepsilon>0$ such that  $f_\phi$ is holomorphic on $B(0,\varepsilon)$ for all $\phi \in H$. Then, by Dunford's theorem,  $f_n =(a_n -\cdot)^{-1}$ converges in the compact open weak operator topology on $B(0,\varepsilon)$.

(e) If $\hat{\omega}$ from (b) satisfes $B(0,\varepsilon)\subseteq \hat{\omega}$ for some $\varepsilon>0$ and if $(a_n)_n$ holomorphically $G$-converges to $a$, then as $(a_n-\cdot)^{-1}$ is a normal family on $B(0,\varepsilon)$, the density of $\mathcal{G}_a$ yields $\mathcal{G}_a=H$. Hence, by (d), $(a_n-\cdot)^{-1}$ is convergent in the compact open weak operator topology.
\end{remark}

\begin{lemma}\label{lem:invble} Let $(a_n)_n$ in $L(H)$ invertible with $(a_n^{-1})_n$ bounded. Then there exists $\varepsilon>0$ and $C>0$ such that $B(0,\varepsilon)\subseteq \bigcap_{n\in \N} \rho(a_n)$ and
\[
    f_n\colon B(0,\varepsilon)\ni \lambda \to (a_n-\lambda)^{-1}
\]
defines a normal family of holomorphic functions with $\sup_n\|f_n\|_{\infty, B(0,\varepsilon)}\leq C$.
\end{lemma}
\begin{proof}
Let $C\coloneqq 2\sup_{n}\|a_n^{-1}\|$, $\varepsilon\coloneqq 1/(2C)$. Then for $\lambda\in B(0,\varepsilon)$ we deduce that $(a_n-\lambda)=a_n(1-a_n^{-1}\lambda)$ is invertible using the Neumann series and that
\[
   \|(a_n-\lambda)^{-1}\|=\|(1-a_n^{-1}\lambda)^{-1}a_n^{-1}\|\leq \frac{C}{2}\sum_{k=0}^\infty \|a_n^{-k} \lambda^k\| \leq \frac{C}{2}\frac{1}{1-1/2} = C.\qedhere
\]
\end{proof}

\begin{theorem}\label{thm:hGcontw} Let $(a_n)_n$ in $L(H)$ invertible with $(a_n^{-1})_n$ bounded, $a\in L(H)$ invertible, $H$ separable. 
Consider the following assertions:
\begin{enumerate}
\item[(i)] $(a_n)_n$ holomorphically $G$-converges to $a$;
\item[(ii)] every subsequence $(a_{\pi(n)})_n$ contains a subsequence holomorphically $G$-converging to $a$;
\item[(iii)]  $a_n^{-1}\to a^{-1}$ in the weak operator topology.
\end{enumerate}
Then (i)$\Rightarrow$(iii)$\Leftrightarrow$(ii).
\end{theorem}
\begin{proof} (i)$\Rightarrow$(iii) By Lemma \ref{lem:invble} $((a_n-\cdot)^{-1})_n$ is a normal family on $B(0,\varepsilon)$ for some $\varepsilon>0$ Hence, by Remark \ref{rem:holG} (e) in conjunction with (i), $((a_n-\cdot)^{-1})_n$ is convergent in the compact open weak operator topology to some $g$. Thus, $a_n^{-1}\to a^{-1}$ in the weak operator topology by Proposition \ref{thm:justification} (ii) applied to $f_n = (a_n-\cdot)^{-1}$, $f_\infty=g$ and $\gamma_n=\gamma_\infty$ be the circle of radius $\varepsilon/2$ around $0$.

(iii)$\Rightarrow$(ii) By Lemma \ref{lem:invble}, $f_n=(a_{\pi(n)}-\cdot)^{-1}$ is a normal family on $B(0,\varepsilon)$. By the separability of $H$ and Theorem \ref{thm:MontalOV}, we may choose a subsequence $(f_{\tilde{\pi}(n)})_n$ converging in the compact open weak operator topology to some holomorphic $f\colon B(0,\varepsilon)\to L(H)$. In particular, we have
\[
  f(0) = \tau_{\textnormal{w}}\text{-}\lim_{n\to\infty} f_{\tilde{\pi}(n)}(0)=\tau_{\textnormal{w}}\text{-}\lim_{n\to\infty} a_{\tilde{\pi}(n)}^{-1} = a^{-1},
\]which shows the assertion.

(ii)$\Rightarrow$(iii) This follows from the implication (i)$\Rightarrow$(ii) in conjunction with a subsequence principle.
\end{proof}

\begin{proof}[Proof of Theorem \ref{thm:hGcontw-mr}] 
It suffices to apply Theorem \ref{thm:hGcontw} and to note that, by Theorem \ref{thm:homoGcon}, (iii) is equivalent to $(\iota_0^*\tilde{a}_n\iota_0)^{-1} \to (\iota_0^*\tilde{a}\iota_0)^{-1}$ in the weak operator topology.
\end{proof}

%
%
%

\section{The one-dimensional problem -- homogenisation}\label{sec:hom1d}

We will apply the abstract findings from the previous section to obtain homogenisation results for divergence form problems with possibly rapidly sign-changing coefficients. As for the analysis concerning well-posedness, we start with the regular case. The regular case will then provide us with the necessary tools to understand the irregular case, where the well-posedness criteria are violated. This violation, however, happens to be only in a very controlled way so that holomorphic $G$-convergence can be shown. 

Define $\Omega\coloneqq (0,1)$, as usual let $\iota_0\colon g_0(\Omega)\hookrightarrow L_2(\Omega)$ be the canonical embedding, where $g_0(\Omega)=\partial[H_0^1(\Omega)]$ and recall $\mm(\beta)=\int_0^1\beta$.

\begin{theorem}\label{thm:hom1D} Assume Setting \ref{set:1Dstart}. Let $(\alpha_n)_n$ be a sequence in $L_\infty(\Omega)$ such that $(\alpha_n^{-1})_n$ converges in $\sigma(L_\infty(\Omega),L_1(\Omega))$ to some $\beta \in L_\infty(\Omega)$ with $\mm(\beta)\neq 0$. Then
\[
   (\iota_0^*\alpha_{n}\iota_0)^{-1} \to \big[\psi \mapsto  \beta \psi - \beta \frac{\mm(\beta \psi) }{\mm(\beta)}\big].
\]in the weak operator topology of $L(g_0(\Omega))$. If, in addition, $\alpha_\infty\coloneqq \beta^{-1}\in L_\infty(\Omega)$ then $\alpha_n\stackrel{G}{\to}\alpha_\infty$.
\end{theorem}
\begin{proof}
The form of the solution operator in Theorem \ref{thm:1Dwp} shows that 
\[
    (\iota_0^*\alpha_n\iota_0)^{-1} = \big[\psi \mapsto  \alpha_n \psi - \alpha_n \frac{\mm(\alpha_n \psi) }{\mm_{\Omega}(\alpha_n)}\big].
\]
Given the assumed convergence in $L_\infty$-weak*, we deduce the corresponding limit (note that in particular ${\mm(\alpha_n \psi) }\to {\mm(\beta \psi) }$ as $n\to\infty$ since by the boundedness of $\Omega$, $L_2(\Omega)\hookrightarrow L_1(\Omega)$. The second assertion follows from Theorem \ref{thm:Gconcocoer} (the additional assumptions on $\beta$ guarantee that $\beta^{-1}\in M_{\sym}(\Omega)$).
\end{proof}
\begin{remark} The treatment in Theorem \ref{thm:hom1D} provides a complete description of the homogenisation problem for coefficients $\alpha$  satisfying $\alpha,\alpha^{-1} \in L_\infty(\Omega)$ under the additional requirement that $\mm(\alpha^{-1})\neq 0$.
\end{remark}

Let us turn to the degenerate case. 

\begin{theorem}\label{thm:hGcon1D} Let $\alpha\in L_\infty(\R; \R)$ be $1$-periodic with $\alpha^{-1}\in L_\infty(\R)$ and assume that for $\Omega\coloneqq (0,1)$, $\mm (\alpha^{-1})=0$; define $\alpha_n\coloneqq \alpha(n\cdot)$. Then for all $\lambda>\|\alpha\|_\infty$ or $\lambda \in \C\setminus \R$,
\[
    (\alpha_n -\lambda)^{-1} \to \mm((\alpha -\lambda)^{-1})
\]
in $\sigma(L_\infty(\Omega),L_1(\Omega))$.
Moreover, 
\[
   \iota_0^* \alpha_n\iota_0 \stackrel{\textnormal{hol-}G}{\to}0^{-1}=\{0\}\times g_0(\Omega),
   \] or, defining,  $\alpha_\infty\coloneqq \{0\}\times L_2(\Omega)$, we have 
   \[
       \alpha_n \stackrel{\textnormal{hol-}G}{\to} \alpha_\infty
   \]
   and
   $\sigma_{g_0(\Omega)}(\alpha_\infty)= \emptyset$.
\end{theorem}
\begin{proof}
For the first statement it suffices to realise that $\alpha$ is $1$-periodic and, thus, so is $(\alpha-\lambda)^{-1}$. Hence, the first assertion follows from Theorem \ref{thm:perweak}. Next Theorem \ref{thm:hom1D} implies for  $\lambda>\|\alpha\|_\infty$ or $\lambda \in \C\setminus \R$, that 
\begin{align*}
(  \iota_0^* (\alpha_n-\lambda)\iota_0)^{-1} & \to \big[\psi \mapsto  \mm((\alpha -\lambda)^{-1}) \psi - \mm((\alpha -\lambda)^{-1})\frac{\mm(\mm((\alpha -\lambda)^{-1}) \psi) }{\mm(\mm((\alpha -\lambda)^{-1}))}\big] \\
& = \big[\psi \mapsto  \mm((\alpha -\lambda)^{-1}) \psi - \mm((\alpha -\lambda)^{-1})\mm(\psi) \big].
\end{align*} The right-hand side defines a holomorphic function $f$, which holomorphically extends to $0$
by $0$; which shows that $ \iota_0^* \alpha_n\iota_0 \text{ holomorphically $G$-converges to }0$. \end{proof}

\section{The higher-dimensional problem -- homogenisation}\label{sec:homdd}

In the higher-dimensional homogenisation problem, we focus on periodic coefficients right away. We recall here the well-known theorem on the homogenisation problem for periodic stratified media (or, equivalently, laminated materials) due to \cite{Murat1997}.  We only consider the case of scalar coefficients.  For $\Omega\subseteq \R^d$, we call $\tilde{a}\in L_\infty(\Omega)^{d\times d}$ \textbf{laminated}, if there exists $\alpha \colon \R\to \R$ such that $\tilde{a}(x) = \diag(\alpha(x_1),\ldots,\alpha(x_1))$. 

\begin{theorem}[{{see \cite{Murat1997} or \cite[Theorem 5.12]{Cioranescu1999}}}]\label{thm:stmed} Let $\tilde{a}\in L_\infty(\R^d)^{d\times d}$ $1$-periodic, laminated with $\tilde{a}(x)\geq c$ for a.e.~$x\in \R^d$ and some $c>0$. Define $\tilde{a}_n(x)\coloneqq \tilde{a}(nx)$. Then for all bounded, open $\Omega\subseteq \R^d$, $(\tilde{a}_n)_n$  $G$-converges to 
\[
   \diag(\mm(\alpha^{-1})^{-1}, \mm(\alpha),\ldots,\mm(\alpha)).
\]
\end{theorem}

For the degenerate case, we focus on piecewise constant coefficients, that is, laminated materials with $\alpha$ being piecewise constant as in Setting \ref{set:spectral1D}. Before we provide the homogenisation theorem, we need a technical lemma, the proof of which is postponed to Section \ref{sec:proofs}. For the rest of this section, we let $\Omega\coloneqq (0,1)^d$.
\begin{lemma}\label{lem:computeproj} There exists a unitary operator $U\colon \ell_2(\N_{>0}^d) \to g_0(\Omega)$ such that for all $\gamma>0$ and $\Gamma \coloneqq \diag(\gamma,1,\ldots,1)$ we have
\[
   U^*(\iota_0^* \Gamma \iota_0)^{-1}U = \left[(x_k)_{k\in \N_{>0}^d} \mapsto \Big (\frac{  \sum_{m=1}^d k_m^2}{\gamma k_1^2 +\sum_{m=2}^d k_m^2 } x_k\Big)_{k\in \N_{>0}^d}\right]\in L(\ell_2(\N_{>0}^d))
\]
\end{lemma}

To compute the inner spectrum for operators of the type $\Gamma$ is now an easy task:
\begin{theorem}\label{thm:innerspGamma} Let $\gamma>0$ and $\Gamma \coloneqq \diag(\gamma,1,\ldots,1)$. Then
\[
   \sigma_{g_0(\Omega)}(\Gamma) = \overline{\big\{\tfrac{\gamma k_1^2 +\sum_{m=2}^d k_m^2  }{ \sum_{m=1}^d k_m^2}; k_1,\ldots,k_m\in \N_{>0}\big\}}
   \]
\end{theorem}
\begin{proof}
The statement follows from the unitary equivalence stated in Lemma \ref{lem:computeproj} and standard results on the spectrum of multiplication operators, see, e.g., \cite[Theorem 2.4.7]{STW_EE21}. 
\end{proof}

We present the homogenisation theorem in this situation next. For this, we let $\alpha_0,\ldots,\alpha_n\in \R\setminus\{0\}$ and define $\alpha$ to be the $1$-periodic extension to the whole of $\R^d$ of the function 
\[
(x_1,\ldots,x_d)\mapsto \sum_{j=0}^{r} \alpha_j \1_{[j/{r+1}, (j+1)/(r+1))}(x_1).
\]
\begin{theorem}\label{thm:mainhom} Define $\alpha_n\coloneqq \alpha(n\cdot)$. Then for all $\lambda>\|a\|_{\infty}$ or $\lambda \in \C\setminus\R$, we have
\[
  (\alpha_n(\m_1)-\lambda)\textnormal{ $G$-converges to } \diag(\mm((\alpha-\lambda)^{-1})^{-1}, \mm(\alpha-\lambda),\ldots,\mm(\alpha-\lambda)).
\]
Moreover, the following holomorphic $G$-convergence statements hold:
\begin{enumerate}
\item[(a)] If $\mm(\alpha^{-1})=0$, then
\[
   \iota_0^*\alpha_n(\m_1)\iota_0\stackrel{\textnormal{hol-}G}{\longrightarrow} \mathbf{0}^{-1}\subseteq g_0(\Omega)\times g_0(\Omega),
\]
where $\mathbf{0}$ is a densely defined, proper restriction of the relation $g_0(\Omega)\times \{0\}$.
\item[(b)] If $\mm(\alpha)=0$ and $\mm(\alpha^{-1})\neq 0$, then
\[
   \iota_0^*\alpha_n(\m_1)\iota_0\stackrel{\textnormal{hol-}G}{\longrightarrow} \frac{1}{\mm(\alpha^{-1})}\Delta_{(0,1)}(\Delta_{(0,1)^{(d-1)}})^{-1}\subseteq g_0(\Omega)\times g_0(\Omega),
\]
where $\Delta_{(0,1)}$ denotes the Dirichlet--Laplace operator on $L_2(0,1)$ tensorised with $d-1$ copies of the identity of $L_2(0,1)$; similarly $\Delta_{(0,1)^{d-1}}$ is the identity on $L_2(0,1)$ tensorised with the Dirichlet--Laplacian on $L_2((0,1)^{d-1})$. 
\item[(c)]If $\mm(\alpha)\neq 0$ and $\mm(\alpha^{-1})\neq 0$, then\[
  \alpha_n(\m_1)\stackrel{\textnormal{hol-}G}{\longrightarrow} 
     \diag(\mm((\alpha)^{-1})^{-1}, \mm(\alpha),\ldots,\mm(\alpha))
\]
\end{enumerate}
\end{theorem}
\begin{proof}
The $G$-convergence result is a special case of Theorem \ref{thm:stmed}. 
We have
\begin{multline*}
    \diag(\mm((\alpha-\lambda)^{-1})^{-1}, \mm(\alpha-\lambda),\ldots,\mm(\alpha-\lambda)) \\
= \mm(\alpha-\lambda)\diag(\frac{\mm((\alpha-\lambda)^{-1})^{-1}}{\mm(\alpha-\lambda)}, 1,\ldots,1) \eqqcolon \mm(\alpha-\lambda)\Gamma_{\lambda}
 \end{multline*}
By Lemma \ref{lem:computeproj} we find a unitary operator independent of $\lambda>\|\alpha\|_\infty$ such that
\begin{align*}
 f(\lambda)\coloneqq U^*(  \iota_0^*\mm(\alpha-\lambda)\Gamma_\lambda\iota_0)^{-1}U & =\frac{1}{\mm(\alpha-\lambda)} \left(\frac{  \sum_{m=1}^d k_m^2}{\frac{\mm((\alpha-\lambda)^{-1})^{-1}}{\mm(\alpha-\lambda)} k_1^2 +\sum_{m=2}^d k_m^2 }\right)_{k\in \N^d_{>0}} \\
 &=\left(\frac{   \sum_{m=1}^d k_m^2}{\frac{1}{\mm((\alpha-\lambda)^{-1})} k_1^2 +\mm(\alpha-\lambda)\sum_{m=2}^d k_m^2 }\right)_{k\in \N^d_{>0}} \\
 &=\left(\frac{\mm((\alpha-\lambda)^{-1})  \sum_{m=1}^d k_m^2}{k_1^2 +\mm((\alpha-\lambda)^{-1})\mm(\alpha-\lambda)\sum_{m=2}^d k_m^2 }\right)_{k\in \N^d_{>0}}
\end{align*}where we identified the multiplication operator in the right-hand side with the function it is multiplying with.

The expression on the right-hand side yields the following convergences eventually implying the statements in (a), (b), and (c) of the present theorem. By the identity theorem for holomorphic functions, it suffices to treat the real limits only. The connectedness of the considered $\omega$ in the definition of holomorphic $G$-convergence follows from the fact that both $\C\setminus\R$ and all $\lambda\in \R$, $|\lambda|$ large enough belong to the inner resolvent of $\alpha_n(\m_1)$, see also Theorem \ref{thm:isp-dd}.

If $\mm(\alpha^{-1})=0$, then, as $\lambda\to 0$ for all canonical basis vectors $\phi$
\[
  f_\phi(\lambda) \to 0.
\]
Note that the convergence does not hold for all $\phi$.

If $\mm(\alpha)=0$ and $\mm(\alpha^{-1})\neq 0$, then, as $\lambda\to 0$, for all $\phi$ such that $U\phi\in \dom(\Delta_{(0,1)^{(d-1)}}(-\Delta_{(0,1)})^{-1})$
\[
Uf(\lambda)U^*U\phi\to -\mm(\alpha^{-1})\Delta_{(0,1)^{(d-1)}}(-\Delta_{(0,1)})^{-1}U\phi,
\]
If $\mm(\alpha)\neq 0$ and $\mm(\alpha^{-1})\neq 0$, then, as $\lambda\to 0$,
\[
  Uf(\lambda)U^* \to
     (\iota_0^*(\mm(\alpha)\Gamma_0)\iota_0)^{-1}
\]and with 
\[
 \diag(\mm((\alpha)^{-1})^{-1}, \mm(\alpha),\ldots,\mm(\alpha))=\mm(\alpha)\Gamma_0
\]the assertion follows.
%
\end{proof}

We briefly describe the inner spectra of the respective $G$-limits. The statement in (a) follows from the fact that the coefficient is not defined everywhere; the equations in (b) and (c) follows from the representation of $f(\lambda)$ in the proof of Theorem \ref{thm:mainhom} and the identity theorem.

\begin{corollary}\label{cor:mainhom} In the situation of Theorem \ref{thm:mainhom}, the inner spectra of the limits in the respective cases (a), (b), and (c) are
\begin{enumerate}
  \item[(a)] $\sigma_{g_0(\Omega)} (\textnormal{hol-}G\text{-}\lim_{n\to \infty}\alpha_n(\m_1)) =\C$;
  \item[(b)]$\sigma_{g_0(\Omega)} (\textnormal{hol-}G\text{-}\lim_{n\to \infty}\alpha_n(\m_1)) =\frac{1}{\mm(\alpha^{-1})} \overline{\{\frac{k_1^2}{k_2^2+\ldots+k_d^2}; k_1,\ldots,k_d\in \N_{>0}\}}$;
  \item[(c)] if, in addition, $\mm(\alpha)\mm(\alpha^{-1})>0$,
  \[\sigma_{g_0(\Omega)} (\textnormal{hol-}G\text{-}\lim_{n\to \infty}\alpha_n(\m_1)) = \overline{\{\frac{\mm(\alpha^{-1})  \sum_{m=1}^d k_m^2}{k_1^2 +\mm(\alpha^{-1})\mm(\alpha)\sum_{m=2}^d k_m^2 }; k_1,\ldots,k_m\in \N_{>0}\}}\]
\end{enumerate}
\end{corollary}
We provide examples for the corresponding cases.
\begin{example}(a) $n=1$, $\alpha_j =(-1)^{j}$, $j\in \{0,1\}$, leading to $\mm(\alpha^{-1})=0$; for all $d\geq 2$.

(b) $n=2$, $\alpha_0=1$, $\alpha_1=-2$, $\alpha_2=1$; $\mm(\alpha) = \tfrac{1}{3} (1-2+1)=0$, $\mm(\alpha^{-1})=(1-\tfrac{1}{2}+1)/3 = \tfrac{1}{2}\neq 0$, for all $d\geq 2$.

(c) $n=2$, $\alpha_0=1$, $\alpha_1=-1$, $\alpha_2=1$; $\mm(\alpha)=\mm(\alpha^{-1})=\tfrac13$, $d=2$.
\end{example}

\section{Proofs}\label{sec:proofs}

\subsection{Results in Section \ref{sec:sturm}}

We start with some (standard) auxiliary material helping to prove Theorem \ref{thm:spectral1D}.

\begin{remark}\label{rem:aux1D}
 (a) Let $\Omega_-,\Omega_+\subseteq \R$ disjoint, open intervals such that $\Omega\coloneqq\Omega_-\cup\{\gamma\}\cup\Omega_+$ is connected for some $\gamma\in \R$. Then for $\phi\in H^1(\Omega_-)\cap H^1(\Omega_+)$ the following conditions are equivalent:
 \begin{enumerate}
   \item[(i)] $\phi$ extends to a function in $H^1(\Omega)$;
   \item[(ii)] $\phi$ extends to a continuous function on $\Omega$.
 \end{enumerate}
 The proof can be carried out by a suitable application of integration by parts ((ii)$\Rightarrow$(i)) and the Sobolev embedding theorem ((i)$\Rightarrow$(ii)).
 
 (b) Let $M\in \K^{N\times N}$ be a block operator matrix 
 \[
     M = \begin{pmatrix} A & B \\ C & D\end{pmatrix}
 \]
 with quadratic matrices $A$ and $D$ and matrices of appropriate size $B$ and $C$. If $A$ is invertible, then
 \[
    \det M = \det A \det (D - CA^{-1}B).
 \]
 The proof is not difficult and uses the elementary representation
 \[
    M =     \begin{pmatrix} 1 &0 \\   CA^{-1} & 1 \end{pmatrix}
\begin{pmatrix} A & 0 \\ 0 & D-CA^{-1}B\end{pmatrix}\begin{pmatrix} 1 & A^{-1}B \\0 & 1 \end{pmatrix}.   
 \]
\end{remark}

\begin{proof}[Proof of Theorem \ref{thm:spectral1D}]
(i)$\Leftrightarrow$(ii): Proposition \ref{prop:elementaryFredholm}.

The remainder of the proof is concerned with establishing the equivalence between (ii) and (iii). For this, let $u\in H_0^1(\Omega)$ be a solution of
\[
    D_\alpha u = -(\alpha u')'=-\beta u.
\]
We characterise $u$ to be a solution in terms of 
\[u_k(\cdot) \coloneqq u|_{(kh,(k+1)h)}(kh+\cdot),\quad k\in\{0,\ldots,r-1\}.\]
Using the characterisation of $H^1(\Omega)$-functions in Remark \ref{rem:aux1D} and keeping Setting \ref{set:spectral1D} in mind, we infer that for $u$ to be a solution it is equivalent that \begin{equation}\label{eq:ODEs}
   -\alpha_k u_k''=-\beta_k u_k \quad(k\in \{0,\ldots,r-1\}).
\end{equation}together with the boundary conditions
\begin{align*}
  & u_0(0) = 0, u_n(h)=0,\quad u_k(h)=u_{k+1}(0)\quad(k\in \{0,\ldots,r-1\})\\
  &  \alpha_k u'_k(h)=-\alpha_{k+1}u'_{k+1}(0)\quad(k\in \{0,\ldots,r-1\})
\end{align*}
With the help of the transition matrix $A^{(k)}$ with respect to $(h,\beta_k/\alpha_k)$, we obtain the following equations replacing the ODEs in \eqref{eq:ODEs}
\begin{equation}\label{eq:trm}
   A^{(k)} \begin{pmatrix} u_k(0) \\ u'_k(0)\end{pmatrix}= \begin{pmatrix} u_k(h) \\ u'_k(h)\end{pmatrix}.
\end{equation}
Thus, taking the unknowns $x_{k,j}\coloneqq u_k(jh)$,  $y_{k,j}\coloneqq u'_k(jh)$, $k\in \{0,\ldots,r\}$ and $j\in \{0,1\}$, we obtain the following set of equations:
\begin{equation}\label{eq:Ax0}
   \mathcal{A} \mathbf{x}  = 0,
\end{equation}
where $\mathbf{x} = (x^{(0)},x^{(1)})$ with
\begin{align*}
  x^{(0)} &=  (x_{0,0}, y_{0,0} ,   x_{1,0} , y_{1,0} , \cdots , x_{r-1,0} ,  y_{r-1,0} , x_{r,0} , y_{r,0} )^\top \\
 x^{(1)} & = (x_{0,1}, y_{0,1} , x_{1,1} , y_{1,1} , \cdots , x_{r-1,1},  y_{r-1,1} , x_{r,1} ,  y_{r,1} )^\top
\end{align*} providing shorthands for the left-hand and right-hand boundary values respectively. The matrix
 \[\mathcal{A} = (\mathcal{A}_{jk})_{j,k\in \{0,1\}}
  = \begin{pmatrix}
   \mathcal{A}_{00} & \mathcal{A}_{01} \\ \mathcal{A}_{10} & \mathcal{A}_{11} 
  \end{pmatrix}\] consists of
\begin{align*}
  \mathcal{A}_{00}& = -1_{2(r+1)} \in \K^{2(r+1)\times 2(r+1)} & \mathcal{A}_{01}& = \diag( A^{(0)},\ldots, A^{(r)}) \in \K^{2(r+1)\times 2(r+1)}  
\end{align*}
with $1_{2(r+1)}=\diag(1,\ldots,1)\in \K^{2(r+1)\times 2(r+1)}$. The equation 
\[
\begin{pmatrix} \mathcal{A}_{00} & \mathcal{A}_{01} \end{pmatrix} \begin{pmatrix} x^{(0)} \\ x^{(1)} \end{pmatrix}  = 0\] takes account of the equations in \eqref{eq:trm}. The other equations,  $\begin{pmatrix} \mathcal{A}_{10} & \mathcal{A}_{11} \end{pmatrix} \begin{pmatrix} x^{(0)} \\ x^{(1)} \end{pmatrix}  = 0$, are the boundary conditions. The formulas are
\begin{align*}
  \mathcal{A}_{10}& =  \begin{pmatrix}
   \mathbf{o}_{2r}^\top & | & | \\ \diag( 1, \alpha_0, 1, \alpha_1, \cdots, 1, \alpha_{r-1}) & \mathbf{e}^{2(r+1)}_{2(r+1)} & \mathbf{o}_{2(r+1)} \\ \mathbf{o}_{2r}^\top & | & |  \end{pmatrix} \in \K^{2(r+1)\times 2(r+1)} \\ \mathcal{A}_{11}& =\begin{pmatrix}
  | & | & \mathbf{o}_{2r}^\top\\ \mathbf{e}^1_{2(r+1)} & \mathbf{o}_{2(r+1)}  &\diag(- 1, \alpha_1, -1, \alpha_2, \cdots, -1, \alpha_{r})  \\ | & | & \mathbf{o}_{2r}^\top
  \end{pmatrix}\in \K^{2(r+1)\times 2(r+1)}.
\end{align*}
where the lower indices of $\mathbf{o}$ and $\mathbf{e}$ refer to the dimension of the respective zero vector and canonical basis vector. The upper index of $\mathbf{e}$ signifies the position of the $1$.

Any solution $u$ of $(D_\alpha +\beta)u=0$ yields a solution of \eqref{eq:Ax0} by evaluating $u$ and $u'$ at the boundary. On the other hand, all solutions $\mathbf{x}$ of \eqref{eq:Ax0} yield a solution $u$ of $(D_\alpha +\beta)u=0$. Hence, $D_\alpha+\beta$ is one-to-one if and only if $d\coloneqq \det \mathcal{A}\neq 0$. Thus, the remainder of this proof is devoted to show that $\det \mathcal{A}\neq 0$ if and only if (iii) holds. Therefore, we need to compute $d$. For this we expand this determinant with respect to the last row of $\mathcal{A}$, to obtain
\[
  d = - \det \mathcal{A}'
\]
with $\mathcal{A}'$ given by
\[
\begin{pmatrix}
 \begin{pmatrix} -1_{2r}& 0 \\ 0 & \begin{pmatrix} 0 \\ -1 \end{pmatrix} \end{pmatrix} & \vline & \diag (A^{(0)}, \ldots, A^{(r)}) \\ \hline
  \begin{pmatrix}
   \mathbf{o}_{2r}^\top & | \\ \diag( 1, \alpha_0, 1, \alpha_1, \cdots, 1, \alpha_{r-1})  & \mathbf{o}_{2r+1}
    \end{pmatrix} &\vline &
    \begin{pmatrix}
  | & | & \mathbf{o}_{2r}^\top \\ \mathbf{e}^1_{2r+1} & \mathbf{o}_{2r+1}  &\diag(- 1, \alpha_1, -1, \alpha_2, \cdots, -1, \alpha_{r}) 
  \end{pmatrix}
\end{pmatrix}
\]
Next, we expand with respect to the row involving the $1$ of $\mathbf{e}^1_{2r+1}$ to obtain
\[
   d = \det\mathcal{A}''
\]
with $\mathcal{A}''$ reading as follows
\[
\begin{pmatrix}
 \begin{pmatrix} -1_{2r}& 0 \\ 0 & \begin{pmatrix} 0 \\ -1 \end{pmatrix} \end{pmatrix} & \vline & \begin{pmatrix}  \begin{pmatrix} a_{01}^{(0)} \\ a_{11}^{(0)} \end{pmatrix}& 0 \\ 0 &\diag (A^{(1)}, \ldots, A^{(r)}) \end{pmatrix}
   \\ \hline
  \begin{pmatrix}
 \diag( 1, \alpha_0, 1, \alpha_1, \cdots, 1, \alpha_{r-1}) & \mathbf{o}_{2r} \end{pmatrix} &\vline &
    \begin{pmatrix}
   \mathbf{o}_{2r}  &\diag(- 1, \alpha_1, -1, \alpha_2, \cdots, -1, \alpha_{r}) 
  \end{pmatrix}
\end{pmatrix}
\]
Next, we expand with respect to the last column of the upper left block operator matrix und obtain
\[
   d =\det\mathcal{A}'''
\]with $\mathcal{A}'''$ given by
\[
\begin{pmatrix}
\begin{pmatrix}-1_{2r} \\ \mathbf{o}_{2r}^\top\end{pmatrix} & \vline & \begin{pmatrix}  \begin{pmatrix} a_{01}^{(0)} \\ a_{11}^{(0)} \end{pmatrix}& 0  & \\ 0 &\diag (A^{(1)}, \ldots, A^{(r-1)}) & 0 \\ 0 &  0 & \begin{pmatrix}a_{00}^{(r)} & a_{01}^{(r)} \end{pmatrix} \end{pmatrix}
   \\ \hline
  \begin{pmatrix}
 \diag( 1, \alpha_0, 1, \alpha_1, \cdots, 1, \alpha_{r-1}) \end{pmatrix} &\vline &
    \begin{pmatrix}
   \mathbf{o}_{2r} &\diag(- 1, \alpha_1, -1, \alpha_2, \cdots, -1, \alpha_{r}) 
  \end{pmatrix}
\end{pmatrix}
\]We now apply the determinant formula provided in Remark \ref{rem:aux1D} using the following block structure
\begin{align*} & M = \begin{pmatrix} A & \vline & B \\ \hline C& \vline & D \end{pmatrix} =\\
&\begin{pmatrix}
 -1_{2r}  & \vline & \begin{pmatrix}  \begin{pmatrix} a_{01}^{(0)} \\ a_{11}^{(0)} \end{pmatrix} & 0 & | & |\\ 0 &\diag (A^{(1)}, \ldots, A^{(r-1)}) & \mathbf{o}_{2r} & \mathbf{o}_{2r}\end{pmatrix}
   \\ \hline 
  \begin{pmatrix} \mathbf{o}_{2r}^\top \\
 \diag( 1, \alpha_0, 1, \alpha_1, \cdots, 1, \alpha_{r-1}) \end{pmatrix} &\vline &
    \begin{pmatrix} 0 & \mathbf{o}_{2(r-1)}^\top & \begin{pmatrix}a_{00}^{(r)} & a_{01}^{(r)} \end{pmatrix} \\
   \mathbf{o}_{2r} &\diag(- 1, \alpha_1, -1, \alpha_2, \cdots, -1, \alpha_{r})  
  \end{pmatrix}
\end{pmatrix}
\end{align*}
 Then
\[
   d= \det \mathcal{B} = \det(D-CA^{-1}B)
\]with $\mathcal{B}=D-CA^{-1}B=D+CB$ given by
\begin{align*}
 \mathcal{B} & =\begin{pmatrix} 0 & \mathbf{o}_{2(r-1)}^\top & a_{00}^{(r)} & a_{01}^{(r)} \\
 \mathbf{o}_{2(r-1)} & \diag(-1,\alpha_1,\ldots,-1,\alpha_{r-1}) & & \\
 0 & \mathbf{o}_{2(r-1)}^\top & -1 & 0 \\
  0 & \mathbf{o}_{2(r-1)}^\top & 0 & \alpha_{r} \\ \end{pmatrix} \\ & \quad + 
   \begin{pmatrix}  \mathbf{o}_{2r}^\top \\
 \diag(1,\alpha_0,\ldots,1,\alpha_{r-1}) \end{pmatrix}    \begin{pmatrix}  a_{01}^{(0)} & \mathbf{o}_{2(r-1)}^\top & | & |\\
  a_{11}^{(0)} & \mathbf{o}_{2(r-1)}^\top & \mathbf{o}_{2r} & \mathbf{o}_{2r} \\ \mathbf{o}_{2(r-1)} &
 \diag(A^{(1)},\ldots,A^{(r-1)})& | & | \end{pmatrix} \\
 & =\begin{pmatrix} 0 & \mathbf{o}_{2(r-1)}^\top & a_{00}^{(r)} & a_{01}^{(r)} \\
 \mathbf{o}_{2(r-1)} & \diag(-1,\alpha_1,\ldots,-1,\alpha_{r-1}) & & \\
 0 & \mathbf{o}_{2(r-1)}^\top & -1 & 0 \\
  0 & \mathbf{o}_{2(r-1)}^\top & 0 & \alpha_{r} \\ \end{pmatrix} \\ & \quad + 
  \begin{pmatrix} \mathbf{o}_{2r+1}^\top \\
  a_{01}^{(0)} & \mathbf{o}_{2r}^\top  \\
  \alpha_0 a_{11}^{(0)} & \mathbf{o}_{2r}^\top  \\
  0 & \diag(A_{\alpha_1},\ldots, A_{\alpha_{r-1}}) & 0 & 0
  \end{pmatrix}\\
  & = \begin{pmatrix} 0                                                                              &  \mathbf{o}_{2(r-1)}^\top & & & \begin{pmatrix}a_{00}^{(r)} & a_{01}^{(r)} \end{pmatrix} \\
  	\begin{pmatrix} a_{01}^{(0)} \\ \alpha_0 a_{11}^{(0)} \end{pmatrix}  & \begin{pmatrix} -1 & 0 \\ 0 & \alpha_1 \end{pmatrix} & & &\begin{pmatrix} 0 & 0\\0 & 0\end{pmatrix} \\
				\begin{pmatrix} 0 \\ 0 \end{pmatrix}  & A_{\alpha_1}& \begin{pmatrix} -1 & 0 \\ 0 & \alpha_2 \end{pmatrix}  & & \begin{pmatrix} 0 & 0\\ 0 & 0\end{pmatrix}  \\ \vdots & 0 & \ddots & \ddots && \\ 
				\begin{pmatrix} 0 \\ 0 \end{pmatrix} & 0 & \cdots & A_{\alpha_{r-1}} & \begin{pmatrix} -1 & 0 \\ 0 &\alpha_r\end{pmatrix}
				\end{pmatrix},				
\end{align*}
where $A_{\alpha_j}\coloneqq \begin{pmatrix} 1 & 0 \\ 0 & \alpha_j \end{pmatrix}A^{(j)}$. Elementary row and column operations together with the above Remark  yield that $d=(\det\mathcal{C})d'$ with 
\[
   \mathcal{C} =  \begin{pmatrix} \begin{pmatrix}-1 & \\ & \alpha_1 \end{pmatrix} & & & \\
                                                                  A_{\alpha_1} & \begin{pmatrix}-1 & \\ & \alpha_2 \end{pmatrix} & & \\
                                                                  & \ddots &\ddots & \\
                                                                                                                                     &  & A_{\alpha_{r-1}} &\begin{pmatrix}-1 & \\ & \alpha_r \end{pmatrix} \\
                                                                    \end{pmatrix} 
\] and 
\[
   d' = \begin{pmatrix} o_{2(r-1)}^\top & \begin{pmatrix}a_{00}^{(r)} & a_{01}^{(r)} \end{pmatrix} \end{pmatrix}\mathcal{C}^{-1} \begin{pmatrix}  \begin{pmatrix}a_{01}^{(0)} & \alpha_0 a_{11}^{(0)} \end{pmatrix} & o_{2(r-1)}^\top  \end{pmatrix}^\top.
\]As $\det \mathcal{C} = (-1)^r \alpha_1\cdots\alpha_r$ is easily computed (and is always non-zero by the assumptions on $\alpha$); the computation if $d'$ is more involved. Denoting the lower left corner $2\times 2$ block of $\mathcal{C}^{-1}$ by $\mathcal{D}$, we obtain
\[
   d'= \begin{pmatrix}a_{00}^{(r)} & a_{01}^{(r)} \end{pmatrix}\mathcal{D}\begin{pmatrix}a_{01}^{(0)} \\  \alpha_0 a_{11}^{(0)} \end{pmatrix}.
\]
Thus, the assertion follows from Proposition \ref{prop:triangMat} below.
\end{proof}

Before we state and prove Proposition \ref{prop:triangMat}, we consider an elementary lemma next, which we state without the easy proof.

\begin{lemma}\label{lem:2by2elem} Let $A \in \K^{N\times N}$, $D\in \K^{M\times M}$ be invertible matrices and $C\in \K^{M\times N}$. Then the matrix $
    \begin{pmatrix} A & 0 \\ C & D \end{pmatrix}$
is invertible and
\[
      \begin{pmatrix} A & 0 \\ C & D \end{pmatrix}^{-1} = \begin{pmatrix} A^{-1} & 0 \\ -D^{-1}CA^{-1} & D^{-1} \end{pmatrix}
\]
\end{lemma}

\begin{proposition}\label{prop:triangMat} Let $A_1,\ldots,A_n \in \K^{\ell\times \ell }$ invertible and $B_1,\ldots, B_{r-1}\in \K^{\ell\times \ell }$. Then the matrix
\[\mathcal{C}\coloneqq \begin{pmatrix} A_1 & & & \\
                                                                  B_1 & A_2  & & \\
                                                                  & \ddots &\ddots & \\
                                                                                                                                     &  & B_{r-1} &A_r \\
                                                                    \end{pmatrix} \] is invertible and the lower left $\ell\times\ell$ block of $\mathcal{C}^{-1}$ reads:
                                                                    \[
                                                                    (-1)^r A_r^{-1}B_{r-1}A_{r-1}^{-1}\cdots B_{1}A_1^{-1}.
                                                                    \]
\end{proposition}
\begin{proof} The invertibility of $\mathcal{C}$ follows by induction from Lemma \ref{lem:2by2elem}. It remains to show the formula. 
We prove this claim by induction on $r\in \N$. For $r=1$ there is nothing to show. Assume the result to be true for $r-1$. By induction hypothesis and Lemma \ref{lem:2by2elem}, the last block row of $\mathcal{C}^{-1}$ reads in block operator matrix form
\[
  \begin{pmatrix} D & (-1)^{r-1}A_r^{-1}B_{r-1}A_{r-1}^{-1}\cdots B_{2}A_2^{-1} & \cdots & A_r^{-1} \end{pmatrix}
\]The product of this and the first block operator column of $\mathcal{C}$ needs to equate to $0$, i.e.,
\[
    DA_1 +(-1)^{r-1}A_r^{-1}B_{r-1}A_{r-1}^{-1}\cdots B_{2}A_2^{-1}B_1  = 0.
\]Hence, 
\[
   D= (-1)^{r}A_r^{-1}B_{r-1}A_{r-1}^{-1}\cdots B_{2}A_2^{-1}B_1 A_1^{-1}
\]as required.
\end{proof}

\subsection{Results in Section \ref{sec:DDWP}}

The technique used for the proof of Theorem \ref{thm:ndtensor} is akin to parts of the proof of \cite[Theorem 2.13]{PTWW13}. 
\begin{proof}[Proof of Theorem \ref{thm:ndtensor}]
We denote by $-\Delta_{d-1}$ the Dirichlet--Laplace operator on $L_2(\hat{\Omega})$, which is a strictly positive selfadjoint operator with compact resolvent; it is well-known that the corresponding eigenvalues are discrete and form a discrete sequence, see \cite[Section 6.5.1]{E10}. The operator $ D_{\rdd,\alpha} -\alpha \Delta_{d-1}$ is essentially self-adjoint. Indeed, one computes the adjoint using $(1-\varepsilon\Delta_{d-1})^{-1} \to 1$ as $\varepsilon\to 0$ as follows. Beforehand, note that $(1-\varepsilon\Delta_{d-1})^{-1}$ commutes with $(D_{\rdd,\alpha} -\alpha \Delta_{d-1})$ in the sense that
\[
(1-\varepsilon\Delta_{d-1})^{-1}  (D_{\rdd,\alpha} -\alpha \Delta_{d-1}) \subseteq  (D_{\rdd,\alpha} -\alpha \Delta_{d-1})(1-\varepsilon\Delta_{d-1})^{-1};
\]this follows from the fact that both $\alpha$ and $D_{\rdd,\alpha}$ are independent of the $(x_2,\ldots,x_{d})$-variables and that $(1-\varepsilon\Delta_{d-1})^{-1} $ is a (multiple of a) resolvent of $\Delta_{d-1}$ and, thus, commutes with $\Delta_{d-1}$.

Next, let $\phi\in \dom((D_{\rdd,\alpha} -\alpha \Delta_{d-1})^*)$; for $\varepsilon>0$ define $\phi_\varepsilon \coloneqq (1-\varepsilon\Delta_{d-1})^{-1}\phi$.  Then for all $u\in \dom(-D_{\rdd,\alpha} -\alpha \Delta_{d-1})$, we have
\begin{align*}
  \langle \phi_\varepsilon, (D_{\rdd,\alpha} -\alpha \Delta_{d-1})u\rangle & = \langle (1-\varepsilon\Delta_{d-1})^{-1}\phi, (D_{\rdd,\alpha} -\alpha \Delta_{d-1})u \rangle \\
  & = \langle \phi, (1-\varepsilon\Delta_{d-1})^{-1}(D_{\rdd,\alpha} -\alpha \Delta_{d-1})u \rangle \\
  & =\langle \phi, (D_{\rdd,\alpha} -\alpha \Delta_{d-1})(1-\varepsilon\Delta_{d-1})^{-1}u \rangle \\
  & =\langle (D_{\rdd,\alpha} -\alpha \Delta_{d-1})^*\phi, (1-\varepsilon\Delta_{d-1})^{-1}u \rangle \\
   & =\langle (1-\varepsilon\Delta_{d-1})^{-1}(D_{\rdd,\alpha} -\alpha \Delta_{d-1})^*\phi, u \rangle.
\end{align*}
This shows that $\phi_\varepsilon \in \dom((D_{\rdd,\alpha} -\alpha \Delta_{d-1})^*)$ and
\[
  (D_{\rdd,\alpha} -\alpha \Delta_{d-1})^*\phi_\varepsilon = (1-\varepsilon\Delta_{d-1})^{-1}(D_{\rdd,\alpha} -\alpha \Delta_{d-1})^*\phi.
\]
The strong operator topology convergence of $(1-\varepsilon\Delta_{d-1})^{-1} \to 1$ implies that $\phi_\varepsilon \to \phi$ as $\varepsilon\to 0$ in the graph norm of $(D_{\rdd,\alpha} -\alpha \Delta_{d-1})^*$. Next, from the above computation, we deduce using $\alpha \Delta_{d-1}= \Delta_{d-1} \alpha$
\begin{align*}
  \langle \phi_\varepsilon, D_{\rdd,\alpha}u\rangle & = \langle \phi_\varepsilon, \alpha \Delta_{d-1} u\rangle+ \langle (1-\varepsilon\Delta_{d-1})^{-1}(D_{\rdd,\alpha} -\alpha \Delta_{d-1})^*\phi, u \rangle \\
  & = \langle \alpha\Delta_{d-1} \phi_\varepsilon,   u\rangle+ \langle (1-\varepsilon\Delta_{d-1})^{-1}(D_{\rdd,\alpha} -\alpha \Delta_{d-1})^*\phi, u \rangle.
\end{align*}
Hence, since $\dom((D_{\rdd,\alpha} -\alpha \Delta_{d-1}))$ is dense in the graph space of $D_{\rdd,\alpha}$ (use again approximation and commutativity properties of $(1-\varepsilon\Delta_{d-1})^{-1}$), we get
\[
  \phi_\varepsilon \in \dom((D_{\rdd,\alpha})^*)= \dom(D_{\rdd,\alpha})
\]and
\[
    D_{\rdd,\alpha}\phi_\varepsilon =\alpha\Delta_{d-1} \phi_\varepsilon+(1-\varepsilon\Delta_{d-1})^{-1}(D_{\rdd,\alpha} -\alpha \Delta_{d-1})^*\phi.
\]Thus,
\[
 (D_{\rdd,\alpha} -\alpha \Delta_{d-1})^*\phi_\varepsilon = (1-\varepsilon\Delta_{d-1})^{-1}(D_{\rdd,\alpha} -\alpha \Delta_{d-1})^*\phi =D_{\rdd,\alpha}\phi_\varepsilon -\alpha\Delta_{d-1}\phi_\varepsilon,
\]which implies that $D_{\rdd,\alpha} -\alpha \Delta_{d-1}$ is essentially self-adjoint.
\noindent
Next, it is elementary that 
\[
    D_{\rdd,\alpha} -\alpha \Delta_{d-1} \subseteq -\Delta_\alpha.
\] Hence, 
\begin{align*}
   & D_{\rdd,\alpha} -\alpha \Delta_{d-1} \subseteq \overline{    D_{\rdd,\alpha} -\alpha \Delta_{d-1}}\subseteq  -\Delta_\alpha \\
   & -\Delta_\alpha^* \subseteq (D_{\rdd,\alpha} -\alpha \Delta_{d-1})^* =\overline{    D_{\rdd,\alpha} -\alpha \Delta_{d-1}}.
\end{align*} 
Consequently, 
\[
\overline{    D_{\rdd,\alpha} -\alpha \Delta_{d-1}} = -\Delta_\alpha,
\] and the statement follows.
\end{proof}

\subsection{Results in Section \ref{sec:hdp-wp}}

\begin{proof}[Proof of Theorem \ref{thm:sp1Dff}]
(i)$\Leftrightarrow$(ii) This is a standard fact for self-adjoint operators. For convenience of the reader, we present an argument invoking von Neumann's spectral theorem. By Theorem \ref{thm:TW-Mana} using that $\alpha$ is real, $D_\alpha+\alpha\lambda_ k$ is a self-adjoint operator. Moreover, $D_\alpha$ has compact resolvent. Hence, by von Neumann's the spectral theorem, there is a sequence $V\colon \N \to \R$ and a unitary operator $U\colon L_2(0,1)\to \ell_2(\N)$ such that $ D_\alpha+\alpha\lambda_ k = U^*M_V U$, where $M_V$ is the operator of multiplying by $V$ on $\ell_2(\N)$ with maximal domain. Thus, $D_\alpha+\alpha\lambda_ k$ is invertible if and only if $M_V$ is. Moreover, in this case, $\|M_V^{-1}\|\leq C$ is equivalent to $ \|\frac{1}{|V|}\|_{\infty} \leq C$. Thus, for all $j\in \N$, $\frac{1}{|V(j)|}\leq C$, thus, for all $\delta \in  (-1/C,1/C)$
 \[
   |V(j)-\delta|\geq |V(j)|-|\delta|\geq \frac{1}{|C|}-|\delta|>0,
 \]implying that $(-1/C,1/C)\subseteq \rho(M_V)=\rho(D_\alpha+\alpha\lambda_ k )$. 
If, now, $(-1/C,1/C)\subseteq \rho(D_\alpha+\alpha\lambda_ k )$, it follows that $\emptyset=(-1/C,1/C)\cap \sigma(D_\alpha+\alpha\lambda_ k )=\{V(j);j\in \N\}\cap (-1/C,1/C)$. Thus, $\|(D_\alpha+\alpha\lambda_ k )^{-1}\|=\|1/V\|_\infty\leq C$.

(ii)$\Leftrightarrow$(iii) (ii) holds if, and only if, for all $\delta \in [-1,1]$, the operator 
 \[
 D_\alpha+\alpha\lambda_k+\tfrac{1}{C}\delta =  D_\alpha+\alpha(\lambda_k+\tfrac{1}{\alpha C}\delta)
 \] is continuously invertible. Thus, the equivalence asserted  follows directly from Theorem \ref{thm:spectral1D}.
\end{proof}

\begin{proof}[Proof of Proposition \ref{prop:pam}](a) The formula is a straightforward consequence of the formula for $p_{\alpha,\beta}$ taking into account that
\[A^{(j)}=\begin{pmatrix} \cosh(m_j h) & \frac{1}{m_j}\sinh(m_j h) \\ m_j \sinh(m_j h) & \cosh(m_j h) \end{pmatrix}.
\]for all $j\in \{0,\ldots,r\}$ and factoring out $\cosh(m_jh)$. 

(b) We treat the case $x=0$ first. By induction on $r$, one obtains for $\mu\geq\mu_0$
\[
 \begin{pmatrix}1 &  \alpha_{r}^{-1} \mu^{-1}  \end{pmatrix} \begin{pmatrix} 1 & \alpha_{r-1}^{-1}\mu^{-1} \\ \alpha_{r-1}\mu  & 1\end{pmatrix}\cdots  \begin{pmatrix} 1 &\alpha_{1}^{-1}\mu^{-1}  \\  \alpha_{1}\mu & 1\end{pmatrix} = \big(\prod_{j=1}^{r-1}(1+\alpha_{j+1}^{-1}\alpha_{j})\big)\begin{pmatrix}1 &  \alpha_{1}^{-1} \mu^{-1}  \end{pmatrix}.
\]Indeed, the base case $r=1$ being trivial, the inductive step is proved by observing
\[
 \begin{pmatrix}1 &  \alpha_{r}^{-1} \mu^{-1}  \end{pmatrix} \begin{pmatrix} 1 & \alpha_{r-1}^{-1}\mu^{-1} \\ \alpha_{r-1}\mu  & 1\end{pmatrix} = (1+\alpha_r^{-1}\alpha_{r-1})\begin{pmatrix}1 &  \alpha_{r-1}^{-1} \mu^{-1}  
\end{pmatrix}.
\]
Thus, we deduce
\begin{align*}
  \mu w(\mu,0) &  = \mu \begin{pmatrix}1 &  \alpha_{r}^{-1} \mu^{-1}  \end{pmatrix} \begin{pmatrix} 1 & \alpha_{r-1}^{-1}\mu^{-1} \\ \alpha_{r-1}\mu  & 1\end{pmatrix}\cdots  \begin{pmatrix} 1 &\alpha_{1}^{-1}\mu^{-1}  \\  \alpha_{1}\mu & 1\end{pmatrix}\begin{pmatrix}\mu^{-1}  \\  \alpha_0  \end{pmatrix} \\ & =\mu \big(\prod_{j=1}^{r-1}(1+\alpha_{j+1}^{-1}\alpha_{j})\big)\begin{pmatrix}1 &  \mu^{-1}\alpha_{1}^{-1}  \end{pmatrix}\begin{pmatrix}\mu^{-1}  \\  \alpha_0  \end{pmatrix}
   \\ & = \prod_{j=0}^{r-1}(1+\alpha_{j+1}^{-1}\alpha_{j}).
\end{align*}
Next, for $x\in [-\mu_0/2,\mu_0/2]^{r+1}$, $\mu w(\mu,x)$ is a rational function in $\mu$ with leading order coefficients of both the enumerator and denominator polynomials coinciding with the respective ones from $\mu w(\mu,0)$. Since $x$ comes from a compact, thus, bounded subset of $\R^{r+1}$, the convergence is uniform in $x$.

(c) For the claim, it suffices to show the assertion concerning the limit; $q(\mu,\cdot)\in \mathcal{O}(\tfrac{1}{\mu})$ is then a consequence of (b). 
For all $s>0$
\[
  \tanh(s)-1 = \frac{1}{e^{s}\cosh(s)}.
\]
Thus, for all $j\in \{0,\ldots,r\}$
\[
   t_j - 1 = \frac{1}{e^{m_j h}\cosh(m_j h)}.
\]
Let $(\mu,x)\in M_{\mu_0}$ and $(m_0,\ldots,m_r)\coloneqq S(\mu,x)$. For $j\in \{0,\ldots,r\}$ let $s_j \in [0,1]$ and consider
\begin{align*}
  f_r(s_r) & \coloneqq \begin{pmatrix} 1 & \alpha_r^{-1}m_r^{-1} s_r\end{pmatrix} \\
  f_0 (s_0) &  \coloneqq \begin{pmatrix} m_0^{-1} s_0 \\ \alpha_0 \end{pmatrix} \\
    f_k (s_k) &  \coloneqq \begin{pmatrix} 1 & \alpha_k^{-1} m_k^{-1} s_k \\ \alpha_k m_k s_k & 1 \end{pmatrix},
\end{align*}where $k\in \{1,\ldots,r-1\}$.
Then,
\begin{align*}
    q(\mu,x)-w(\mu,x) & = f_r(t_r)f_{r-1}(t_{r-1}) \cdots f_0(t_0) - f_r(1)f_{r-1}(1) \cdots f_0(1) \\
    & = \sum_{j=0}^r f_r(t_r)f_{r-1}(t_{r-1}) \cdots f_{j+1}(t_{j+1}) (f_j (t_j)-f_j(1)) f_{j-1}(1)\cdots f_0(1).
\end{align*}
Since for all $j\in \{0,\ldots,r\}$
\begin{align*}
  \|f_j (t_j)-f_j(1)\| & \leq c_j \mu |t_j-1| \\
  \|f_j (t_j)\| & \leq c_j \mu
\end{align*}
for some $c_j>0$ (depending on $\mu_0$, $\alpha$) independently of $\mu$ and $x$, we obtain with $c\coloneqq \max_j c_j$
\[
  | \mu (q(\mu,x) - w(\mu,x)) |\leq (r+1) c^{r+1} \mu^{n+2} \max_j |t_j-1|.
\]
Since, however, for all $j\in \{0,\ldots,r\}$,  $|t_j-1|\to 0$ exponentially fast as $\mu\to\infty$ by our preliminary observation, we get the desired assertion.

(d) Since $C_c(\R^{r+1})\cap C(M_{\mu_0})$ consists of uniformly continuous functions, we obtain that $\overline{C_c(\R^{r+1})\cap C(M_{\mu_0})}^{\|\cdot\|_\infty}=C_0(\R^{r+1})\cap C(M_{\mu_0})=C_0(M_{\mu_0})$ is uniformly continuous. By (b) and (c), we obtain uniform continuity of
\[
    M_{\mu_0} \ni (\mu,x) \mapsto \mu q(\mu,x)-\prod_{j=0}^{r-1}(1+\alpha_{j+1}^{-1}\alpha_{j}) \in C_0 (M_{\mu_0}),
\]which implies the assertion.
\end{proof}

The proof of  Theorem \ref{thm:charcrit1} requires a preparation. 
\begin{lemma}\label{lem:homdeg0} (a) Let $t>0$, $\alpha_0,\ldots,\alpha_r\in \R\setminus\{0\}$. Then for all $\mu>0$,
\[
  p_{\mu \alpha}(t) = p_{\alpha}(t).
\]
(b) Let $\mu>0$ and $t=\tanh(\mu h)$. Then 
the following conditions are equivalent
\begin{enumerate}
\item[(i)] $\tilde{q}(\mu,\ldots,\mu)\neq 0$;
\item[(ii)] $p_{\alpha}(t)\neq 0$.
\end{enumerate}
\end{lemma}
\begin{proof}
(a) For all $\mu>0$, we compute\[
  \begin{pmatrix}
   \ast & \mu^{-1}\ast \\
   \mu \ast & \ast
  \end{pmatrix}    \begin{pmatrix}
   \ast & \mu^{-1}\ast \\
   \mu \ast & \ast
  \end{pmatrix} = \begin{pmatrix}
   \ast & \mu^{-1}\ast \\
   \mu \ast & \ast
  \end{pmatrix},
\]where the $\ast$-symbols symbolise possibly different arbitrary numbers independent of $\mu$.

Then, using the preliminary observation, we deduce by induction on $r$, that
\begin{align*}
 p_{\mu \alpha}( t) & = \begin{pmatrix}1 &  \mu^{-1}\alpha_{r}^{-1}  t  \end{pmatrix} \begin{pmatrix} \ast & \mu^{-1}\ast \\ \mu \ast  & \ast\end{pmatrix}\begin{pmatrix}t  \\  \mu\alpha_0  \end{pmatrix} \\
 & = \begin{pmatrix} \ast & \mu^{-1} \ast\end{pmatrix}\begin{pmatrix}t  \\  \mu\alpha_0  \end{pmatrix} = \ast.
\end{align*}

(b) We compute, using (a),
\[
   \mu \tilde{q}(\mu,\ldots,\mu) = p_{\mu \alpha}(t) = p_{\alpha}(t),
\]which yields the assertion.
\end{proof}

\begin{proof}[Proof of Theorem \ref{thm:charcrit1}]
 (i)$\Rightarrow$(ii) Put $d=0$ in \eqref{eq:qtildecrit}, the equation defining the $\tilde{q}$-criterion.
 
 (ii)$\Leftrightarrow$(iii), (iv)$\Leftrightarrow$(v) This is Lemma \ref{lem:homdeg0} (b).
 
 (iv)$\Rightarrow$(i) Since $\chi(\alpha)\neq 0$, by Theorem \ref{thm:asy}, we find $0<\delta_1\leq 1$ and $k_0\in \N$ such that for all $k\geq k_0$ and $x \in [-\delta_1,\delta_1]^{r+1}$, we have $q(\mu_k,x)\neq 0$. Consider the Taylor series $\sum_{\ell=0}^\infty c_\ell \xi^\ell= \sqrt{1+\xi}$ for $|\xi|\leq 1$. It is known that the series converges absolutely for $|\xi| \leq c$ for all $c<1$. In particular,
 \[
    \kappa \coloneqq \sup_{\xi\in [-1/2,1/2]} \sum_{\ell=1}^\infty c_{\ell}\xi^{\ell-1} <\infty. 
 \]  We find $k_1\geq k_0$ such that for all $k\geq k_1$ we have $1/\lambda_k\leq \min\{1/2,\frac{\kappa}{\delta_1}\}$. Thus, for $k\geq k_1$ and $-1\leq s\leq 1$ we compute
 \begin{align*}
    |\sqrt{\lambda_k + s}-\mu_k| & = |\sqrt{\lambda_k}\sqrt{1 + \frac{s}{{\lambda_k}}} -\sqrt{\lambda_k}| 
     = \sqrt{\lambda_k}\big|\sum_{\ell=1}^\infty c_\ell \left(\frac{s}{{\lambda_k}}\right)^\ell\big| 
    \leq  \frac{1}{\sqrt{\lambda_k}}\kappa \leq \delta_1.
 \end{align*}
 Since $q(\mu_k,0)\neq 0$ for all $k\in \{0,\ldots,k_1\}$, by continuity of $q$ and $\sqrt{\cdot}$, we find $0<\delta_0\leq \delta_1$ so that $\alpha_0,\ldots,\alpha_r$ and $(\lambda_k)_k$ satisfy the $\tilde{q}$-criterion.
 \end{proof}
 
 \subsection{Results in Section \ref{sec:homdd}}
 
 \begin{proof}[Proof of Lemma \ref{lem:computeproj}]
 Let $q=\grad u, r=\grad v \in g_0(\Omega)=\grad[H_0^1((0,1)^d)]$. We compute
 \begin{align*}
      q = (\iota_0^* \Gamma \iota_0)^{-1} r & \iff      \iota_0^* \Gamma \iota_0 \iota_0^*q =  \iota_0^*r\\
      & \iff \iota_0^* \Gamma \iota_0 \iota_0^*\grad u = \iota_0^* \grad v \\
      & \iff  \dive\iota_0\iota_0^* \Gamma \iota_0 \iota_0^*\grad u = \dive\iota_0\iota_0^* \grad v \\
      & \iff  \dive \Gamma \grad u = \dive \grad v \\
     & \iff   u = (\dive \Gamma \grad)^{-1}\dive \grad v \\
          & \iff  \iota_0^*q=\iota_0^*\grad u =\iota_0^*\grad (\dive \Gamma \grad)^{-1}\dive\iota_0 \iota_0^*r.
 \end{align*}
 Thus, $(\iota_0^* \Gamma \iota_0)^{-1}=\iota_0^*\grad (\dive \Gamma \grad)^{-1}\dive\iota_0$. Next, we derive a more explicit description of this expression. For this, we use the polar decomposition for unbounded operators, see, e.g., \cite[Proposition B.8.6]{PMTW20} for 
 \[\grad\colon H_0^1(\Omega)\subseteq L_2(\Omega)\to L_2(\Omega)^d.\] Thus, we find a unitary operator $U\colon \ran(|\grad|)\to \ran(\grad)$ such that 
 \[
     \iota_0^*\grad = U|\grad|;
 \]as a consequence $\dive\iota_0=-(\iota_0^*\grad)^*=|\grad| U^*$. Hence, we need to evaluate
 \[
   |\grad| (\dive \Gamma \grad)^{-1}|\grad|.
 \]As the latter operator is bounded, it suffices to compute its action on the dense subset $\dom(|\grad|^2)$.
 In order to derive an expression for this operator, we employ the spectral theorem for the Dirichlet--Laplace, $-|\grad|^2=\Delta$ on $L_2((0,1)^d)$ and the perturbed Dirichlet--Laplace $\Delta^{(\gamma)}\coloneqq \dive \Gamma \grad$ on $L_2((0,1)^d)$. Since $\Gamma$ is a diagonal matrix, it follows that the orthonormal basis of eigenfunctions 
 \[
     (e_k)_{k\in \N^d} = ((x_1,\ldots,x_d)\mapsto \frac{1}{2^{d/2}}\sin(k_1\pi x_1)\cdots \sin(k_d\pi x_d))_{k\in \N^d}
 \]
 for $-\Delta$ is also one for $-\Delta^{(\gamma)}$. The difference, however, being the eigenvalues, which in the former case are $\lambda_k = \pi^2\sum_{m=1}^d k_m^2$ and in the latter $\lambda_k^{(\gamma)} =\pi^2(\gamma k_1^2 +\sum_{m=2}^d k_m^2 )$ for all $k\in \N^d_{>0}$ (in order for $\lambda_k^{(\gamma)}>0$, one requires $\gamma>0$.) Thus, 
 \[
    |\grad| (\dive \Gamma \grad)^{-1}|\grad|
 \]
 is unitarily equivalent to the diagonal operator on $ L(\ell_2(\N^d_{>0}))$ of multiplying entry-wise with
 \[
    \N_{>0}^d \ni k\mapsto     \frac{  \sqrt{\sum_{m=1}^d k_m^2} \sqrt{\sum_{m=1}^d k_m^2}}{\gamma k_1^2 +\sum_{m=2}^d k_m^2 }= \frac{  \sum_{m=1}^d k_m^2}{\gamma k_1^2 +\sum_{m=2}^d k_m^2 }\qedhere
 \]
\end{proof}

\section{Conclusion}\label{sec:con}

We presented a new method for homogenisation. This method is predominantly tailored for self-adjoint and sign-changing coefficients with $0$ contained in the (essential) spectrum of the corresponding divergence form operator. The application of this new method required a detailed analysis for when divergence form as well as certain Sturm--Liouville operators with sign-changing coefficients lead to continuously invertible operators. The resulting condition on zeros of certain polynomials was used to identify the inner spectrum of the considered conductivities. For classical homogenisation problems, it was shown that these inner spectra behave rather tame comparing the pre-asymptotics to the homogenised coefficients. In the non-standard setting, which requires the new homogenisation method, the holomorphic $G$-limits are entirely unexpected as the homogenisation of a standard divergence form equation with local but sign-changing coefficients can be shown to lead in certain cases to a limit equation which is both nonlocal and -- in contrast to the second order equation one started out with -- of 4th order. This is particularly marked by a discontinuous behaviour of the inner spectra.


\bibliographystyle{abbrv}

\noindent
Marcus Waurick \\ Institute for Applied Analysis\\ Faculty of Mathematics and Computer Science\\ TU Bergakademie Freiberg\\
Freiberg, Germany\\
Email:
{\tt marcus.wau\rlap{\textcolor{white}{hugo@egon}}rick@math.\rlap{\textcolor{white}{balder}}tu-freiberg.de}

\end{document}